\documentclass[11pt,oneside]{amsart}

\usepackage{amsmath}

\usepackage{enumitem}
\setlist[enumerate]{font=\normalfont}

\usepackage{tikz-cd}

\usepackage{fix-cm}

\DeclareMathAlphabet{\mathcal}{OMS}{cmsy}{m}{n}

\usepackage[arrow,curve,matrix,tips,2cell]{xy}

\usepackage[hmargin=2cm,vmargin=1.5cm]{geometry}

\usepackage{comment}

\usepackage{anyfontsize}
\usepackage{mathrsfs}

\usepackage{bm}

\usepackage{bbm}

\usepackage{fancyhdr}
\usepackage[obeyspaces]{url}

\usepackage{tensor}

\usepackage{mathtools}

\usepackage{accents}

\usepackage{amscd}
\usepackage{amssymb}
\usepackage{latexsym}
\usepackage{url}

\usepackage{graphicx}

\usepackage[draft=false]{hyperref}
\usepackage[noabbrev,capitalise,nameinlink]{cleveref}
\usepackage{bookmark}
\hypersetup{                                                         
  pdfauthor={},
  pdftitle={},
  colorlinks=true, 
  linkcolor={},
    citecolor={},
    urlcolor={},
    linktocpage=true                                                      
}

\newtheorem{theorem}{Theorem}[section]
\newtheorem*{theorem*}{Theorem}
\newtheorem{lemma}[theorem]{Lemma}
\newtheorem*{lemma*}{Lemma}
\newtheorem{corollary}[theorem]{Corollary}
\newtheorem{proposition}[theorem]{Proposition}

\newtheorem{remark}[theorem]{Remark}
\newtheorem{definition}[theorem]{Definition}
\newtheorem*{definition*}{Definition}

\newtheorem{question}[theorem]{Question}
\newtheorem*{question*}{Question}

\newtheorem{example}[theorem]{Example}
\newtheorem{examples}[theorem]{Examples}

\newtheorem{observation}[theorem]{Observation}
\newtheorem{observations}[theorem]{Observations}

\makeatletter
\def\revddots{\mathinner{\mkern1mu\raise\p@
\vbox{\kern7\p@\hbox{.}}\mkern2mu
\raise4\p@\hbox{.}\mkern2mu\raise7\p@\hbox{.}\mkern1mu}}
\makeatother 
\newcommand{\bgl}{\begin{equation}} 
\newcommand{\egl}{\end{equation}}
\newcommand{\bgloz}{\begin{equation*}} 
\newcommand{\egloz}{\end{equation*}}
\newcommand{\bgln}{\begin{eqnarray}} 
\newcommand{\egln}{\end{eqnarray}}
\newcommand{\bglnoz}{\begin{eqnarray*}} 
\newcommand{\eglnoz}{\end{eqnarray*}}
\newcommand{\btheo}{\begin{theorem}}
\newcommand{\etheo}{\end{theorem}}
\newcommand{\btheooz}{\begin{theorem*}}
\newcommand{\etheooz}{\end{theorem*}}

\newcommand{\blemma}{\begin{lemma}}
\newcommand{\elemma}{\end{lemma}}
\newcommand{\blemmaoz}{\begin{lemma*}}
\newcommand{\elemmaoz}{\end{lemma*}}
\newcommand{\bproof}{\begin{proof}}
\newcommand{\eproof}{\end{proof}}
\newcommand{\bbew}{\begin{beweis}}
\newcommand{\ebew}{\end{beweis}}
\newcommand{\bremark}{\begin{remark}}
\newcommand{\eremark}{\end{remark}}
\newcommand{\bdefin}{\begin{definition}}
\newcommand{\edefin}{\end{definition}}
\newcommand{\bdefinoz}{\begin{definition*}}
\newcommand{\edefinoz}{\end{definition*}}
\newcommand{\bex}{\begin{example}}
\newcommand{\eex}{\end{example}}
\newcommand{\bexs}{\begin{examples}\em}
\newcommand{\eexs}{\end{examples}}

\newcommand{\bobs}{\begin{observation}\em}
\newcommand{\eobs}{\end{observation}}
\newcommand{\bobss}{\begin{observations}\em}
\newcommand{\eobss}{\end{observations}}

\newcommand{\bprop}{\begin{proposition}}
\newcommand{\eprop}{\end{proposition}}
\newcommand{\bcor}{\begin{corollary}}
\newcommand{\ecor}{\end{corollary}}
\newcommand{\bfa}{\begin{cases}} 
\newcommand{\efa}{\end{cases}}

\newcommand{\bquestion}{\begin{question}}
\newcommand{\equestion}{\end{question}}
\newcommand{\bquestionoz}{\begin{question*}}
\newcommand{\equestionoz}{\end{question*}}
\newtheorem{introtheorem}{Theorem}

\newtheorem{introcor}[introtheorem]{Corollary}

\newcommand{\cB}{\mathcal B}
\newcommand{\cC}{\mathcal C}
\newcommand{\cD}{\mathcal D}
\newcommand{\cE}{\mathcal E}
\newcommand{\cF}{\mathcal F}

\newcommand{\cI}{\mathcal I}

\newcommand{\cL}{\mathcal L}
\newcommand{\cM}{\mathcal M}

\newcommand{\cR}{\mathcal R}

\newcommand{\cT}{\mathcal T}
\newcommand{\cU}{\mathcal U}
\newcommand{\cV}{\mathcal V}
\newcommand{\cW}{\mathcal W}
\newcommand{\cX}{\mathcal X}
\newcommand{\cY}{\mathcal Y}
\newcommand{\cZ}{\mathcal Z}

\def\Cz{\mathbb{C}}

\def\Nz{\mathbb{N}}

\def\Vz{\mathbb{V}}

\def\Zz{\mathbb{Z}}
\newcommand{\fX}{\mathfrak X}

\newcommand{\mfa}{\mathfrak a}

\newcommand{\mfr}{\mathfrak r}
\newcommand{\mfs}{\mathfrak s}

\newcommand{\mfz}{\mathfrak z}
\newcommand{\an}[1]{``#1''} 
\newcommand{\ti}{\tilde}

\newcommand{\ma}{\mapsto} 
\newcommand{\onto}{\twoheadrightarrow} 
\newcommand{\Rarr}{\Rightarrow} 

\def\SEMI{\mbox{$\times\kern-2pt\vrule height5pt width.6pt \kern3pt $}}

\newcommand{\Spec}{{\rm Spec\,}} 

\newcommand{\id}{{\rm id}}
\newcommand{\alg}{{\rm alg}}

\newcommand{\ev}{\operatorname{ev}} 
\newcommand{\lspan}{{\rm span}} 

\newcommand{\falls}{\text{ if }} 
\newcommand{\plim}{\varprojlim} 
\newcommand{\supp}{{\rm supp}} 

\newcommand{\lge}{\left\{} 
\newcommand{\rge}{\right\}} 
\newcommand{\gekl}[1]{\lge #1 \rge} 
\newcommand{\menge}[2]{\gekl{ #1 \colon #2 }} 
\newcommand{\bmU}{\bm{U}}
\newcommand{\bmV}{\bm{V}}

\newcommand{\oset}[2]{%
  \mathop{#2}\limits^{\vbox to -1.66ex{%
  \kern -1.4ex\hbox{$#1$}\vss}}}

\newcommand{\crL}{\mathscr{L}}

\newcommand{\crX}{\mathscr{X}}
\newcommand{\crY}{\mathscr{Y}}
\newcommand{\crZ}{\mathscr{Z}}

\newcommand{\inte}{{\rm int}}

\newcommand{\bd}{{\rm bd}}

\newcommand{\ess}{{\rm ess}}
\newcommand{\sync}{{\rm sync}}

\DeclareFontFamily{U}{mathb}{\hyphenchar\font45}
\DeclareFontShape{U}{mathb}{m}{n}{
      <5> <6> <7> <8> <9> <10>
      <10.95> <12> <14.4> <17.28> <20.74> <24.88>
      mathb10
      }{}
\DeclareSymbolFont{mathb}{U}{mathb}{m}{n}

\DeclareMathSymbol{\sqbullet}{1}{mathb}{"0D}
\usepackage{newtxtext}
\usepackage{newtxmath}

\begin{document}
\title{Minimal covers with continuity-preserving transfer operators for topological dynamical systems}

\thispagestyle{fancy}

\author[K. A. Brix]{Kevin Aguyar Brix}
\author[J. B. Hume]{Jeremy B. Hume}
\author[X. Li]{Xin Li}

\address[K. A. Brix]{Department of Mathematics and Computer Science, University of Southern Denmark, 5230 Odense, Denmark}
\email{kabrix@imada.sdu.dk}

\address[J. B. Hume]{School of Mathematics and Statistics, University of Glasgow, University Place, Glasgow G12 8QQ, United Kingdom}
\email{jeremybhume@gmail.com}

\address[X. Li]{School of Mathematics and Statistics, University of Glasgow, University Place, Glasgow G12 8QQ, United Kingdom}
\email{Xin.Li@glasgow.ac.uk}



\thanks{Brix was supported by a DFF-international postdoc (case number 1025--00004B) and a Starting grant from the Swedish Research Council (2023-03315).
This project has received funding from the European Research Council (ERC) under the European Union's Horizon 2020 research and innovation programme (grant agreement No. 817597).}

\begin{abstract}
Given a non-invertible dynamical system with a transfer operator, we show there is a minimal cover with a transfer operator that preserves continuous functions. 
We also introduce an essential cover with even stronger continuity properties.
For one-sided sofic subshifts, this generalizes the Krieger and Fischer covers, respectively.
Our construction is functorial in the sense that certain equivariant maps between dynamical systems lift to equivariant maps between their covers, and these maps also satisfy better regularity properties. 
As applications, we identify finiteness conditions which ensure that the thermodynamic formalism is valid for the covers.
This establishes the thermodynamic formalism for a large class of non-invertible dynamical systems, e.g. certain piecewise invertible maps.
When applied to semi-\'etale groupoids, our minimal covers produce \'etale groupoids which are models for $C^*$-algebras constructed by Thomsen.
The dynamical covers and groupoid covers are unified under the common framework of topological graphs.
\end{abstract}

\maketitle


\setlength{\parindent}{0cm} \setlength{\parskip}{0.5cm}

\section{Introduction}

Single transformations, i.e., continuous maps $T \colon X \to X$ on a topological space $X$, are classical objects of study in topological dynamics. 
Such transformations can be viewed as discrete-time dynamical systems, where an application of $T$ describes how the system evolves after one discrete time step. 
If $T$ is invertible, then the inverse of $T$ allows us to go back in time, but, in general, transfer operators allow us to at least go back in time on average. 
Roughly speaking, the idea is to attach to every point $x \in X$ a measure $\lambda_x$ on $X$ that is supported on the pre-image $T^{-1}(x)$ of $x$ under $T$. 
Given a function $f \colon X \to \Cz$, the transfer operator $\Lambda$ attached to $(\lambda_x)_{x \in X}$ produces (under mild assumptions on $f$ such as integrability) 
a new function $\Lambda(f)$ by averaging over pre-images, i.e., by sending $x \in X$ to $\Lambda(f)(x) \coloneq \int_{T^{-1}(x)} f {\rm d} \lambda_x$. 
If $T$ is invertible with inverse $T^{-1}$, then pre-images of points $x \in X$ are given by singletons $\gekl{T^{-1}(x)}$, and the canonical choice for $\lambda_x$ is the point mass at $T^{-1}(x)$.
In this case, the transfer operator is simply given by pre-composition with $T^{-1}$ (i.e., $\Lambda(f) = f \circ T^{-1}$). 
The reader may consult \cite{Bal}, for example, for more information about transfer operators.

In this paper, we are interested in the property of transfer operators --- sometimes called the Feller property --- that they preserve continuity, 
i.e., for every continuous (and bounded) function $f \colon X \to \Cz$, $\Lambda(f)$ is again continuous. 
In general, transfer operators need not have this property. 
However, it turns out that in various different contexts, this property is desirable for several reasons.

For example, in the context of the thermodynamic formalism (see for instance \cite{Rue,PP}), if the transfer operator preserves continuous functions, then this is the starting point for restricting the transfer operator to subspaces of more regular (such as H{\"o}lder continuous) functions where properties of the transfer operator such as quasi-compactness allow a further detailed analysis of its spectrum. In addition, for compact Hausdorff spaces $X$, positive linear functionals on the space $C(X)$ of continuous functions $X \to \Cz$ are in one-to-one correspondence with Radon measures on $X$. Hence, if the transfer operator $\Lambda$ preserves $C(X)$, its dual operator $\Lambda^*$ acts on Radon measures on $X$. This is the starting point for studying invariant measures on $X$ by analyzing eigenmeasures of $\Lambda^*$, for example equilibrium measures or Gibbs measures.

The property that transfer operators preserve continuous functions is also relevant in operator algebras. 
There is a long history of fruitful interactions between topological dynamics and $C^*$-algebras, 
and many example classes of $C^*$-algebras arising from dynamics (such as crossed products or Cuntz--Krieger algebras \cite{Cuntz-Krieger1980}) have played an important role in the development of $C^*$-algebra theory. 
At the same time, a $C^*$-algebraic perspective was instrumental in the classification of Cantor minimal systems \cite{GPS}.
Given a continuous map $T \colon X \to X$ on a locally compact Hausdorff space $X$, a transfer operator $\Lambda$ for $X$ that preserves continuous functions gives rise to a $C_0(X)$-$C_0(X)$-correspondence, 
which in turn allows one to apply Pimsner's machinery \cite{Pim} (see also \cite{Katsura2004}) to construct a $C^*$-algebra encoding dynamical properties of the original transformation $T$. 
Related to this, a very general class of $C^*$-algebras of dynamical origin is given by groupoid $C^*$-algebras \cite{Ren}. 
Given a locally compact Hausdorff groupoid $G$, its groupoid $C^*$-algebra is constructed by defining a convolution product on $C_c(G)$, the space of continuous functions $G \to \Cz$ with compact support. 
This works particularly well for {\'e}tale groupoids, where the range and source maps are local homeomorphisms. 
That property ensures that the transfer operator given by summation over pre-images of the range map preserves continuous functions.
This in turn  is the key reason why $C_c(G)$ is closed under convolution. 
For general locally compact Hausdorff groupoids, the construction of the convolution product requires the existence of a Haar system, 
which by its very definition again ensures that the corresponding transfer operator preserves continuous functions on our groupoid. 
This means that the property of preserving continuity is at the heart of many constructions of $C^*$-algebras from dynamical data.

The goal of the present paper is to describe a systematic and explicit way of constructing transfer operators with the desirable property of preserving continuity out of general transfer operators which might lack this property. 
More precisely, given a continuous map $T \colon X \to X$, where $X$ is locally compact and Hausdorff, and a transfer operator $\Lambda$ for $T$, we construct a cover (i.e., an extension) $T_{\min} \colon X_{\min} \to X_{\min}$ of $T$ together with a transfer operator $\bm{\lambda}_{\min}$ for $T_{\min}$ which extends $\Lambda$ and satisfies $\bm{\lambda}_{\min}(C_c(X_{\min})) \subseteq C_c(X_{\min})$. Moreover, $T_{\min}$ is the minimal cover with these properties, in the sense that it has a universal property as we will make precise. In addition, our cover construction is functorial, and most importantly, it is accessible, in the sense that it is possible to describe it explicitly for concrete examples, as we demonstrate in Sections \ref{sec:ExSubshifts} and \ref{sec:PiecewInv}. 
We also construct a modification called the minimal essential cover that often has better properties because it is closer to the original map.
\setlength{\parindent}{0.5cm} \setlength{\parskip}{0cm}

When applied to one-sided subshifts, our constructions reproduce and generalize the covers considered by Fischer \cite{Fischer1975}, Krieger \cite{Krieger1984}, Matsumoto \cite{Matsumoto1999} and Carlsen \cite{CarlsenPhD} (see also \cite{Lind-Marcus2021} and the references therein).  
We are therefore able to prove these classical covers have (previously unknown) functorial and universal properties, as well as expand the class of subshifts they may be considered for (especially for the Fischer covers).
\setlength{\parindent}{0cm} \setlength{\parskip}{0.5cm}


We identify finiteness conditions which ensure that our minimal cover $T_{\min}$ is positively expansive if our original map $T$ is.  Since the thermodynamic formalism has been established for factors of  open positively expansive maps that are topologically transitive, our cover construction (which is always open) allows us to establish the thermodynamic formalism for the original transformation whose transfer operator might be badly behaved. In addition, the analysis of our minimal cover for particular examples such as the tent map has led us to the construction of another, related cover for general piecewise invertible maps, and --- again under a finiteness assumption --- produces positively expansive covers which again allow us to establish a thermodynamic formalism for piecewise invertible maps.

In a different setting, for locally compact Hausdorff groupoids which are not {\'e}tale but whose range and source maps are only locally injective (Thomsen calls these semi-{\'e}tale), our construction produces a cover groupoid which is {\'e}tale and whose reduced groupoid $C^*$-algebra coincides with Thomsen's construction in \cite{Tho}. This provides a conceptual explanation for the work in \cite{Tho,Tho2011}. Whenever given an {\'e}tale groupoid $G$ and an arbitrary closed (but not necessarily invariant) subspace $C$ of the unit space of $G$, the restricted groupoid $G_C^C$ has locally injective range and source maps (but it will not be {\'e}tale in general because $C$ is not assumed to be invariant). Hence our construction applies to all these examples and produces {\'e}tale groupoids for each of them.

Let us now explain our construction and main results in more detail. 
We work in the general setting of \emph{topological graphs}, which encompasses single transformations as well as topological groupoids. 
A topological graph $Z$ is a pair of locally compact Hausdorff spaces $Y$ and $X$ with continuous maps $r,s\colon Y \to X$. 
We regard $Y$ as the set of edges, $X$ as the set of vertices, and $r$, $s$ are the range and source maps. 
However, we may also think of a topological graph as defining a generalized discrete-time dynamical system on $X$, where $s(y)$ is sent to $r(y)$ for each $y\in Y$.
\setlength{\parindent}{0.5cm} \setlength{\parskip}{0cm}

In its most abstract form, a transfer operator for a topological graph $Z$ is a positive linear operator $\Lambda \colon \mathcal{D}\to\mathcal{R}$ 
defined on a sub-algebra $\mathcal{D}\subseteq \ell^{\infty}_{c}(Y)$ of compactly supported $\Cz$-valued functions on the edges $Y$ into a sub-algebra $\mathcal{R}\subseteq \ell^{\infty}_{c}(X)$ 
(subject to compatibility conditions) 
such that each linear functional $f\mapsto \Lambda(f)(x)$, $x\in X$, is supported on $r^{-1}(x)$.
When the fibres $r^{-1}(x)$ are discrete, we can induce a transfer operator from a weight function $w\colon Y \to [0,\infty)$ by the formula $\Lambda(f)(x) \coloneq \sum_{y \in r^{-1}(x)} w(y) f(y)$, 
where $f\in \ell^{\infty}_{c}(Y)$.
If the fibres are not discrete, then a typical example of a transfer operator comes from Borel measures $\lambda = (\lambda)_{x\in X}$ with $\text{supp}(\lambda_{x})\subseteq r^{-1}(x)$ for $x\in X$ 
such that the function $x\mapsto \int f {\rm d} \lambda_x$ is Borel for every Borel function $f \colon Y\to \mathbb{C}$ with compact support. 
\setlength{\parindent}{0cm} \setlength{\parskip}{0.5cm}

Given a topological graph $Z = (r,s \colon Y\to X)$ and a transfer operator $\Lambda$, 
its \emph{minimal cover} is a new topological graph $\hat{Z} = (\hat{r}, \hat{s} \colon \hat{Y}\to\hat{X})$ equipped with a transfer operator $\hat{\bm{\lambda}}$ that is continuous 
in the sense that $\hat{\bm{\lambda}}(C_{c}(\hat{Y}))\subseteq C_{c}(\hat{X})$. 
There are continuous proper surjections $\pi_{\hat{Y}} \colon \hat{Y} \to Y$, $\pi_{\hat{X}} \colon \hat{X} \to X$, 
injective (in general non-continuous) maps $\iota_Y \colon Y \to \hat{Y}$, $\iota_X \colon X \to \hat{X}$ with dense images such that $(\pi_{\hat{Y}},\pi_{\hat{X}})$ and $(\iota_Y,\iota_X)$ are graph morphisms 
with $(\pi_{\hat{Y}}, \pi_{\hat{X}}) \circ (\iota_Y, \iota_X) = (\id_Y, \id_X)$, and in addition, 
we have $\Lambda(f \circ \iota_Y) = \hat{\bm{\lambda}}(f) \circ \iota_X$ for all $f \in C_c(\hat{Y})$. 
The minimal cover is universal for such maps (see Section \ref{ss:cover_top_graph} for details).
We also construct a modification we call the \emph{minimal essential cover}.
It is given by a sub-graph $\tilde{Z} = (\ti{r}, \ti{s} \colon \ti{Y}\to\ti{X})$ of $\hat{Z}$ with transfer operator 
$\tilde{\bm{\lambda}} \colon  C_{c}(\tilde{Y})\to C_{c}(\tilde{X})$ given by the restriction of $\hat{\bm{\lambda}}$ to $C_{c}(\tilde{Y})$,
and it remedies the discontinuity of the embeddings $\iota_Y$, $\iota_X$ (at the expense that they are only defined on a dense sub-graph).
The essential cover is in fact closer to the original topological graph and is therefore often a better tool for the analysis of the original system.
This construction does however require a mild condition we call post-critical sparsity (see Section \ref{ss:min_ess}).
\setlength{\parindent}{0.5cm} \setlength{\parskip}{0cm}

By viewing a single transformation with a transfer operator as a topological graph (every continuous map $T \colon X \to X$ gives rise to the topological graph $Y = X$, $r = T$, $s = \id_X$),
we obtain minimal dynamical covers $(T_{\min} \colon X_{\min} \to X_{\min} , \bm{\lambda}_{\min})$ and minimal essential covers $(T_{\ess} \colon X_{\ess} \to X_{\ess},\bm{\lambda}_{\ess})$ of single transformations.
However, to ensure that the cover is again a single transformation, we insist that $\hat{Y} = \hat{X}$ and $\hat{s} = \id_{\hat{X}}$. 
It turns out that the minimal dynamical cover is the shift on the (one-sided) infinite path space of the minimal topological graph cover (see Section \ref{s:SingleTrafo}).
\setlength{\parindent}{0cm} \setlength{\parskip}{0.5cm}


Let us now present our main results. In this paper, we especially focus on graphs whose range maps $r \colon Y\to X$ are locally injective (every edge $y\in Y$ has an open neighbourhood $U$ such that $r|_{U}$ is injective). 
The natural transfer operator is induced from the weight function $w=1$, i.e., $\Lambda(f)(x) = \sum_{y\in r^{-1}(x)}f(y)$. 
In this case, the range map of the minimal cover $\hat{r} \colon \hat{Y}\to \hat{X}$ is a local homeomorphism (open and locally injective), and the transfer operator is also induced from the weight $w=1$. 
Similar statements hold for the minimal dynamical cover as well as for the minimal essential cover.
\setlength{\parindent}{0.5cm} \setlength{\parskip}{0cm}

A map $T$ is \emph{open at $x$} if $T(x)$ is an interior point of $T(U)$ for every open neighbourhood $U$ of $x$.
The \emph{critical point set of $T$} is given as $C_{T} = \menge{x\in X}{T\text{ is not open at }x }$. 
Recall also that a subset $C\subseteq X$ is \emph{meagre} if it is a countable union of sets whose closures have empty interior.
\setlength{\parindent}{0cm} \setlength{\parskip}{0.5cm}

\begin{introtheorem}[See Theorem~\ref{thm:up_loc_inj}]
\label{thm:thefirstone}
Suppose $T \colon X\to X$ is locally injective and $X$ is a second countable locally compact Hausdorff space. If the critical point set $C_{T}$ is meagre, then the minimal essential cover $T_{\ess} \colon X_{\ess} \to X_{\ess}$ exists and is a local homeomorphism defined on a second countable locally compact Hausdorff space. The factor map $\pi_{\ess} \colon (T_{\ess},X_{\ess})\to (T, X)$ is continuous, almost one-to-one and proper.

If $(\dot{T}, \dot{X})$ is another local homeomorphism defined on a second countable locally compact Hausdorff space with continuous, almost one-to-one and proper factor map $\dot{\pi} \colon (\dot{T}, \dot{X})\to (T,X)$, then there is a unique continuous factor map $\chi \colon (\dot{T}, \dot{X})\to (T_{\ess}, X_{\ess})$ such that $\dot{\pi} = \pi_{\ess} \circ \chi$. Moreover, $\chi$ is proper, almost one-to-one and $\chi$ restricts to bijections $(\dot{T})^{-1}(\dot{x})\to (T_{\ess})^{-1}(\chi(\dot{x}))$, for all $\dot{x} \in \dot{X}$.
\end{introtheorem}
\setlength{\parindent}{0cm} \setlength{\parskip}{0cm}

For a subshift $\sigma \colon X\to X$, its critical point set $C_{\sigma}$ is meagre if and only if the interior of $\sigma(X)$ is dense in $\sigma(X)$ (see Corollary \ref{cor:loc_inj_meagre} and Lemma \ref{weakly_open}).
\setlength{\parindent}{0cm} \setlength{\parskip}{0.5cm}

The last mapping property in Theorem \ref{thm:thefirstone} is an analogue of $s$-bijective factor maps between Smale spaces. 
Indeed, if $\chi \colon (T, X)\to (T',X')$ is a factor map between two pre-Wieler solenoids (systems satisfying Axiom 1 and 2 of \cite{Wie}), 
then the induced factor map between their natural extensions (which are Smale spaces by \cite{Wie}) is $s$-bijective if and only if $\chi$ restricts to bijections $T^{-1}(x) \to (T')^{-1}(\chi(x))$, for all $x\in X$. 
In light of this, the following lifting result for almost one-to-one factor maps between local injections is an analogue of Putnam's lifting result in \cite[Theorem~1.1]{Put}.

\begin{introtheorem}[See Theorem~\ref{thm:loc_inj_lifting_thm}]
    Suppose $(T, X)$ and $(T', X')$ are locally injective dynamical systems defined on second countable locally compact Hausdorff spaces with meagre critical point sets, 
    and suppose $\chi \colon (T, X)\to (T', X')$ is a continuous, almost one-to-one and proper factor map. 
    Then there is a unique continuous factor map $\chi_{\ess} \colon (T_{\ess}, X_{\ess}) \to (T'_{\ess}, X'_{\ess})$ such that $\pi_{\ess} \circ \chi_{\ess} = \chi \circ \pi'_{\ess}$. 
    Moreover, $\chi_{\ess}$ is almost one-to-one, proper and $\chi_{\ess}$ restricts to bijections $T_{\ess}^{-1}(\tilde{x}) \to (T'_{\ess})^{-1}(\chi_{\ess}(\tilde{x}))$, for all $\tilde{x} \in X_{\ess}$.
\end{introtheorem}

The construction of the minimal topological graph cover proceeds in steps. 
We first construct a sequence $Z_n$ of topological graphs with a system of maps between the underlying spaces, and then the minimal cover $\hat{Z}$ arises as the projective limit of these $Z_n$ (see Section \ref{ss:DescCovers}). 
With applications to the thermodynamic formalism in mind, we identify a finiteness condition which ensures that positive expansivity is inherited by the minimal dynamical cover $T_{\min}$ from the original transformation $T$. First of all, a continuous map $T\colon X\to X$ on a compact metric space $(X,d)$ is \emph{positively expansive} if there is a constant $c>0$ such that for every two distinct $x, z\in X$ there is $n\in \Nz$ such that $d(T^n x, T^n z) \geq c$. Note that a positively expansive map must be locally injective, and we assume throughout that the transfer operator is induced from the function $w=1$. We establish the following relation between positive expansivity of $T$ and $T_{\min}$.

\begin{introtheorem}[see Theorem~\ref{thm:hatTExpansive-check}]
Let $T\colon X\to X$ be a continuous and positively expansive map on a compact metric space $X$. If our construction of the minimal topological graph cover terminates after finitely many steps, i.e. $\hat{Z} = Z_n$ for some $n \in \Nz$, then the minimal dynamical cover $T_{\min} \colon X_{\min} \to X_{\min}$ is continuous, open, and positively expansive.
\end{introtheorem}
\setlength{\parindent}{0cm} \setlength{\parskip}{0cm}

Among subshifts, it is precisely the sofic systems for which our construction terminates after finitely many steps. 
Interestingly, the minimal \emph{dynamical} cover rarely terminates in finitely many steps, even for sofic systems, and this emphasizes the importance of the graph perspective.
The above theorem leads to the following result, which has applications for the thermodynamic formalism.

\begin{introcor}[see Corollary~\ref{cor:Prep-TDF}]
\label{cor:intro}
Let $T\colon X\to X$ be a continuous and positively expansive map on a compact metric space $(X,d)$. Assume that $T$ is topologically transitive, that $T$ has meagre critical set, and that our construction of the minimal topological graph cover from Section \ref{ss:DescCovers} terminates after finitely many steps. Then the minimal essential dynamical cover $T_{\ess}$ is open, topologically transitive and positively expansive.
\end{introcor}
The point is that the original transformation $T$ will not be open in general (see for example \cite{Hir,Ros}). As we explain in Remark~\ref{rem:TDF-FiniteSteps}, it follows that the thermodynamic formalism holds for transformations $T \colon X \to X$ satisfying the conditions in Corollary~\ref{cor:intro}.
\setlength{\parindent}{0cm} \setlength{\parskip}{0.5cm}

Moreover, the analysis of our minimal cover leads to the following application to piecewise invertible maps (which we introduce in \S~\ref{sec:PiecewInv}, but compare \cite{Hof,Tsu} and the references therein for similar notions).
\begin{introtheorem}[see Corollary~\ref{cor:PiecewInvNiceCover}]
\label{thm:intro-PiecewInv}
Let $T$ be an expansive piecewise invertible map in the sense of \S~\ref{sec:PiecewInv} which is topologically transitive. Assume that $T$ satisfies the finiteness condition from Proposition~\ref{prop:PiecewInv-SFT}. Then $T$ admits a cover consisting of a topologically transitive shift of finite type.
\end{introtheorem}
\setlength{\parindent}{0cm} \setlength{\parskip}{0cm}

Note that $T$ does not need to be open nor positively expansive in the sense explained above. Again, Theorem~\ref{thm:intro-PiecewInv} has applications to the thermodynamic formalism because the formalism has been established for topologically transitive shifts of finite type (see for instance \cite{PP,PU}).
\setlength{\parindent}{0cm} \setlength{\parskip}{0.5cm}

Let us now present the main application of our construction in the setting of groupoids and their $C^*$-algebras. 
Let $G$ be a locally compact Hausdorff groupoid with locally injective range and source maps. 
Given two functions $f, g \in C_c(G)$, the usual convolution $(f * g)(\gamma) \coloneq \sum_{\alpha \beta = \gamma} f(\alpha) g(\beta)$ is bounded with compact support though it need not be continuous.
This observation led Thomsen to construct a $C^*$-algebra for $G$ using the analogue of the left regular representation for {\'e}tale groupoids (see \cite{Tho}). 
Our cover construction yields a conceptual explanation for Thomsen's work.
In the special case of an action by a local injection, Thomsen already established groupoid models for this construction in \cite{Tho2011}.

\begin{introtheorem}[see Proposition~\ref{prop:B=C*G}]
View $r, s \colon G \to G^{(0)}$ as a topological graph in the sense above, with transfer operator induced from the function $w=1$. Let $\hat{G}$ be our minimal cover for this weighted topological graph.
Then $\hat{G}$ carries the structure of an {\'e}tale groupoid, and the groupoid structure is uniquely determined by the property that it extends the groupoid structure on $G$. Moreover, the reduced groupoid $C^*$-algebra $C^*_r(\hat{G})$ is canonically isomorphic to the $C^*$-algebra of $G$ which Thomsen constructed in \cite{Tho}.
\end{introtheorem}

\section{Preliminaries}\label{prelim}

We will use the following notations: We let $\Nz$ denote the non-negative integers, including zero. In a topological space, $\inte(\cdot)$ stands for interior, $\overline{(\cdot)}$ for closure and $\bd(\cdot)$ for boundary. Throughout this paper, if $f\colon Y \to X$ is a map we let $\supp^\circ(f) \coloneq \menge{ y\in Y }{ f(y) \neq 0 }$ be the open support (which is open when $f$ is continuous, but need not be open in general),
and the support $\supp(f)$ of $f$ is defined as the closure of the open support in $Y$, $\supp(f) \coloneqq \overline{(\supp^\circ(f))}$.

Given a locally compact Hausdorff space $Y$, we denote by $C_c(Y)$ the set of continuous complex-valued functions $Y \to \Cz$ with compact support,
and we let $C_0(Y)$ denote the set of continuous complex-valued functions $Y \to \Cz$ which vanish at infinity.
In other words, $C_0(Y)$ is the closure of $C_c(Y)$ with respect to the uniform norm $\Vert \cdot \Vert$. 
We let $C_b(Y)$ denote the set of continuous bounded complex-valued functions $Y \to \Cz$. 
We will also use the set $\ell^{\infty}_c(Y)$ of bounded complex-valued functions $Y \to \Cz$ with compact support. 

The \emph{Baire $\sigma$-algebra} is the smallest $\sigma$-algebra turning all $f\in C_{c}(Y)$ measurable, where $\mathbb{C}$ is equipped with the Borel $\sigma$-algebra. Alternatively, it is the $\sigma$-algebra generated by all compact $G_{\delta}$ sets \cite{Cohn}. When $Y$ is second countable, the Baire $\sigma$-algebra coincides with the Borel $\sigma$-algebra. In general, it is strictly smaller. We shall denote the bounded Baire measurable functions with compact support by $B_{c}(Y)$. The reason we work with $B_c(Y)$ is because (non-negative) Baire measures are in one-to-one correspondence with positive linear functionals $I \colon B_c(Y) \to \Cz$ which satisfy $\lim_n I(f_n) = 0$ whenever $f_n \in B_c(Y)$ is a monotonic sequence converging pointwise to $0$, and moreover such functionals are uniquely determined by their restrictions to $C_c(Y)$. Alternatively, the reader may restrict to second countable spaces and replace Baire measures by Radon measures throughout.

Recall that a net $(f_{\lambda})_{\lambda}$ in $\ell^{\infty}_c(Y)$ converges in the inductive limit topology to $f$ if there is a compact set $K$ such that $\text{supp}(f_{\lambda})\subseteq K$ eventually and $(f_{\lambda})_{\lambda}$ converges uniformly on $K$ to $f$. This turns $\ell_{c}^{\infty}(Y)$ into a topological $^*$-algebra. For a topological $^*$-algebra $A$, its \emph{multiplier algebra} \cite[Section~3]{Joh} $M(A)$ consists of pairs $(R,L)$ of continuous $A$-linear operators satisfying $aL(b) = R(a)b$, for all $a,b\in A$.  The multiplier algebra of $\ell^{\infty}_{c}(Y)$ consists of $(R,L) = (f,f)$, where $f \colon X\to\mathbb{C}$ is a function (no assumptions), i.e., $R$ and $L$ are given by pointwise multiplication by arbitrary $\Cz$-valued functions.

Let $r\colon Y\to X$ be a map between topological spaces. Let $\mathcal{D}$ be a $^*$-sub-algebra of $\ell^{\infty}_{c}(Y)$ that is closed in the inductive limit topology and contains $C_{c}(Y)$. We shall call such an algebra a \emph{domain}. A $^*$-sub-algebra $\mathcal{R}\subseteq \ell_{c}^{\infty}(X)$ that is closed in the inductive limit topology of uniform convergence and contains $C_{c}(X)$ will be called a \emph{range}.

We will consider positive linear operators of the form $\Lambda \colon \mathcal{D}\to \mathcal{R}$, where $\mathcal{D}$ is a domain and $\mathcal{R}$ is a range satisfying $\mathcal{D}r^{*}(\mathcal{R})\subseteq \mathcal{D}$ and $\Lambda(f)(x) = 0$ whenever $f|_{r^{-1}(x)}  = 0$. 
We call such an operator a \emph{transfer operator}, or, alternatively, an \emph{$r$-operator} when we want to emphasize the original map $r$.
The last condition implies
\begin{equation}\label{transfer_equation}
    \Lambda(f(g
\circ r)) = \Lambda(f)g,\text{ for all }f\in \mathcal{D}\text{ and }g\in \mathcal{R}.
\end{equation}

By positivity and $C_{c}(Y)\subseteq \mathcal{D}$, we immediately have the property that $\Lambda$ is 
\setlength{\parindent}{0cm} \setlength{\parskip}{0cm}

\begin{enumerate}
    \item[(BA):]\emph{bounded above} in the sense that for every compact $K\subseteq Y$, there is $M_{K} > 0$ such that $\|\Lambda(f)\|\leq M_{K}\|f\|$ 
      for every positive $f\in \mathcal{D}$ with $\text{supp}(f)\subseteq K$.
\end{enumerate}
Let us consider some other properties of $\Lambda$ that will be important to us later. We will say $\Lambda$ is
\begin{enumerate}
    \item[(C):] \emph{continuous} if $\mathcal{D} = C_{c}(Y)$ and $\mathcal{R}= C_{c}(X)$,
    \item[(BB):] \emph{bounded below} if for every compact $K\subseteq Y$, there is $N_{K} > 0$ such that $\Lambda(f)\circ r\geq N_{K}f$ for every positive $f\in C_{c}(Y)$ with $\text{supp}(f)\subseteq K$, and
    \item[(F):] \emph{full} if for every $y\in Y$, we have $\Lambda(f)(r(y)) > 0$ for every positive $f\in C_{c}(Y)$ with $f(y) > 0$.
\end{enumerate}
\setlength{\parindent}{0cm} \setlength{\parskip}{0.5cm}

We now study these conditions.

\subsection{Continuity}
We describe some consequences of continuity of an $r$-operator and an equivalent formulation in terms of Radon measures.

The support of a measure $\mu$ on a locally compact Hausdorff space $Y$ is defined as follows: 
first let $N$ be the biggest open subset of $Y$ such that $\mu(N) = 0$. 
The support $\supp(\mu)$ of $\mu$ is defined as the complement of $N$ in $Y$, $\supp(\mu) \coloneq Y \setminus N$. 
We will need the following characterization of points in the support: 
a point $y \in Y$ lies in $\supp(\mu)$ if and only if $\int f {\rm d}\mu > 0$ for every $f \in C_c(Y)$ with $0 \leq f \leq 1$ and $f(y) > 0$ (see for instance \cite[Exercise~7.2]{Fol}).

\begin{lemma}\label{charecterization}
A map $\Lambda\colon C_{c}(Y)\to C_{c}(X)$ is an $r$-operator if and only if $\Lambda(f(g\circ r)) = \Lambda(f)g$ for all $f\in C_{c}(Y)$ and $g\in C_{c}(X)$.
\end{lemma}
\setlength{\parindent}{0cm} \setlength{\parskip}{0cm}

\begin{proof}
The \an{only if} part is trivial. If $f|_{r^{-1}(x)} = 0$, then by compactness of the support of $f$ and continuity there is a positive function $\psi\in C_{c}(Y)$ such that for every $\varepsilon > 0 $, there is $0\leq \phi\leq 1$ in $C_{c}(X)$ with $\phi(x) = 1$ and $f(\phi\circ r)\leq \varepsilon\psi$. Hence, $\Lambda(f)(x) = \Lambda(f(\phi\circ r))(x)\leq \varepsilon\Lambda(\psi)(x)$. As $\varepsilon > 0$ is arbitrary, it follows that $\Lambda(f)(x) = 0$.
\end{proof}

\begin{lemma}\label{full}
If $\Lambda$ is full, then for every positive $f\in C_{c}(Y)$, we have $\supp^\circ(\Lambda(f)) = r(\supp^\circ(f))$.
\end{lemma}
\begin{proof}
$\Lambda$ being full is equivalent to $r(\supp^\circ(f))\subseteq \supp^\circ(\Lambda(f))$. The reverse direction follows from $\Lambda(f)(x) = 0$ whenever $f|_{r^{-1}(x)} = 0.$ 
\end{proof}
\begin{lemma}\label{lem:full}
If $\Lambda$ is full and continuous, then $r$ is open.
\end{lemma}
\begin{proof}
This follows from Lemma \ref{full}, continuity of $\Lambda(f)$, for $f\geq 0$ in $C_{c}(Y)$, and the fact that a basis for the topology on $Y$ is given by $\menge{\supp^\circ(f)}{f\geq 0 \text{ and }f\in C_{c}(Y) }$.
\end{proof}
\begin{remark}
  Conversely, if $Y$ is second countable and $r \colon Y \to X$ is a continuous open surjection, then Blanchard showed that there exists a continuous and full $r$-operator (see e.g. \cite[Corollary~B.18]{Wil}).
\end{remark}
\setlength{\parindent}{0cm} \setlength{\parskip}{0.5cm}

Recall from \cite[Section 3.2]{Wil} that an \emph{$r$-system} is a family of Radon measures $(\lambda_{x})_{x\in X}$ on $Y$ such that $\text{supp}(\lambda_{x})\subseteq r^{-1}(x)$ for all $x\in X$, and the function  $x\mapsto\int f {\rm d} \lambda_x$ is in $C_{c}(X)$ for all $f\in C_{c}(Y)$. 
We may define a continuous $r$-operator $\Lambda\colon C_{c}(Y)\to C_{c}(X)$ via $\Lambda(f)(x) = \int f {\rm d} \lambda_x$.
Conversely, if $\Lambda\colon C_{c}(Y)\to C_{c}(X)$ is a continuous $r$-operator, then the positive linear functional $f\mapsto \Lambda(f)(x)$ induces, by the Riesz-representation theorem, a Radon measure $\lambda_{x}$, and $(\lambda_{x})_{x\in X}$ is an $r$-system.

Using Gelfand duality, there is a way to view every $r$-operator in our sense as continuous (and therefore induced from an $r$-system) on a \emph{covering} of the original space, as we will explain in \S~\ref{ss:cover_top_graph}. The following lemma is the first ingredient.
Below, $\Spec(A)$ denotes the space of characters (i.e., non-zero homomorphisms $A \to \Cz$) endowed with the topology of pointwise convergence. 
There is a canonical homeomorphism $\Spec(A) \cong \Spec(\overline{A})$, where $\overline{A}$ is the closure of $A$ in the uniform norm.

\begin{lemma}\label{cc}
Let $W$ be a locally compact Hausdorff space and $A\subseteq \ell_{c}^{\infty}(W)$ be a $^*$-sub-algebra that is closed in the inductive limit topology and contains $C_{c}(W)$. Then $\hat{A} = C_{c}(\Spec(A))$, where $\hat{(-)}\colon A \to C_{0}(\Spec(A))$ is the Gelfand transform. Moreover, $\hat{(-)}\colon A\to C_{c}(\Spec(A))$ is a homeomorphism relative to the inductive limit topologies on $A$ and $C_{c}(\Spec(A))$.
\end{lemma}
\setlength{\parindent}{0cm} \setlength{\parskip}{0cm}

\begin{proof}
Set $\cW \coloneq \Spec(A)$ and let $\overline{A}$ be the closure of $A$ as above. Let $\iota \colon W \to \cW$ be defined as $\iota(z) = \text{ev}_{z} \colon A\to\mathbb{C}$. The dual of the inclusion $C_{c}(W)\subseteq A$ is a proper and continuous map $\pi \colon \cW \to W$. It follows from $\pi\circ\iota = \text{id}_{W}$ and $\iota^* \circ \hat{(-)} = \text{id}_{A}$ (i.e., $\hat{a} \circ \iota = a$ for all $a \in \overline{A}$) that for $f\in \overline{A}$, we have $\text{supp}(\hat{f})\subseteq \pi^{-1}(\text{supp}(f))$ and $\text{supp}(f)= \text{supp}(\hat{f}\circ\iota)\subseteq \pi(\text{supp}(\hat{f}))$. Therefore, the image of $\overline{A}\cap\ell^{\infty}_{c}(Z)$ under $\hat{(-)}$ is $C_{c}(\cW)$. 
\setlength{\parindent}{0.5cm} \setlength{\parskip}{0cm}

To finish proving $\hat{A} = C_{c}(\Spec(A))$, it suffices to prove $\overline{A}\cap\ell^{\infty}_{c}(W) = A$. Since $A$ is dense in $\overline{A}$, for every $a\in \overline{A}\cap\ell^{\infty}_{c}(W)$, we can find a sequence $a_{n}\in A$ such that $a_{n}\to a$ uniformly. Let $\phi\in C_{c}(W)$ be a function such that $a\phi = a$. Then $a_{n}\phi\in A$ and $a\phi_{n}\to a$ in the inductive limit topology. Hence, $a\in A$.
From $\text{supp}(\hat{f})\subseteq \pi^{-1}(\text{supp}(f))$ and $\text{supp}(f)= \text{supp}(\hat{f}\circ\iota)\subseteq \pi(\text{supp}(\hat{f}))$ for $f\in A$, it follows that $\hat{(-)}$ is a homeomorphism relative to the inductive limit topologies.
\end{proof}

\begin{remark}\label{rmk:multiplier}
If $A$ is an algebra satisfying the hypothesis of Lemma \ref{cc}, then $M(A)$ is canonically isomorphic to the algebra of all functions $f \colon W \to \Cz$ which satisfy the condition that $f \phi \in A$ for all $\phi\in C_{c}(W)$.
\end{remark}
\setlength{\parindent}{0cm} \setlength{\parskip}{0.5cm}

\subsection{The bounded below condition}
Let us first consider some examples satisfying the bounded below condition.

\bdefin
We will say that $r \colon Y \to X$ is \emph{locally finite degree} if for every compact set $K\subseteq Y$, $\text{deg}_{K}(r) \coloneq \sup_{x\in X}|r^{-1}(x)\cap K| <\infty$. We shall call a function $w \colon Y \to (0,\infty)$ a \emph{strictly positive weight} if for every compact set $K\subseteq Y$, there are $C_{K}, D_{K} > 0$ such that $C_{K}\geq w(y)\geq D_{K}$ for all $y\in K$. 
\edefin
\setlength{\parindent}{0cm} \setlength{\parskip}{0cm}

Note that a function $w$ is a strictly positive weight if and only if $(w^{1/n})_{n\in\mathbb{N}}$ is an approximate unit for $\ell_{c}^{\infty}(Y)$ equipped with the inductive limit topology, and this property is where the name is derived from.
\setlength{\parindent}{0cm} \setlength{\parskip}{0.5cm}

The following observation is easy to see.
\blemma
When $r$ is locally finite degree, a function $w \colon Y \to (0,\infty)$ determines a transfer operator $\Lambda \colon \ell_{c}^{\infty}(Y)\mapsto \ell^{\infty}_{c}(X)$ by the formula
\[\Lambda(f)(x) = \sum_{y\in r^{-1}(x)}w(y)f(y)\]
which is bounded above. If $w$ is a strictly positive weight, then $\Lambda$ is also bounded below.
\elemma

Moreover, the converse holds, as we now explain.
\begin{proposition}
If $\Lambda$ is bounded below, then $r$ is locally finite degree. Furthermore, there is a strictly positive weight $w$ such that $\Lambda(f)(x) = \sum_{y\in r^{-1}(x)}w(y)f(y)$, for all $f\in \mathcal{D}$.
\end{proposition}
\setlength{\parindent}{0cm} \setlength{\parskip}{0cm}

\begin{proof}
Let $K\subseteq Y$ be a compact set. Choose an open set $U$ such that $K\subseteq U$ and $\overline{U} \eqcolon L$ is compact, and let $M_{L}, N_{L} > 0$ be the numbers in the bounded above and bounded below conditions. If $\{y_{i}\}^{m}_{i=1}\subseteq r^{-1}(x)\cap K$ is a collection of distinct points, then let $\{\phi_{i}\}^{m}_{i=1}\subseteq C_{c}(U)$ be a collection of positive functions such that $0\leq \phi_{i}\leq 1$, $\phi_{i}(y_{i}) = 1$ and $\phi_{i}\phi_{j} = 0$, for all $i\neq j\leq m$. Hence, $0\leq \phi = \sum^{m}_{i=1}\phi_{i}\leq 1$. By the bounded above condition, we have $M_{L}\geq\Lambda(\phi)$ and by  the bounded below condition, we have  $\Lambda(\phi_{i})(x)\geq N_{L}$ for all $1 \leq i \leq m$. Hence, $M_{L}\geq \Lambda(\phi)(x)\geq mN_{L}$. Therefore, $\frac{M_{L}}{N_{L}}\geq m $. It follows that $\text{deg}_{K}(r)\leq \frac{M_{L}}{N_{L}} < \infty$.
\setlength{\parindent}{0cm} \setlength{\parskip}{0.5cm}

Since $r^{-1}(x)$ is discrete, for every $y\in r^{-1}(x)$ we may choose some positive $\phi_{y}\in C_{c}(Y)$ such that $\supp(\phi_y)\cap r^{-1}(x) = \{y\}$ and $\phi_y(y) = 1$. Note that $w(y) \coloneq \Lambda(\phi_{y})$ does not depend on the choice of $\phi_{y}$. For every $f\in \mathcal{D}$, we have $\text{supp}(f)\cap r^{-1}(x)$ is finite. Hence, $g = f - \sum_{y\in r^{-1}(x)}\phi_{y}f(y)$ is in $\mathcal{D}$ and satisfies $g|_{r^{-1}(x)} = 0$. Therefore, 
\[\Lambda(f)(x) = \Lambda \Big( \sum_{y\in r^{-1}(x)}\phi_{y}f(y) \Big) (x) = \sum_{y\in r^{-1}(x)}w(y)f(y).\]
Since for $y\in K$, we may choose $\phi_{y}\in C_{c}(U)$, it follows from the bounded above and bounded below conditions that $M_{L}\geq w(y) = \Lambda(\phi_{y})(x)\geq N_{L}\phi_{y}(y) = N_{L}$ so we can choose $C_{K} = M_{L}$ and $D_{K} = N_{L}.$
\end{proof}
Hence every $r$-operator that is bounded below has a unique extension to an $r$-operator that is bounded below with $\mathcal{D} = \ell^{\infty}_{c}(Y)$ and $\mathcal{R}= \ell_{c}^{\infty}(X)$, and is completely determined by a strictly positive weight $w$.
\setlength{\parindent}{0cm} \setlength{\parskip}{0.5cm}

\bex
\label{weights}
The $r$-operators that are bounded below are not the only examples determined by a function $w \colon Y\to [0,\infty)$. Indeed, $\Lambda(f)(x) = \sum_{y\in r^{-1}(x)}w(y)f(y)$, $f\in\ell^{\infty}_{c}(Y)$, $x\in X$, defines an $r$-operator $\Lambda \colon \ell_{c}^{\infty}(Y)\to \ell_{c}^{\infty}(X)$ if and only if for every compact set $K\subseteq Y$, there is $M_{K} > 0$ such that
$M_{K} \geq \sum_{y\in r^{-1}(x)\cap K}w(y)$ for all $x\in X$. We shall call a function $w$ satisfying this property a \emph{weight}.

\eex

\bremark
The last example of $r$-operators were of a special type with $\mathcal{D} = \ell^{\infty}_{c}(Y)$ and $\mathcal{R}= \ell^{\infty}_{c}(X)$. An $r$-operator with maximal domain and range can always be thought of as an extension of the Baire measures $(\lambda_{x})_{x\in X}$ determined by the functionals $(\Lambda(-)(x)|_{C_{c}(Y)})_{x\in X}$. Conversely, suppose $(\lambda_{x})_{x\in X}$ is a system of Baire measures such that $\text{sup}_{x\in X}\lambda_{x}(K) < \infty$ for every compact $G_{\delta}$ set $K$ and suppose, for every $x\in X$, we choose an extension of $\int (-){\rm d} \lambda_x$ to a positive linear functional $\Lambda(-)(x) \colon \ell^{\infty}_{c}(r^{-1}(x))\to \mathbb{C}$. Such a choice always exists by the Hahn--Banach theorem and the Axiom of Choice. By positivity of the extensions $\Lambda(-)(x)$ and the fact that every element in $\ell^{\infty}_{c}(Y)$ is dominated by a function in $C_{c}(Y)$ with the same norm and the support depending only on the support of the element, it follows that $\Lambda \colon \ell^{\infty}_{c}(Y)\to \ell^{\infty}_{c}(X)$ is an $r$-operator.
\eremark

\section{Topological graphs and their covers}
\label{s:top_graphs}
\subsection{Topological graphs}
A \emph{topological graph} is a pair of locally compact Hausdorff spaces $Y$ and $X$ (edges and vertices) with continuous maps $r,s\colon Y \to X$
that assign the range and source, respectively, to each edge.  We shall use the notation $Z = (r,s \colon Y\to X)$ to refer to a topological graph. When we refer to $Z$ satisfying a topological property, it shall always mean \emph{both} $Y$ and $X$ satisfy it, unless otherwise stated. 
\bremark
A topological graph equipped with a full and continuous $r$-operator is a topological quiver in the sense of Muhly and Tomforde \cite{Muhly-Tomforde2005}.
If $r$ is a local homeomorphism, then $r,s\colon Y \to X$ is a topological graph in the sense of Katsura \cite{Katsura2004}.
\eremark
\setlength{\parindent}{0cm} \setlength{\parskip}{0cm}

Given two topological graphs $Z = (r, s \colon Y \to X)$ and $Z' = (r', s' \colon Y' \to X')$, a \emph{graph morphism} from $Z$ to $Z'$ is given by a pair $\zeta = (\zeta_{Y},\zeta_{X})$ of maps $\zeta_{Y} \colon Y \to Y'$, $\zeta_{X} \colon X \to X'$ such that $r' \circ \zeta_{Y} = \zeta_{X} \circ r$ and $s' \circ \zeta_{Y} = \zeta_{X} \circ s$. We will use the notation $\zeta \colon Z\to Z'$ to refer to a graph morphism. Graph morphisms are composed component-wise. If we refer to $\zeta$ satisfying some mapping property (e.g. continuous), it shall always mean \emph{both} $\zeta_{Y}$ and $\zeta_{X}$ satisfy it, unless otherwise stated.
\setlength{\parindent}{0cm} \setlength{\parskip}{0.5cm}

\subsection{Measured topological graphs}
\label{measured_graphs}
A \emph{measured topological graph} is a topological graph $Z = (r, s \colon Y \to X)$ together with an $r$-operator $\Lambda \colon \mathcal{D}\to\mathcal{R}$ satisfying the additional property $\mathcal{D}s^{*}(\mathcal{R})\subseteq \mathcal{D}$. We shall denote a measured topological graph by $(Z, \Lambda)$. 

\blemma
\label{lem:MinDef}
Given a measured topological graph $(Z,\Lambda)$, there always exist the smallest domain and range satisfying $\Lambda(\mathcal{D})\subseteq\mathcal{R}$, $\mathcal{D}r^{*}(\mathcal{R})\subseteq \mathcal{D}$ and $\mathcal{D}s^{*}(\mathcal{R})\subseteq \mathcal{D}$. We denote them by $\cD_{\min}$ and $\cR_{\min}$ and call them the domain and range of minimal definition of $(Z,\Lambda)$. We write $\Lambda_{\min} \colon \cD_{\min} \to \cR_{\min}$ for the restriction of $\Lambda$ to $\cD_{\min}$.
\elemma
\setlength{\parindent}{0cm} \setlength{\parskip}{0cm}

\bproof
We construct $\cD_{\min}$ and $\cR_{\min}$ recursively. Let $\cF_0 \coloneq C_c(Y)$ and $\cE_0 \coloneq C_c(X)$, and for $n\geq 1$ we define 
\begin{align*}
  \cF_n &\coloneqq \overline{{}^*\text{-}\alg}\big(C_c(Y), C_c(Y) r^* \Lambda (\cF_{n-1}), C_c(Y) s^* \Lambda (\cF_{n-1})\big) \subseteq \mathcal{D};\\
  \cE_n &\coloneqq \overline{{}^*\text{-}\alg}\big(C_c(X), \Lambda(\cF_{n-1})\big) \subseteq\mathcal{R},
\end{align*}
where the closures are in the inductive limit topology. Then $\mathcal{D}_{\text{min}}$ and $\mathcal{R}_{\text{min}}$ are the closures of $\bigcup_{n}\mathcal{F}_{n}$ and $\bigcup_{n}\mathcal{E}_{n}$ in the inductive limit topologies.
\eproof
\setlength{\parindent}{0cm} \setlength{\parskip}{0.5cm}

We will say a measured topological graph is
\setlength{\parindent}{0cm} \setlength{\parskip}{0cm}

\begin{enumerate}
\item \emph{minimal} if $\mathcal{D} = \mathcal{D}_{\text{min}}$ and $\mathcal{R}= \mathcal{R}_{\text{min}}$,
\item \emph{maximal} if $\mathcal{D} = \ell^{\infty}_{c}(Y)$ and $\mathcal{R}= \ell_{c}^{\infty}(X)$,
\item \emph{continuous} if $\mathcal{D} = C_{c}(Y)$ and $\mathcal{R}= C_{c}(X)$,
\item \emph{Baire} if $\mathcal{D} = B_{c}(Y)$, $\mathcal{R}= B_{c}(X)$ and the $r$-operator is defined by a system of Baire measures $\{\lambda_{x}\}_{x\in X}$ on $Y$ such that $\text{supp}(\lambda_{x})\subseteq r^{-1}(x)$ and $\Lambda(f)(x) = \int f {\rm d} \lambda_x$ for all $x\in X$, 
\item \emph{weighted} if $\mathcal{D} = \ell^{\infty}_{c}(Y)$, $\mathcal{R}= \ell_{c}^{\infty}(X)$ and the $r$-operator is defined by a weight $w \colon Y\to[0,\infty)$ as in Example~\ref{weights}, and
\item \emph{strictly positively weighted} if it is weighted and the corresponding weight is strictly positive.
\end{enumerate}
We shall denote a weighted topological graph by $(Z,w)$. We will often consider the case where $w$ takes the constant value $1$, i.e., $w=1$.
\setlength{\parindent}{0cm} \setlength{\parskip}{0.5cm}

Let $(Z,\Lambda)$ and $(Z',\Lambda')$ be measured topological graphs. A \emph{measured} graph morphism from $(Z,\Lambda)$ to $(Z',\Lambda')$ is a graph morphism $\zeta \colon Z \to Z'$ such that $\zeta^{*}_{Y}(\mathcal{D}')\mathcal{D}\subseteq \mathcal{D}$, $\zeta_{X}^{*}(\mathcal{R}')\mathcal{R}\subseteq\mathcal{R}$ and $\Lambda\circ \zeta^{*}_{Y} = \zeta^{*}_{X}\circ\Lambda'$. The left hand side here is defined as the limit, in $M(\cR)$, of $( \Lambda((f \circ \zeta_Y) e_i) )_i$ for positive $f\in \cR$, where $(e_{i})\subseteq C_{c}(Y)$ is an approximate unit (i.e., a net such that $\lim_i e_i g = g$ holds in $C_c(Y)$ for all $g \in C_c(Y)$). We shall denote a measured graph morphism by $\zeta \colon (Z,\Lambda)\to (Z',\Lambda')$. To aid in brevity, we shall refer to a measured graph morphism as simply a ``morphism'' if the notation $\zeta \colon (Z,\Lambda) \to (Z',\Lambda')$ is displayed. Let us see what it means to be a measured graph morphism in some special cases.

Let $(Z,\Lambda)$ and $(Z', \Lambda)$ be two $\sigma$-compact Baire measured topological graphs, and suppose $\zeta \colon Z\to Z'$ is a graph morphism. Then $\zeta$ is a morphism  $\zeta \colon (Z,\Lambda) \to (Z',\Lambda')$ if and only if $\zeta$ is Baire measureable, and $(\zeta_{Y})_{*}\lambda_{x} = \lambda'_{\zeta_{X}(x)}$, for all $x\in X$. This characterization of morphisms holds true if, instead of assuming $\sigma$-compactness of the graphs, we assume that the pre-images of pre-compact sets under $\zeta_{Y}$ and $\zeta_{X}$ are pre-compact.

We consider the general case. A set $B\subseteq Y$ is \emph{locally} Baire measurable if $B\cap K$ is Baire measurable, for every compact $G_{\delta}$ set $K$. For a measure $\mu$, and $B$ a measurable set, denote by $\mu|_{B}$ the measure $A\mapsto \mu(A\cap B)$. In general, a graph morphism $\zeta \colon Z \to Z'$ between two Baire measured topological graphs $(Z,\Lambda)$ and $(Z',\Lambda')$ is measured if and only if it is locally Baire measurable and $\lim_{K}(\zeta_{Y})_{*}(\lambda_{x}|_{K}) = \lambda'_{\zeta_{X}(x)}$, for all $x\in X$, where $K$ ranges over all compact $G_{\delta}$ sets.

A graph morphism $\zeta \colon Z\to Z'$ between two weighted topological graphs $(Z,w)$ and $(Z',w')$ is measured if and only if 
\[\sum_{y\in r^{-1}(x) \colon \zeta_{Y}(y) = y'}w(y) = w'(y')\text{ for all }y'\in Y'\text{ and }x\in X\text{ satisfying }\zeta_{X}(x) = r(y').\]
In the special case where $Z$ and $Z'$ are weighted by $w=1$, the above equation simply says that $\zeta_{Y} \colon  r^{-1}(x)\to r^{-1}(\zeta_{X}(x))$ is a bijection, for all $x\in X$.

If we take any measured topological graph $(Z,\Lambda)$, then $\text{id}_{Z} \colon (Z,\Lambda)\to (Z,\Lambda_{\text{min}})$ is a morphism.

\begin{lemma}\label{morphism_restriction}
Suppose $(Z,\Lambda)$, $(Z',\Lambda')$ are measured topological graphs and $\zeta \colon (Z,\Lambda)\to (Z',\Lambda')$ is a morphism. If $\zeta$ is continuous, then the morphism we denote by $\zeta_{\min} \colon  (Z,\Lambda_{\min})\to (Z',\Lambda'_{\min})$ induced by $\zeta$ is a continuous morphism.
\end{lemma}
\setlength{\parindent}{0cm} \setlength{\parskip}{0cm}

\begin{proof}
This follows from induction on the $^*$-sub-algebras $\mathcal{F}'_{n}$, $\mathcal{E}'_{n}$. For the induction step, observe that we have $\Lambda(f\circ\zeta_{Y}) = \Lambda'(f)\circ\zeta_{X}\in M(\mathcal{R})$ for all $0\leq f\in \mathcal{F}'_{n}$. Recall that the left hand side here is defined as the limit, in $M(\cR)$, of $( \Lambda((f \circ \zeta_Y) e_i) )_i$, where $(e_{i})\subseteq C_{c}(Y)$ is an approximate unit. Now Dini's lemma (for nets) implies that, for an approximate unit $(e_{i})$ as before, the net $\Lambda((f\circ\zeta_{Y}) e_{i})\phi\in\mathcal{E}_{n+1}$ converges uniformly to $(\Lambda'(f)\circ\zeta_{X})\phi$ for every $\phi\in \mathcal{E}_{n+1}$ and hence $(\Lambda'(f)\circ\zeta_{X})\phi\in\mathcal{E}_{n+1}$. 
\end{proof}
\setlength{\parindent}{0cm} \setlength{\parskip}{0.5cm}

\begin{remark}\label{rmk:functorial}
The assignment $(Z,\Lambda)\mapsto (Z,\Lambda_{\text{min}})$, $\zeta \mapsto \zeta_{\min}$ is a functor. 
In fact, it is a retraction of the category of measured topological graphs with continuous morphisms onto the sub-category of minimally measured topological graphs with continuous morphisms.
\end{remark}

\subsection{Covers of topological graphs}
\label{ss:cover_top_graph}
Let us now explain how to describe a measured topological graph equivalently in terms of continuous data, by which we mean the notion of covers as in the following definition. 
\bdefin
Let $Z$ be a topological graph. 
A \emph{continuously measured cover} (or simply a \emph{cover}) of a topological graph $Z$ is a continuously measured topological graph $(\mathcal{Z},\bm{\lambda})$, say $\cZ = (\mfr, \mfs \colon \cY \to \cX)$ and $\bm{\lambda}$ is given by a system of Radon measures $\gekl{\lambda_x}_{x \in \cX}$ on $\cY$, together with a continuous and proper morphism $\pi \colon \mathcal{Z}\to Z$ and a (not necessarily continuous) morphism $\iota \colon Z\to \mathcal{Z}$ with dense image such that $\pi\circ\iota = \id_{Z}$ and $\text{supp}(\lambda_{\iota_{X}(x)})\subseteq \overline{\iota_{Y}(r^{-1}(x))}$ for all $x\in X$. It follows that $\iota$ is an injection and $\pi$ is a surjection. We will denote a cover by $(\mathcal{Z},\bm{\lambda})_{Z}$.
\edefin

\begin{remark}
We denote the $r$-operator in the cover with a lower script Greek letter to emphasize its connections with the associated system of Radon measures. Also, we think of it as having a reduced role in the structure compared to a measured topological graph.
\end{remark}

A \emph{cover morphism} $\mfz_{\zeta} \colon (\mathcal{Z},\bm{\lambda})_{Z}\to (\mathcal{Z}',\bm{\lambda}')_{Z'}$ is a continuous morphism $\mfz \colon (\mathcal{Z},\bm{\lambda})\to (\mathcal{Z}',\bm{\lambda}')$ together with a (not necessarily continuous) morphism $\zeta \colon Z\to Z'$ satisfying $\mfz \circ \iota = \iota' \circ \zeta$.
If $\zeta \colon Z\to Z'$ is continuous, we additionally have $\pi' \circ \mfz = \zeta \circ \pi$. 
In this case, we shall call $\mfz \colon (\mathcal{Z},\bm{\lambda})_{Z} \to (\mathcal{Z}',\bm{\lambda}')_{Z'}$ \emph{continuous}.

We now set out to describe an equivalence of categories between measured topological graphs and covers in the sense above. 
Given a cover $(\mathcal{Z},\bm{\lambda})_{Z}$, let $\mathcal{D} \coloneq \iota_{Y}^{*}(C_{c}(Y))$, $\mathcal{R} \coloneq \iota_{X}^{*}(C_{c}(X))$ and set $\Lambda \coloneq  (\iota_{X})^{*}\circ \lambda\circ (\iota^{*}_{Y})^{-1}$. Conversely, given a measured topological graph $(Z,\Lambda)$, let $\cY \coloneq \Spec(\mathcal{D}) $, $\cX \coloneq \Spec(\mathcal{R})$, let the Gelfand dual of $r^{*} \colon \mathcal{D}\to M(\mathcal{R})$ be $\mfr \colon \cY \to \cX$, the Gelfand dual of $s^{*} \colon \mathcal{D}\to M(\mathcal{R})$ be $\mfs \colon \cY \to \cX$ and set $\bm{\lambda} \coloneq \Lambda \colon C_{c}(\cY) \to C_{c}(\cX)$. The Gelfand duals of the inclusions $C_{c}(Y)\subseteq D$ and $C_{c}(X)\subseteq\mathcal{R}$ yield $\pi_{\cY} \colon \cY \to Y$ and $\pi_{\cX} \colon \cX \to X$. Moreover, there are (not necessarily continuous) mappings $\iota_{X} \colon X\to \cX$, $\iota_{Y} \colon Y \to \cY$ given by sending $x\in X$ or $y\in Y$ to the character $\text{ev}_{x} \colon D\to \mathbb{C}$ or $\text{ev}_{y} \colon R\to\mathbb{C}$. Note that we are using Lemma~\ref{cc} to identify $\mathcal{D}$ with $C_{c}(\cY)$ and $\mathcal{R}$ with $C_{c}(\cX)$.

At the level of morphisms, a morphism $\zeta \colon (Z,\Lambda)\to (Z',\Lambda')$ induces a morphism $\mfz_{\zeta} \colon (\mathcal{Z},\lambda)_{Z}\to (\mathcal{Z}',\lambda')_{Z'}$ by Gelfand duality applied to $\zeta_{Y}^{*} \colon \mathcal{D}'\to M(\bar{D})$ and $\zeta^{*}_{X} \colon \mathcal{R}'\to M(\bar{R})$. If $\mfz_{\zeta} \colon (\mathcal{Z},\lambda)_{Z}\to (\mathcal{Z}',\lambda')_{Z'}$ is a morphism of covers, then $\zeta \colon (Z,\Lambda)\to (Z',\Lambda')$ is a morphism.

\begin{theorem}\label{thm:correspondence}
The correspondence $(Z,\Lambda)\leftrightarrow (\mathcal{Z}, \lambda)_{Z}$ and $\zeta \leftrightarrow \mfz_{\zeta}$ is an equivalence of categories. Continuous morphisms are sent to continuous morphisms. Moreover, $(Z,\Lambda)$ is bounded below if and only if $(\mathcal{Z},\lambda)$ is bounded below.
\end{theorem}
\setlength{\parindent}{0cm} \setlength{\parskip}{0cm}

\bproof
Given a cover $(\mathcal{Z},\lambda)_{Z}$, it is easy to see that $\mathcal{D}$ as defined above is a domain, $\mathcal{R}$ as defined above is a range, and that $\mathcal{D}r^{*}(\mathcal{R})\subseteq \mathcal{D}$, $\mathcal{D}s^{*}(\mathcal{R})\subseteq \mathcal{D}$. From $\text{supp}(\lambda_{\iota_{X}(x)})\subseteq \overline{\iota_{Y}(r^{-1}(x))}$ for all $x\in X$, it follows that the operator $\Lambda = (\iota_{X})^{*}\circ \lambda\circ (\iota^{*}_{Y})^{-1}$ is an $r$-operator $\Lambda \colon \mathcal{D}\to\mathcal{R}$ and hence $(Z,\Lambda)$ is a measured topological graph.
\setlength{\parindent}{0cm} \setlength{\parskip}{0.5cm}

Conversely, given a measured topological graph $(Z,\Lambda)$, it is easy to see that $\pi_{\cX}$ and $\pi_{\cY}$ as defined above are continuous surjections satisfying $\pi_{\cX}\circ \mfr = r \circ \pi_{\cY}$ as well as $\pi_{\cX} \circ \mfs = s \circ \pi_{\cY}$. Moreover, $\bm{\lambda}$ as defined above is a positive linear operator $C_{c}(\cY) \to C_{c}(\cX)$ satisfying $\bm{\lambda}(f(g\circ \tilde{r})) = \bm{\lambda}(f)g$ for all $f\in C_{c}(\cY)$ and $g\in C_{c}(\cX)$. It thus follows from Lemma~\ref{charecterization} that $\bm{\lambda}$ is an $\mfr$-system. In addition, it is easy to see that $\iota_X$ and $\iota_Y$ as constructed above have dense images and satisfy the relations $\pi_{\cY}\circ\iota_{Y} = \text{id}_{Y}$, $\pi_{\cX}\circ \iota_{X} = \text{id}_{X}$ and $\iota_{X}^{*}\circ \bm{\lambda} = \Lambda \circ \iota_{Y}^{*}$.

Lemma~\ref{cc} implies that, up to canonical isomorphisms, our constructions are inverse to each other.
\eproof
\setlength{\parindent}{0cm} \setlength{\parskip}{0.5cm}

\subsection{Applications}
\label{sec:app_min_cover}
The forgetful functor $(\mathcal{Z},\bm{\lambda})_{Z}\mapsto (\mathcal{Z},\bm{\lambda})$, $\mfz_{\zeta} \mapsto \zeta$, goes from the category of measured topological graphs (with continuous morphisms) to the category of continuously measured topological graphs (also with continuous morphisms).
We will specifically be interested in the functor 
\[
  (Z,\Lambda) \mapsto (Z,\Lambda_{\min}) \mapsto (\mathcal{Z},\bm{\lambda})_Z \mapsto (\mathcal{Z},\bm{\lambda}) \eqcolon (\hat{Z}, \hat{\bm{\lambda}}),
\]
where the first arrow is given by restriction to domain and range of minimal definition (see Lemmas~\ref{lem:MinDef} and \ref{morphism_restriction} as well as Remark~\ref{rmk:functorial}), the second arrow is given by the correspondence in Theorem~\ref{thm:correspondence} and the last arrow is the forgetful functor. 
\bdefin
We call $(\hat{Z}, \hat{\bm{\lambda}})$ the \emph{minimal cover} of $(Z,\Lambda)$. We write $\hat{Z} = (\hat{r}, \hat{s} \colon \hat{Y} \to \hat{X})$ and denote the structure maps by $\pi \colon \hat{Z} \to Z$ the continuous, proper morphism and by $\iota \colon Z \to \hat{Z}$ the (not necessarily continuous) morphism.
\edefin
This cover satisfies a universal property which we explain now. First, observe that by (the proof of) the category correspondence in Theorem~\ref{thm:correspondence}, $\iota \colon Z\to\hat{Z}$ is actually a morphism $\iota \colon (Z,\Lambda_{\text{min}})\to (\hat{Z},\hat{\bm{\lambda}})$. Therefore, $\iota \colon (Z,\Lambda)\to (\hat{Z},\hat{\bm{\lambda}})$ is a morphism.
\begin{theorem}
\label{thm:mincover}
Suppose that $(\dot{Z},\dot{\bm{\lambda}})$ is another continuously measured topological graph that covers $(Z,\Lambda)$ in the sense that there is a continuous proper morphism $\dot{\pi} \colon \dot{Z}\to Z$ and a morphism $\dot{\iota} \colon (Z,\Lambda)\to (\dot{Z},\dot{\bm{\lambda}})$ with dense image such that $\dot{\pi}\circ\dot{\iota} = \text{id}_{Z}.$ Then there is a unique continuous proper morphism $\zeta \colon (\dot{Z},\dot{\bm{\lambda}})\to (\hat{Z},\hat{\bm{\lambda}})$ such that $\pi\circ \zeta = \dot{\pi}$ and $\zeta \circ \dot{\iota} = \iota$.
\end{theorem}
\setlength{\parindent}{0cm} \setlength{\parskip}{0cm}

\begin{proof}
Since $\dot{\iota} \colon (Z,\Lambda)\to (\dot{Z},\dot{\bm{\lambda}})$ is a morphism, it follows that $(\dot{Z},\dot{\bm{\lambda}})_{Z}$ is a cover of $Z$. Since $\dot{\iota}^{*}_{X}\circ \dot{\bm{\lambda}} = \Lambda\circ\dot{\iota}^{*}_{Y}$, its corresponding measured topological graph is of the form $\Lambda' \colon \mathcal{D}'\to \mathcal{R}'$ given by the restriction of $\Lambda$, where $\mathcal{R}'\subseteq\mathcal{R}$ and $\mathcal{D}'\subseteq \mathcal{D}$. Hence $\mathcal{D}_{\text{min}}\subseteq \mathcal{D}'$ and $\mathcal{R}_{\text{min}}\subseteq\mathcal{R}'$, so that there is a morphism $\text{id}_{Z} \colon  (Z,\Lambda')\to (Z,\Lambda_{\text{min}})$. Apply the cover functor from Theorem~\ref{thm:correspondence} and the forgetful functor to this morphism to obtain a morphism $\zeta \colon (\dot{Z},\dot{\bm{\lambda}})\to (\hat{Z},\hat{\bm{\lambda}})$ satisfying $\pi\circ \zeta = \text{id}_{Z}\circ\dot{\pi} = \dot{\pi}$ and $\zeta \circ \dot{\iota} = \text{id}_{Z}\circ \iota = \iota$.
\end{proof}
\setlength{\parindent}{0cm} \setlength{\parskip}{0.5cm}

Let us consider some special cases of this theorem. 

We first consider Baire measured topological graphs. Suppose $\dot{\iota} \colon Z\to \dot{Z}$ is a morphism which satisfies $\dot{\pi}\circ\dot{\iota} = \text{id}_{Z}$ for some continuous morphism $\dot{\pi} \colon \dot{Z}\to Z$. Then the inverse image of pre-compact sets by $\dot{\iota}$ are pre-compact. It follows that $\dot{\iota} \colon (Z,\Lambda)\to (\dot{Z},\dot{\bm{\lambda}})$ from a Baire measured graph into a continuously measured topological graph is a morphism if and only if $\dot{\iota}$ is Baire measurable and the measures $(\dot{\iota}_{Y})_{*}\lambda_{x}$, $\dot{\lambda}_{\iota_{X}(x)}$ agree on $C_{c}(\dot{Y})$, for all $x\in X$. Since every characteristic function of a compact $G_{\delta}$ set is a sequential limit of continuous and compactly supported functions, it follows that $(\dot{\iota}_{Y})_{*}\lambda_{x}=\dot{\lambda}_{\iota_{X}(x)}$ for all $x\in X$.

\begin{corollary}
Suppose $(\dot{Z},\dot{\bm{\lambda}})$ is a continuously measured topological graph that covers the Baire measured graph $(Z,\Lambda)$ in the sense that there is a continuous and proper morphism $\dot{\pi} \colon \dot{Z}\to Z$ and a Baire measurable morphism $\dot{\iota} \colon Z\to \dot{Z}$ with dense image satisfying $\dot{\pi}\circ\dot{\iota} = \text{id}_{Z}$ and $ (\dot{\iota}_{Y})_{*}\lambda_{x}=\dot{\lambda}_{\dot{\iota}_{X}(x)}$ for all $x\in X$. Then there is a unique continuous and proper morphism $\zeta \colon \dot{Z}\to \hat{Z}$ satisfying $\pi\circ\zeta = \dot{\pi}$, $\zeta\circ\dot{\iota} = \iota$ and $(\zeta_{Y})_{*}\dot{\lambda}_{\dot{x}} = \hat{\lambda}_{\zeta_{X}(\dot{x})}$ for all $\dot{x}\in \dot{X}$.
\end{corollary}
Let us see what the universal property corresponds to for strictly positively weighted topological graphs. Note that by Theorem~\ref{thm:correspondence}, any cover of a strictly positively weighted topological graph is also strictly positively weighted.
\begin{corollary}        Suppose $(\dot{Z},\dot{w})$ is a continuously weighted topological graph that covers the strictly positively weighted graph $(Z,w)$ in the sense that there is a continuous and proper morphism $\dot{\pi} \colon \dot{Z}\to Z$ and a morphism $\dot{\iota} \colon Z\to \dot{Z}$ with dense image satisfying $\dot{\pi}\circ\dot{\iota} = \text{id}_{Z}$ and
    \[\sum_{y\in r^{-1}(x) \colon \dot{\iota}_{Y}(y) = \dot{y}}w(y) = \dot{w}(\dot{y})\text{ for all }\dot{y}\in \dot{Y}\text{ and }x\in X\text{ satisfying }\dot{\iota}_{X}(x) = r(\dot{y}).\]
    Then there is a unique continuous and proper morphism $\zeta \colon \dot{Z}\to \hat{Z}$ satisfying $\pi \circ \zeta = \dot{\pi}$, $\zeta \circ\dot{\iota} = \iota$ and 
    \[\sum_{\dot{y}\in r^{-1}(\dot{x}) \colon \zeta_{\dot{Y}}(\dot{y}) = \hat{y}}\dot{w}(y) = \hat{w}(\hat{y})\text{ for all }\hat{y}\in \hat{Y}\text{ and }x\in \dot{X}\text{ satisfying } \zeta_{X}(\dot{x}) = r(\hat{y}).\]
\end{corollary}

\begin{example}\label{example:local_injections}
  The range map of a strictly positively weighted topological graph is locally injective and is weighted by $w =1$ if and only if the $r$-operator satisfies the following property: 
  for every $f\in C_c(Y)$, there is a closed ideal $\mathcal{I}$ such that $f\cdot\mathcal{I}\neq 0$ and $\Lambda \colon \mathcal{I} \to \ell^{\infty}_c(X)$ is a $^*$-homomorphism. It follows from this characterization, Lemma \ref{cc}, and Lemma \ref{lem:full} that the range map of the cover associated to such a topological graph is a local homeomorphism, and is also weighted by $w =1$. Such graphs and their cover form an important example class for this paper.
\end{example}
If $(Z,w)$ is a weighted topological graph that is not strictly positively weighted, then the $r$-system from the minimal cover is not necessarily induced by a weight.
Therefore, we consider them in the category of maximally measured topological graphs. Let us investigate the universal property in this category. For $x\in X$, let $\Lambda_{x}  \coloneq  \Lambda(-)(x)$.

\begin{corollary}
Suppose $(\dot{Z},\dot{\bm{\lambda}})$ is a continuously measured topological graph that covers the maximally measured graph $(Z,\Lambda)$ in the sense that there is a continuous and proper morphism $\dot{\pi} \colon \dot{Z}\to Z$ and a morphism $\dot{\iota} \colon Z\to \dot{Z}$ with dense image such that $\dot{\pi}\circ\dot{\iota} = \text{id}_{Z}$ and $ (\dot{\iota}_{Y})_{*}\Lambda_{x}$ is an extension of $\dot{\lambda}_{\dot{\iota}_{X}(x)}$, for all $x\in X$. Then there is a unique continuous and proper morphism $\zeta \colon \dot{Z}\to \hat{Z}$ satisfying $\pi\circ\zeta = \dot{\pi}$, $\zeta\circ\dot{\iota} = \iota$ and $(\zeta_{Y})_{*}\dot{\lambda}_{\dot{x}} = \hat{\lambda}_{\zeta_{X}(\dot{x})}$ for all $\dot{x}\in \dot{X}$.
\end{corollary}

\begin{remark}
The difference in these three corollaries highlights the importance of the choice of ambient domain and range $^*$-sub-algebras.
\end{remark}

\subsection{Description of the covers}
\label{ss:DescCovers}

Recall the terminology from the proof of Lemma \ref{lem:MinDef}.
Given a measured topological graph $(Z,\Lambda)$, let $X_{n} = \Spec(\mathcal{E}_{n})$ and $Y_{n} = \Spec(\mathcal{F}_{n})$. By Gelfand duality, there are maps $X_{n}\to X_{n-1}$, $Y_{n}\to Y_{n-1}$ such that $\hat{X}$ and $\hat{Y}$ are naturally identified with the projective limits $\plim_{n} X_{n}$ and $\plim_{n}Y_{n}$. Let us now describe $X_{n}$ and $Y_{n}$ in terms of $X_{n-1}$ and $Y_{n-1}$. Let $\mathcal{M}(Y_{n-1})$ be the space of Radon measures on $Y_{n-1}$, which we identify with the positive linear functionals on $C_{c}(Y_{n-1})$ equipped with the weak* topology.

Let $\iota_{n, X} \colon X\to X_{n}$ and $\iota_{n,Y} \colon Y\to Y_{n}$ be the canonical inclusions as characters of $\mathcal{E}_{n}$, $\mathcal{F}_{n}$. For $x\in X$, let $x_{n}  \coloneq  \iota_{n,X}(x)$, and $\lambda_{n-1, x_{n}}  \coloneq  \Lambda(-)(x) \colon \mathcal{F}_{n-1}\to \mathbb{C}$. 

\bprop
\label{prop:Xn=XxM}
$X_n$ is (homeomorphic to) the closure of the image of the map
\[
 X \to X \times \mathcal{M}(Y_{n-1}), \, x \ma (x, \lambda_{n-1,x_n}).
\]
\eprop
\setlength{\parindent}{0cm} \setlength{\parskip}{0cm}

\bproof
For the purpose of the proof, we call the closure of the image $\bar{X}$. We start by defining a map $\bar{X} \to X_n$: given $\bar{x} = (x,\bar{\lambda}) = \lim_{i} (x_i,\lambda_{n-1,(x_{i})_n})$, define a character $\chi \colon \cE_n \to \Cz$ by sending $f \in C_c(X)$ to $f(x) = \lim f(x_i)$ and $g = \Lambda(h) \in \Lambda(\cF_{n-1})$, for $h \in \cF_{n-1}$, to $\lim g((x_{i})_n) = \lim \lambda_{n-1,(x_i)_n}(h) = \bar{\lambda}(h)$. By construction, $\chi = \lim \ev_{\iota_{n,X}(x_i)}$, so that our character extends continuously to a character (again denoted by $\chi$) on $\mathcal{E}_n$. By construction, the map sending $\bar{x}\mapsto \chi(g)$ is continuous for any generator $g\in C_{c}(X)\cup \Lambda(\cF_{n-1})$. Therefore, our map $\bar{x}\mapsto \chi$ is continuous.
\setlength{\parindent}{0cm} \setlength{\parskip}{0.5cm}

Now we construct a map $X_n \to \bar{X}$: given $\chi \in X_n = \Spec(\mathcal{E}_n)$, the restriction $\chi \vert_{C_c(X)}$ is non-zero as $\Lambda(\cF_{n-1}) \subseteq C_c(X) \Lambda(\cF_{n-1})$. Hence $\chi \vert_{C_c(X)} = \ev_x$ for some $x \in X$. Moreover, the mapping $\chi\mapsto (x, \bar{\lambda})$, where $\bar{\lambda}  \coloneq  \chi\circ\Lambda|_{\mathcal{F}_{n-1}}\in\mathcal{M}(Y_{n-1})$ is continuous, and its image lies in $\bar{X}$ because $X$ is dense in $X_n$ and our map sends $x \in X$ to $(x,\lambda_{n-1,x_n})$.

It is easy to see these two maps are inverse to each other. 
\eproof
\setlength{\parindent}{0cm} \setlength{\parskip}{0.5cm}

As before, in the following, we use the notation $y_n \coloneq \iota_{n,Y}(y)$ for $y \in Y$.
\bprop
\label{prop:Yn=YxMxM}
$Y_n$ is (homeomorphic to) the closure of the image of the map
\[
 Y \to Y \times \mathcal{M}(Y_{n-1}) \times \mathcal{M}(Y_{n-1}), \, y \ma (y, \lambda_{r(y)_n}, \lambda_{s(y)_n}).
\]
\eprop
\setlength{\parindent}{0cm} \setlength{\parskip}{0cm}

\bproof
For the purpose of the proof, let us call the closure of the image $\bar{Y}$. 
We start by defining a map $\bar{Y} \to Y_n$: given $\bar{y} = (y,\bar{\lambda}_r,\bar{\lambda}_s) = \lim_{i} (y_i,\lambda_{r(y_{i})_{n}},\lambda_{s(y_i)_{n}})$, define a character $\chi \colon \cF_n \to \Cz$ by sending $f \in C_c(Y)$ to $f(y) = \lim_{i} f(y_i)$, $g = a r^*\Lambda(b) \in C_c(Y) r^*\Lambda(\cF_{n-1})$, for $a \in C_c(Y)$ and $b \in \cF_{n-1}$, to \[\text{lim}_{i} g((y_i)_{n}) = \text{lim}_{i} a(y_i) \lambda_{r(y_{i})_{n}}(b) = a(y) \bar{\lambda}_r(b)\] and $g' = a' s^*\Lambda(b') \in C_c(Y) s^*\Lambda(\cF_{n-1})$, for $a' \in C_c(Y)$ and $b' \in \cF_{n-1}$, to \[\text{lim}_{i} g'((y_i)_{n}) = \text{lim}_{i} a'(y_i) \lambda_{s(y_i)_{n}}(b') = a'(y) \bar{\lambda}_s(b').\] By construction, $\chi = \lim \ev_{\iota_{n,Y}(y_i)}$, so that our character extends continuously to a character (again denoted by $\chi$) on $\mathcal{F}_n$. It is easy to see our map sending $\bar{y}$ to $\chi$ is continuous.
\setlength{\parindent}{0cm} \setlength{\parskip}{0.5cm}

Now we construct a map $Y_n \to \bar{Y}$: given $\chi \in Y_n = \Spec(\mathcal{F}_n)$, the restriction $\chi \vert_{C_c(Y)}$ is non-zero as $\mathcal{F}_{n} = C_c(Y)\mathcal{F}_{n}$. Hence $\chi \vert_{C_c(Y)} = \ev_y$ for some $y \in Y $. Moreover, the composition
\[
 \mathcal{F}_{n-1} \to \mathcal{E}_n \to M(\mathcal{F}_{n}) \cong C(Y_n) \to \Cz,
\]
where $M(\mathcal{F}_n)$ is the multiplier algebra of $\mathcal{F}_n$, the first map is given by $\Lambda$, the second map is given by $r^*$, and the last map is given by $\chi$, defines $\bar{\lambda}_r \in \mathcal{M}(Y_{n-1})$. Similarly, the composition
\[
 \mathcal{F}_{n-1} \to \mathcal{E}_n \to M(\mathcal{F}_n) \cong C(Y_n) \to \Cz,
\]
where the first map is given by $\Lambda$, the second map is given by $s^*$, and the last map is given by $\chi$, defines $\bar{\lambda}_s \in \mathcal{M}(Y_{n-1})$. Define our map $Y_n \to \bar{Y}$ by setting $\chi \ma (y,\bar{\lambda}_r,\bar{\lambda}_s)$. This is continuous. The image lies in $\bar{Y}$ because $Y$ is dense in $Y_n$ and our map sends $y \in Y$ to $(y,\lambda_{r(y)_n},\lambda_{s(y)_n})$.

By construction, these two maps are inverse to each other. 
\eproof
\setlength{\parindent}{0cm} \setlength{\parskip}{0.5cm}


\subsection{Essential covers}
\label{ss:Ess}
We now define covers of topological graphs with better regularity properties and that are easier to compute.

\bdefin
Let $Z = (r,s \colon Y\to X)$ be a topological graph. We say a subset $\crX \subseteq X$ is \emph{backwards invariant} if $s(y)\in \crX$ whenever $r(y)\in \crX$ for all $y \in Y$, i.e., $s(r^{-1}(\crX)) \subseteq \crX$. Then $\crY = r^{-1}(\crX)$ and $Z|_{\crX}  \coloneq  (r,s \colon \crY \to \crX)$ defines a sub-graph of $Z$. We call such a sub-graph arising in this way \emph{invariant}. 
\edefin
Recall that a subset $\crX$ of a topological space $X$ is \emph{co-meagre} if its complement is meagre, i.e., a countable union of sets whose closure have empty interior. In case $X$ is locally compact and Hausdorff (more generally a Baire space), $\crX \subseteq X$ is co-meagre if and only if it contains a dense $G_{\delta}$ set.

\bdefin
An \emph{essential cover} of $Z$ is a continuously measured graph $(\mathcal{Z},\bm{\lambda})$ together with a proper and continuous surjection $\pi \colon \mathcal{Z}\to Z$ with the property that there is an invariant sub-graph $Z \vert_{\crX}$ of $Z$ such that $Z \vert_{\crX}$ and the restriction $\mathcal{Z} \vert_{\pi^{-1}(\crX)}$ to the (automatically invariant) set $\pi^{-1}(\crX)$ are co-meagre, and the restriction of $\pi$ gives rise to a bijection $\mathcal{Z} \vert_{\pi^{-1}(\crX)} \to Z \vert_{\crX}$.
\edefin

Our first example of essential covers comes from a restriction of the cover $(\cZ,\bm{\lambda})_Z$ of a measured topological graph $(Z,\Lambda)$ from the correspondence of Theorem~\ref{thm:correspondence}. 

\bdefin
We call a measured topological graph $(Z,\Lambda)$ \emph{essentially continuous} if there is a co-meagre invariant sub-graph $Z \vert_{\crX} \subseteq Z$ such that $\iota  \colon Z \vert_{\crX}\to\mathcal{Z}$ is continuous. In that case, let $\mathcal{Z}_{\ess} \coloneq \overline{\iota(Z \vert_{\crX})}$ (with edge space $\cY_{\ess}$ and vertex space $\cX_{\ess}$), let $\pi_{\ess} \colon \mathcal{Z}_{\ess}\to Z$ be the restriction of $\pi$ to $\cZ_{\ess}$, and let $\bm{\lambda}_{\ess} \colon C_{c}(\mathcal{Y}_{\ess})\to C_{c}(\mathcal{X}_{\ess})$ be the restriction of $\bm{\lambda}$ to $C_{c}(\mathcal{Y}_{\ess})$.
\edefin

We need the following characterisation of the continuity points of $\iota \colon Z\to\mathcal{Z}$.

\begin{proposition}\label{continuity_points}
    Suppose that $\pi \colon \cW \to W$ is a continuous, proper surjection between locally compact Hausdorff spaces, $C$ is a subset of $W$, $\iota \colon C \to \cW$ is a map satisfying $\pi\circ\iota = \text{id}_{C}$ and that $\iota(C)$ is dense in $\cW$. Then $\iota \colon C\to \cW$ is continuous at $w \in C$ if and only if $\pi^{-1}(w) = \{\iota(w)\}$.
\end{proposition}
\setlength{\parindent}{0cm} \setlength{\parskip}{0cm}

\begin{proof}
    Suppose $|\pi^{-1}(w)| = 1$. By properness of $\pi$, for every open neighbourhood $\mathcal{U}\subseteq \cW$ of $\pi^{-1}(w) = \{\iota(w)\}$, there is an open neighbourhood $U\subseteq W$ of $w$ such that $\iota(U\cap C)\subseteq \pi^{-1}(U)\subseteq \mathcal{U}$. Hence, $\iota$ is continuous at $w$.
\setlength{\parindent}{0.5cm} \setlength{\parskip}{0cm}

    If $|\pi^{-1}(w)| > 1$, then let $\iota(w)\neq \omega \in \pi^{-1}(w)$. By density of $\iota(C)$, there is a net $(w_{i})$ in $C$ such that $(\iota(w_{i}))$ converges to $\omega$. Therefore, $(w_{i} = \pi(\iota(w_{i})))$ converges to $w$ but $\iota(w)\neq \omega$. Hence, $\iota$ is not continuous at $w$.
\end{proof}
\setlength{\parindent}{0cm} \setlength{\parskip}{0.5cm}

\bcor
Given a measured topological graph $(Z,\Lambda)$ which is essentially continuous,  $(\mathcal{Z}_{\ess},\bm{\lambda}_{\ess})$ is an essential cover. 
\ecor
\setlength{\parindent}{0cm} \setlength{\parskip}{0cm}

\bproof
Since $Z \vert_{\crX}$ is an invariant sub-graph and $\iota^{*}_{X}\circ \bm{\lambda} = \Lambda\circ \iota^{*}_{Y}$, $\mathcal{Z}_{\ess}$ is a closed sub-graph of $\mathcal{Z}$ such that $\supp(\lambda_{\chi})\subseteq r^{-1}(\chi)\cap \mathcal{Y}_{\ess}$, for all $\chi\in \mathcal{X}_{\ess}$. Hence the restriction $\bm{\lambda}_{\ess}$ of $\bm{\lambda}$ is an $r$-operator $C_{c}(\mathcal{Y}_{\ess})\to C_{c}(\mathcal{X}_{\ess})$. Moreover, density of $Z \vert_{\crX}$ implies that the restriction $\pi_{\ess} \colon \mathcal{Z}_{\ess}\to Z$ of $\pi$ remains a surjection. Since $\iota \colon  Z\to \mathcal{Z}$ is a section of $\pi$, Proposition~\ref{continuity_points} implies that $\iota(Z \vert_{\crX}) = \pi_{\ess}^{-1}(Z \vert_{\crX})$. This proves that $(\mathcal{Z}_{\ess}, \bm{\lambda}_{\ess})$ is an essential cover.
\eproof
\setlength{\parindent}{0cm} \setlength{\parskip}{0.5cm}

\bdefin
Given a measured topological graph $(Z,\Lambda)$ which is essentially continuous, we call $(\mathcal{Z}_{\ess},\bm{\lambda}_{\ess})$ \emph{the essential cover of} $(Z,\Lambda)$. 
\edefin

We consider an alternative construction of the essential cover through the Gelfand correspondence which explains the name and shows that the essential cover is independent of the choice of invariant sub-graph $Z|_{\crX}$.
\setlength{\parindent}{0.5cm} \setlength{\parskip}{0cm}

A topological measured graph $(Z,\Lambda)$ is essentially continuous if and only if there is a co-meagre invariant sub-graph $Z \vert_{\crX} = (r,s \colon \crY \to \crX)$ such that $f \vert_{\crX}$ is continuous for all $f \in \mathcal{R}$ and and $g \vert_{\crY}$ is continuous for all $g \in \mathcal{D}$. Given an essentially continuous measured graph $(Z,\Lambda)$, we define the \emph{essential ideals} $\mathcal{I}_{\mathcal{D}}\subseteq \mathcal{D}$, $\mathcal{I}_{\mathcal{R}}\subseteq \mathcal{R}$ to be the collection of functions that are zero on a co-meagre set. These ideals are closed by the Baire category theorem. We call the quotients $\mathcal{D}_{\ess} \coloneq  \mathcal{D}/\mathcal{I}_{\mathcal{D}}$ and $\mathcal{R}_{\ess} \coloneq  \mathcal{D}/\mathcal{I}_{\mathcal{D}}$ the \emph{essential domain} and the \emph{essential range}.

Using the Baire category theorem and continuity of functions in $\mathcal{D}$ on $\crY$, it follows that $f\in \mathcal{I}_{\mathcal{D}}$ if and only if $f \vert_{\crY} = 0$. Similarly, we obtain that $g\in\mathcal{I}_{\mathcal{R}}$ if and only if $g \vert_{\crX} = 0$. So, for $f\in\mathcal{I}_{\mathcal{D}}$, by invariance $(r^{-1}(\crX) = \crY)$, we have $\Lambda(f) \vert_{\crX} = \Lambda(f \vert_{\crY}) = 0.$ Hence, $\Lambda(\mathcal{I}_{\mathcal{D}})\subseteq \mathcal{I}_{\mathcal{R}}$. Therefore, $\Lambda$ induces a positive operator $\Lambda_{\ess} \colon \mathcal{D}_{\ess}\to \mathcal{R}_{\ess}$ with the property that $\Lambda_{\ess}(f)(x) = 0$ whenever $x\in \crX$ and $f\in\mathcal{D}_{\ess}$ satisfy $f \vert_{r^{-1}(x)} = 0$. The properties $g \circ r \vert_{\crY} = 0$ and $g \circ s \vert_{\crY} = 0$ for any $g\in\mathcal{I}_{\mathcal{R}}$ and the ideal characterization above implies the $^*$-homomorphisms $r^{*},s^{*} \colon \mathcal{R}\to M(\mathcal{D})$ descend to $^*$-homomorphisms $r^{*}, s^{*} \colon \mathcal{R}_{\ess}\to M(\mathcal{D}_{\ess})$.

Using the facts that, by Lemma~\ref{cc}, the Gelfand transforms $\hat{(-)} \colon \mathcal{D}\to C_{c}(\mathcal{Y})$ and $\hat{(-)} \colon \mathcal{R}\to C_{c}(\mathcal{X})$ are homeomorphisms (see also Theorem~\ref{thm:correspondence}, where $\cY$ and $\cX$ are defined), every closed ideal in $C_{c}(W)$ --- for $W$ a locally compact Hausdorff space --- is of the form $C_{c}(U)$ for some open subspace $U$ in $W$, and the above descriptions of the ideals $\mathcal{I}_{\mathcal{D}}$, $\mathcal{I}_{\mathcal{R}}$, it follows that $\hat{\mathcal{I}}_{\mathcal{D}} = C_{c}(\mathcal{X}\setminus \mathcal{X}_{\ess})$ and $\hat{\mathcal{I}}_{\mathcal{R}} = C_{c}(\mathcal{Y}\setminus \mathcal{Y}_{\ess})$, where $\mathcal{Y}_{\ess} =  \overline{\iota_{Y}(\crY)}$ and $\mathcal{X}_{\ess}  =  \overline{\iota_{X}(\crX)}$ as above, and we may identify $\Lambda_{\ess}$ with the restriction $\bm{\lambda}_{\ess} \colon C_{c}(\mathcal{Y}_{\ess})\to C_{c}(\mathcal{X}_{\ess})$ of $\bm{\lambda}$.

Since the construction of the ideals $\cI_\cD$ and $\cI_\cR$ does not involve a choice of invariant sub-graph $Z \vert_{\crX}$, it follows that the essential cover is independent of such a choice.
\setlength{\parindent}{0cm} \setlength{\parskip}{0cm}

\begin{remark}
Since essential covers are only defined from the essential domain and range, a morphism of covers is not necessarily uniquely determined by a morphism of the associated measured graphs.
\end{remark}
\setlength{\parindent}{0cm} \setlength{\parskip}{0.5cm}

We introduce a property of functions that is relevant for essential covers. 
\bdefin
A continuous map $f \colon Y \to X$ is called \emph{weakly open} if $f^{-1}(U)$ is dense whenever $U$ is open and dense. 
\edefin
\setlength{\parindent}{0cm} \setlength{\parskip}{0cm}

Note that this is equivalent to the statement that $\inte(f(V))\neq\emptyset$ for every open set $V\subseteq Y$ when $X$ and $Y$ are locally compact and Hausdorff. This last property is sometimes called \emph{quasi-open} in the literature (see for instance \cite{BOT06}). For Baire spaces $X$ and $Y$, $f \colon Y \to X$ is weakly open if and only if $f^{-1}(U)$ is co-meagre whenever $U$ is co-meagre.
\setlength{\parindent}{0cm} \setlength{\parskip}{0.5cm}

\begin{proposition}\label{pi_i_weakly_open}
     Suppose $\pi \colon \cW \to W$ is a continuous and proper surjection between locally compact Hausdorff spaces, $C$ is a co-meagre subset of $W$ and $\iota \colon C\to \cW$ is a continuous section (in the sense that $\pi\circ\iota = \text{id}_{C}$) such that $\iota(C)$ is dense in $\cW$. Then $\pi$ and $\iota$ are weakly open.
\end{proposition}
\setlength{\parindent}{0cm} \setlength{\parskip}{0cm}

\begin{proof}
If $U$ is a dense open subset of $W$, then by continuity of $\iota$ and Proposition \ref{continuity_points}, we have $\pi^{-1}(C\cap U) = \iota(C\cap U)$. By the Baire category theorem, $C\cap U$ is dense in $W$, so continuity of $\iota$ and density of $C\cap U$, $\iota(C)$ imply that $\iota(C\cap U)$ is dense in $\cW$. Therefore, $\pi^{-1}(U)$ is dense.
\setlength{\parindent}{0cm} \setlength{\parskip}{0.5cm}

Note that $\iota(C) = \pi^{-1}(C)$, and therefore $\iota(C)$ is a co-meagre subset of $\cW$. So, given a dense open set $\mathcal{U}\subseteq \cW$, $\mathcal{U}\cap \iota(C)$ is dense in $\cW$ by the Baire category theorem. Since $\pi$ is a continuous surjection, $\pi(\mathcal{U}\cap \iota(C)) = \iota^{-1}(\mathcal{U})$ is dense in $W$ and therefore dense in $C$.
\end{proof}
\setlength{\parindent}{0cm} \setlength{\parskip}{0.5cm}

Just as in the case for continuously measured topological graphs, there is a topological obstruction for a graph $Z$ to be essentially continuously measured.
\begin{corollary}\label{cor:wkly_open}
    If $(Z,\Lambda)$ is full and essentially continuously measured, where $Z = (r, s \colon Y \to X)$, then $r \colon Y\to X$ is weakly open.
\end{corollary}
\setlength{\parindent}{0cm} \setlength{\parskip}{0cm}

\begin{proof}
Let $(\mathcal{Z}_{\ess},\bm{\lambda}_{\ess})$ be the associated essential cover and $\iota \colon Z \vert_{\crX}\to \mathcal{Z}_{\ess}$ a continuous section defined on the invariant co-meagre sub-graph $Z \vert_{\crX}$.
Then, by Proposition \ref{continuity_points} and invariance of $Z \vert_{\crX}$, for all $x\in \crX$, we have $r^{-1}(\iota_{X}(x)) = r^{-1}(\pi^{-1}(x)) = \pi^{-1}(r^{-1}(x)) = \iota_{Y}(r^{-1}(x))$. 
Hence, by fullness of $\Lambda$, it follows that $(\bm{\lambda}_{\ess})_{\chi} \colon C_{c}(r^{-1}(\chi)) \to \mathbb{C}$ is faithful for all $\chi$ in the co-meagre set $\iota_{X}(\crX)$. It then follows by continuity of $\bm{\lambda}_{\ess}$ and an argument similar to Lemma~\ref{lem:full} (using Lemma~\ref{full}) that $r \colon \mathcal{Y}_{\ess}\to \mathcal{X}_{\ess}$ satisfies $\emptyset \neq r(\mathcal{U}\cap \iota_{Y}(\crY))\subseteq \inte(r(\mathcal{U}))$ for $\mathcal{U}$ of the form $\mathcal{U} = \menge{y \in \cY_\ess}{f(y) > 0}$, where $0 \leq f\in C_{c}(\mathcal{Y}_{\ess})$. Hence $r \colon \mathcal{Y}_{\ess}\to \mathcal{X}_{\ess}$ is weakly open. Since $r \colon \crY \to \crX$ can be written as $(\pi_{X})_{\ess}\circ r\circ \iota_{Y} \vert_{\crY}$, it follows from Proposition \ref{pi_i_weakly_open} that $r \colon \crY \to \crX$ is weakly open and therefore (by the Baire category theorem) so is $r \colon Y\to X$.
\end{proof}
\setlength{\parindent}{0cm} \setlength{\parskip}{0.5cm}

\subsection{Minimal essential cover}
\label{ss:min_ess}

Given a measured topological graph $(Z,\Lambda)$, when is its minimal restriction $(Z,\Lambda_{\text{min}})$ essentially continuous?
Let us provide a sufficient condition called post-critical sparsity.

Given a measured topological graph $(Z,\Lambda)$, we consider the set of critial vertices 
\[C_{\Lambda} \coloneq \menge{x\in X}{\Lambda(f) \text{ is not continuous at } x, \text{ for some }f\in C_{c}(Y)}.\]
We also consider the collection of all forward orbits of $C_{\Lambda}$, which is the set
\[
P_{\Lambda} \coloneq \menge{x \in X}{\exists \ y_{n}, \dotsc, y_{1} \in Y, \ n\in\mathbb{N} \text{ with } s(y_{i+1}) = r(y_{i}) \ \forall \ 1 \leq i < n \text{ and } r(y_{n}) = x, \ s(y_{1}) \in C_{\Lambda}}.\]
We call $x\in P_{\Lambda}$ a \emph{post-critical vertex} for $(Z,\Lambda)$. 
Notice that $r(s^{-1}(P_{\Lambda}))\subseteq P_{\Lambda}$, so the set $X\setminus P_{\Lambda}$ satisfies $s(r^{-1}(X\setminus P_{\Lambda}))\subseteq X\setminus P_{\Lambda}$.
It follows that $R_{\Lambda} \coloneq (r,s \colon Y\setminus r^{-1}(P_{\Lambda})\to X\setminus P_{\Lambda})$ is an invariant sub-graph of $Z$. We will call $R_{\Lambda}$ the \emph{regular} sub-graph of $(Z,\Lambda)$. The following proposition justifies the name.

\begin{proposition}\label{continuous_on_regular}
Let $(Z,\Lambda)$ be a measured topological graph and $(\mathcal{Z},\lambda)$ its associated cover. Then $\iota \colon Z\to\mathcal{Z}$ is continuous at $z\in R_{\Lambda}$.
\end{proposition}
\setlength{\parindent}{0cm} \setlength{\parskip}{0cm}

Note that the conclusion we obtain is actually stronger than continuity of the restriction of $\iota$ to $R_{\Lambda}$, viewed as a map $R_{\Lambda} \to \cZ$. For the proof, we show that the map $\iota_{X}$ is continuous at $x\in X$ if and only if every $f\in\mathcal{R}$ is continuous at $x$, and a similar statement is true for $\iota_{Y}$. The following lemma is the key step.

\begin{lemma}\label{continuity_criterion}
    If $x\in X$ is not a critical vertex and $f\in\mathcal{D}$ is continuous at every point $y\in r^{-1}(x)$, then $\Lambda(f)$ is continuous at $x$.
\end{lemma}
\begin{proof}
    By the Tietze extension theorem and the hypothesis, we may choose $ \tilde{f}\in C_{c}(Y)$ such that $f \vert_{r^{-1}(x)} = \tilde{f} \vert_{r^{-1}(x)}$. Let $g(\tilde{x}) \coloneq \sup_{\tilde{y}\in r^{-1}(\ti{x})}|f(\tilde{y}) - \tilde{f}(\tilde{y})|$. By continuity at points $y\in r^{-1}(x)$ and compactness of the supports of $\tilde{f}$ and $f$, we have that $g$ is continuous at $x$. Let $0\leq \phi\in C_{c}(Y)$ be such that $|f-\tilde{f}|\leq (g\circ r)\phi$. Positivity of $\Lambda$ implies $\Lambda(|f-\tilde{f}|)\leq g\Lambda(\phi)$. So, by continuity of $g$ and $\Lambda(\phi)$ at $x$, it follows that $\Lambda(|f-\tilde{f}|)$ is continuous at $x$. In addition, our choice of $\ti{f}$ ensures that $\Lambda(|f-\tilde{f}|)(x) = 0$. Therefore, continuity of $\Lambda(\tilde{f})$ at $x$ and $|\Lambda(f - \tilde{f})|\leq \Lambda(|f - \tilde{f}|)$ imply that $\Lambda(f)$ is continuous at $x$.
\end{proof}

\begin{proof}[Proof of Proposition \ref{continuous_on_regular}]
The argument is by induction on the $^*$-sub-algebras $\mathcal{E}_{n}$, $\mathcal{F}_{n}$ appearing in Section \ref{measured_graphs}, using Lemma \ref{continuity_criterion} and the fact that $s^{-1}(P_{\Lambda})\subseteq r^{-1}(P_{\Lambda})$.
\end{proof}
\setlength{\parindent}{0cm} \setlength{\parskip}{0.5cm}

\bdefin
We call a measured topological graph $(Z,\Lambda)$ \emph{post-critically sparse} if $P_{\Lambda}$ and $r^{-1}(P_{\Lambda})$ are contained in meagre sets. 
\edefin
\setlength{\parindent}{0cm} \setlength{\parskip}{0cm}

If $(Z,\Lambda)$ is post-critically sparse, then Proposition~\ref{continuous_on_regular} implies that $(Z,\Lambda_{\text{min}})$ is essentially continuous, with continuous section $\iota \colon R_{\Lambda}\to \mathcal{Z}$ of the minimal cover. 
\bdefin
Given a measured topological graph $(Z,\Lambda)$ which is post-critically sparse, we define $\tilde{Z} \coloneq  \mathcal{Z}_{\ess}$, $\tilde{\bm{\lambda}} \coloneq  \bm{\lambda}_{\ess}$ and call the continuously measured graph $(\tilde{Z},\tilde{\bm{\lambda}})$ \emph{the minimal essential cover of} $(Z,\Lambda)$. We write $\ti{\pi}$ for the canonical projection $\tilde{Z} \to Z$.
\edefin
\setlength{\parindent}{0cm} \setlength{\parskip}{0.5cm}

Below, we describe the minimality property of $(\tilde{Z},\tilde{\bm{\lambda}})$.
If $\Lambda$ is an $r$-operator, we let $\lambda_{x}$ denote the restriction of the linear functional $\Lambda(-)(x)$ to $C_{c}(Y)$. 
Note that $\ti{\pi}$ satisfies $(\tilde{\pi}_{\tilde{Y}})_{*}(\tilde{\lambda}_{x}) = \lambda_{\tilde{\pi}_{\tilde{X}}(x)}$ on a dense subset of $x\in\tilde{X}$.

\begin{theorem}\label{universal}
    Let $(Z,\Lambda)$ be a post-critically sparse measured topological graph. Suppose that $(\dot{Z},\dot{\Lambda})$ along with $\dot{\pi} \colon \dot{Z}\to Z$ is an essential cover of $Z$ such that $(\dot{\pi}_{\dot{Y}})_{*}(\dot{\lambda}_{x}) = \lambda_{\dot{\pi}_{\dot{X}}(x)}$ for all $x$ in some dense subset of $\dot{X}$. Then there is a unique continuous proper morphism $\zeta \colon (\dot{Z},\dot{\Lambda})\to (\tilde{Z},\tilde{\Lambda})$ such that $\tilde{\pi} \circ \zeta = \dot{\pi}$.
\end{theorem}
\setlength{\parindent}{0cm} \setlength{\parskip}{0cm}

Before proving this, we require a lemma.

\begin{lemma}\label{bootstrap1}
    Let $(Z,\Lambda)$ and $(Z', \Lambda')$ be post-critically sparse topological graphs. If $\zeta \colon Z\to Z'$ is a continuous and proper morphism and $(\zeta_{Y})_{*}\lambda_{x} = \lambda'_{\zeta_{X}(x)}$ for all $x$ in a dense subset of $X$, then $(\zeta_{Y})_{*}\lambda_{x} = \lambda'_{\zeta_{X}(x)}$ holds for all $x\in (X\setminus C_{\Lambda})\cap \zeta_{X}^{-1}(X'\setminus C_{\Lambda'})$. Moreover,
   \begin{equation}\label{intertwine}
     \Lambda(f\circ\zeta_{Y})(x) = \Lambda'(f)(\zeta_{X}(x)),
    \end{equation}
for all $f\in\mathcal{D}_{\min}'$ and $x\in (X\setminus P_{\Lambda}) \cap \zeta_{X}^{-1}(X'\setminus P_{\Lambda'})$.
\end{lemma}
\begin{proof}
    Let $\crX$ be a dense subset of $X$ such that $(\zeta_{Y})_{*}\lambda_{x} = \lambda'_{\zeta_{X}(x)}$ for all $x\in \crX$. 
    By density of $\crX$ and continuity of $\Lambda(f\circ\zeta_{Y})$ and $\Lambda'(f) \circ \zeta_X$ on $(X\setminus C_{\Lambda})\cap \zeta_{X}^{-1}(X'\setminus C_{\Lambda'})$ (which is true by construction), 
    it follows that this equation holds for all $x\in (X\setminus C_{\Lambda})\cap \zeta_{X}^{-1}(X'\setminus C_{\Lambda'})$. 
    For $f\in \mathcal{D}'_{\min}$ and $x\in (X\setminus P_{\Lambda}) \cap \zeta_{X}^{-1}(X'\setminus P_{\Lambda'})$, $f$ is continuous on $r^{-1}(\zeta_X(x))$ by Proposition~\ref{continuous_on_regular}, 
    so there is $\tilde{f}\in C_{c}(Y')$ such that $\tilde{f} \vert_{r^{-1}(\zeta_{X}(x))} = f \vert_{r^{-1}(\zeta_{X}(x))}$. 
    Therefore
    \begin{equation}
      \Lambda(f\circ\zeta_{Y})(x) = \lambda_{x}(\tilde{f}\circ\zeta_{Y}) = \lambda'_{\zeta_{X}(x)}(\tilde{f}) = \Lambda'(f)(\zeta_{X}(x)),
    \end{equation}
    as wanted.
\end{proof}

\begin{proof}[Proof of Theorem \ref{universal}]
Let us prove uniqueness, assuming existence. Let $\crZ = (r,s \colon \crY \to \crX) \subseteq Z$  and $\dot{\pi}^{-1}(\crZ) \eqcolon \dot{\crZ} = (r,s \colon \dot{\crY} \to \dot{\crX})\subseteq \dot{Z}$ be co-meagre and invariant graphs such that $\dot{\pi} \colon \dot{\crZ} \to \crZ$ is a bijection. If $\zeta$ and $\zeta'$ both are continuous mappings satisfying $\tilde{\pi}\circ \zeta = \dot{\pi} = \tilde{\pi}\circ \zeta'$, then they are equal on the sub-graph $\dot{\pi}^{-1}(R_{\Lambda})$. Note that $\dot{\pi}$ is weakly open, as the restriction $\dot{\pi} \colon \dot{\crZ} \to \crZ$ is a homeomorphism (this follows from Proposition \ref{continuity_points}). So $\dot{\pi}^{-1}(R_{\Lambda})$ is dense, and hence $\zeta = \zeta'$. 
Moreover, by the weak openness of $\dot{\pi}$ and the Baire category theorem, we have that $C\cap R_{\Lambda}$ and $\dot{\pi}^{-1}(C\cap R_{\Lambda})$ are co-meagre sub-graphs. So, we may assume without loss of generality that $C\subseteq R_{\Lambda}$.

Now, we construct $\zeta$. 
By Lemma \ref{bootstrap1} applied to the morphism $\dot{\pi} \colon  \dot{Z}\to Z$, we have 
\begin{equation}\label{intertwine2}
  \dot{\Lambda}(f\circ\dot{\pi}_{\dot{Y}})(x) = \Lambda(f)(\dot{\pi}_{\dot{X}}(x)),
\end{equation}
for all $f\in\mathcal{D}_{\min}$ and $x\in E\subseteq \dot{\pi}^{-1}(X\setminus P_{\Lambda})$.
Identify $(\mathcal{D}_{\min})_{\ess}$ with $\mathcal{D}_{\min} \vert_{\crY}$ and $(\mathcal{R}_{\min})_{\ess}$ with $\mathcal{R}_{\min} \vert_{\crX}$.
To construct $\zeta_{Y}^{*} \colon (\mathcal{D}_{\min}) \vert_{\crY} \to C_{c}(\dot{Y})$ and $\zeta_{X}^{*} \colon (\mathcal{R}_{\min}) \vert_{\crX} \to C_{c}(\dot{X})$, 
it suffices to show for every $f\in \mathcal{D}_{\min}$ and $g\in \mathcal{R}_{\min}$, there are $\alpha\in C_{c}(\dot{Y})$, $\beta\in C_{c}(\dot{X})$ such that $f \vert_{\crY} \circ \dot{\pi}_{\dot{Y}} \vert_{\dot{\crY}} = \alpha \vert_{\dot{\crY}}$ and $g \vert_{\crX} \circ \dot{\pi}_{\dot{X}} \vert_{\dot{\crX}} = \beta \vert_{\dot{\crX}}$. We proceed by induction on the $^*$-sub-algebras $\mathcal{F}_{n}$, $\mathcal{E}_{n}$.
For $f\in \mathcal{F}_{0} = C_{c}(Y)$ and $g\in C_{c}(X)$, this is trivial. Suppose the result is true for $\mathcal{F}_{n}$ and $\mathcal{E}_{n}$. For $g = \Lambda(f)\in \Lambda(\mathcal{F}_{n})$, let $\alpha\in C_{c}(\dot{Y})$ be such that $f \vert_{\crY} \circ \dot{\pi}_{\dot{Y}} \vert_{\dot{\crY}} = \alpha \vert_{\dot{\crY}}$. By Equation~\eqref{intertwine2}, we have $g \vert_{\crX} \circ \dot{\pi}_{\dot{X}} \vert_{\dot{\crX}} = \Lambda(f \vert_{\crY}) \circ \dot{\pi}_{\dot{X}} \vert_{\dot{\crX}} = \dot{\Lambda}(f \vert_{\crY} \circ \dot{\pi}_{\dot{Y}} \vert_{\dot{\crY}}) = \dot{\Lambda}(\alpha) \vert_{\dot{\crX}}$. As $\mathcal{E}_{n+1}$ is generated by $\Lambda(\mathcal{F}_{n})$ and $\mathcal{E}_{n}$, the result is true for $\mathcal{E}_{n+1}$. A similar argument shows it is true for $\mathcal{F}_{n+1}$.
\end{proof}

\begin{remark}\label{redundant_measure}
    Note that the measure condition ``$(\dot{\pi}_{\dot{Y}})_{*}(\dot{\lambda}_{x}) = \lambda_{\dot{\pi}_{\dot{X}}(x)}$ for all $x$ in some dense subset of $\dot{X}$'' in Theorem \ref{universal} is redundant when the cover and the original graph are weighted by $w =1$. It follows from $\dot{\pi}$ being a bijection between dense and invariant sub-graphs of the domain and co-domain.
\end{remark}
\setlength{\parindent}{0cm} \setlength{\parskip}{0.5cm}

Let us prove a lifting result for morphisms. 
\begin{theorem}
Let $(Z,\Lambda)$ and $(Z',\Lambda')$ be post-critically sparse graphs. Suppose $\zeta \colon Z\to Z'$ is a continuous and proper graph morphism such that $(\zeta_{Y})_{*}\lambda_{x} = \lambda'_{\zeta_{X}(x)}$ for all $x$ in a dense subset and assume $R_{\Lambda}\cap 
        \zeta^{-1}(R_{\Lambda'})$ is dense. Then there is a unique morphism $\tilde{\zeta} \colon (\ti{Z},\tilde{\bm{\lambda}})\to (\tilde{Z}',\tilde{\bm{\lambda}}')$ such that $\tilde{\pi}\circ \tilde{\zeta} = \zeta\circ \tilde{\pi}$.
\end{theorem}
\setlength{\parindent}{0cm} \setlength{\parskip}{0cm}

\begin{proof}
The proof is similar to that of Theorem \ref{universal}. For instance, uniqueness follows from $\tilde{\zeta}$ agreeing with $\zeta$ on the dense sub-graph $R_{\Lambda}\cap \zeta^{-1}(R_{\Lambda'}) $, and Equation~\eqref{intertwine} allows us to construct, by induction on the $^*$-sub-algebras $\mathcal{E}'_{n}$ and $\mathcal{F}_{n}'$, the $^*$-homomorphisms $\zeta^{*}_{X} \colon \mathcal{R}_{\min}' \vert_{X'\setminus P_{\Lambda'}} \to \mathcal{R}_{\min} \vert_{(X\setminus P_{\Lambda}) \cap \zeta^{-1}_{X}(X'\setminus P_{\Lambda'})} $ and $\zeta_{Y}^{*}$, following the same strategy as in Theorem~\ref{universal}.
\end{proof}

\begin{corollary}
    The mapping sending $(Z,\Lambda)$ to $(\tilde{Z},\tilde{\Lambda})$ and $\zeta \colon Z\to Z'$ to $\tilde{\zeta} \colon (\tilde{Z},\tilde{\Lambda})\to (\tilde{Z}',\tilde{\Lambda}')$ is a functor from the category of post-critically sparse measured topological graphs with continuous, proper and weakly open morphisms $\zeta \colon Z\to Z'$ satisfying $(\zeta_{Y})_{*}\lambda_{x} = \lambda'_{\zeta_{X}(x)}$ for all $x$ in a dense subset of $X$ to the category of continuously measured topological graphs with continuous, proper and weakly open measured morphisms.
\end{corollary}

\begin{proof}
    Density of $R_{\Lambda}\cap\zeta^{-1}(R_{\Lambda'})$ is automatic by weak openness. The fact that the composition of two such morphisms still satisfies $(\zeta_{Y})_{*}\lambda_{x} = \lambda'_{\zeta_{X}(x)}$ on a dense subset follows from Lemma~\ref{bootstrap1} (implying equality on a co-meagre set), weak openness and the Baire category theorem. Weak openness of $\tilde{\zeta}$ follows from this morphism agreeing with $\zeta$ on the (co-meagre) graph  $R_{\Lambda}\cap\zeta^{-1}(R_{\Lambda'})$.
\end{proof}
\begin{remark}
The properness of the morphism $\zeta$ in our results can be weakened to the following condition: 
a morphism $\zeta \colon Z\to Z'$ is \emph{$r$-proper} if for every pair of compact sets $K\subseteq Y'$ and $L\subseteq X$, we have that $\zeta_{Y}^{-1}(K) \cap r^{-1}(L)$ is compact. 
This will guarantee that $\Lambda'(f\circ \zeta_{Y})$ is defined for all $f\in C_{c}(Y)$.
\end{remark}
\setlength{\parindent}{0cm} \setlength{\parskip}{0.5cm}

\section{The case of a single transformation}
\label{s:SingleTrafo}
In this section, we study our constructions in the special case of dynamical systems, i.e, topological graphs with $Y = X$, $s = \text{id}_{X}$ and $r = T \colon X\to X$ continuous, where $X$ is a locally compact Hausdorff space. From the minimal cover as a topological graph, we construct the minimal cover as a dynamical system.

\bremark
As we explain, the minimal cover of $T \colon X \to X$ as a dynamical system can also be constructed directly, and this direct construction would not even require continuity of $T$.
\eremark

\subsection{Measured topological dynamical systems}
\label{ss:meas_top_dyn_sys}
 We say a topological dynamical system $T \colon X\to X$ is \emph{measured} if it is measured as a topological graph, with $\mathcal{B} \coloneq \mathcal{D} = \mathcal{R}$ and a $T$-operator $\Lambda \colon \mathcal{B}\to \mathcal{B}$. We denote a measured dynamical system by $(T,\Lambda)$ and call $\Lambda$ its transfer operator.

\bdefin
We let $\cB_{\min}$ be the smallest $^*$-sub-algebra of $\ell^{\infty}_{c}(X)$ which is closed in the inductive limit topology, contains $C_c(X)$ and satisfies $C_c(X) T^* \cB_{\min} \subseteq \cB_{\min}$ and $\Lambda(\cB_{\min}) \subseteq \cB_{\min}$. We write $\Lambda_{\min}$ for the restriction of $\Lambda$ to $\mathcal{B}_{\min}$ and call $(T, \Lambda_{\min})$ the \emph{minimal dynamical restriction}.
\edefin
\setlength{\parindent}{0cm} \setlength{\parskip}{0cm}

Just like the minimal restriction for measured topological graphs, $\mathcal{B}_{\text{min}}$ can be sequentially generated by $^*$-sub-algebras defined as $\cB_0 = C_c(X)$ and for all $n\geq 1$,
\[
  \cB_n \coloneq \overline{{}^*\text{-}\alg} \big(C_c(X), \Lambda(\cB_{n-1}), C_c(X) T^*(\cB_{n-1}) \big).
\]
Then $\cB_n \subseteq \cB_{n+1}$ for all $n\in \Nz$ by construction and it is easy to see that $\cB_{\min}$ is given by the closure of $\bigcup_{n}\mathcal{B}_{n}$ in the inductive limit topology.
\setlength{\parindent}{0cm} \setlength{\parskip}{0.5cm}

By an argument similar to the proof of Lemma \ref{morphism_restriction}, if $\chi \colon (T,\Lambda) \to (T',\Lambda')$ is a continuous morphism between measured dynamical systems, then so is the morphism $\chi_{\min} \colon (T,\Lambda_{\text{min}})\to (T',\Lambda'_{\text{min}})$ induced by $\chi$.

If we view $(T,\Lambda)$ as a measured topological graph, then its minimal restriction may not necessarily be another measured dynamical system. We have $\mathcal{D}_{\min}\subseteq \mathcal{R}_{\min}$ but equality does not necessarily hold. However, it turns out that $\cB_{\min}$ is generated by $C_{c}(X)(T^{n})^{*}\mathcal{R}_{\text{min}}$, $n\geq 0$.

From a topological graph $Z$, we can construct a dynamical system. 
Assuming $s \colon Y\to X$ is proper, the \emph{left infinite path space} $Z^{-\infty} \coloneq \menge{(\ldots, y_{-2}, y_{-1})\in \prod_{i < 0}Y}{s(y_{i}) = r(y_{i+1})\text{ }\forall i < 0}$, 
equipped with the relative topology from the product topology on $\prod_{i < 0}Y$, is locally compact, 
and there is a naturally defined dynamical system on $Z^{-\infty}$ given by  $\sigma(\ldots, y_{-2}, y_{-1}) = (\ldots, y_{-3}, y_{-2})$, called the \emph{shift}.
\setlength{\parindent}{0.5cm} \setlength{\parskip}{0cm}

If $(Z,\Lambda)$ is a measured topological graph with $s \colon Y\to X$ proper, we can induce a measured structure on $\sigma \colon Z^{-\infty}\to Z^{-\infty}$ as follows.
Let $\mathcal{B}^{-\infty}\subseteq\ell^{\infty}_{c}(Z^{-\infty})$ be the closed $^*$-algebra generated by functions of the form $\underline{g} = g_{-n}\otimes\cdots\otimes g_{-1}$ for $g_{i}\in \mathcal{D}$, 
where $\underline{g}(\underline{y}) = g_{-n}(y_{-n})\cdots g_{-1}(y_{-1})$ for all $\underline{y} = (...,y_{-2},y_{-1})\in Z^{-\infty}$.
Then $\underline{g}\circ\sigma = \underline{g}\otimes 1$, and from this it is easy to see $\sigma^{*}(\mathcal{B}^{-\infty})\subseteq M(\mathcal{B}^{-\infty})$. 
For $\underline{g}$ as above, let $\Lambda_{-\infty}(\underline{g}) = g_{-n}\otimes\cdots\otimes (g_{-2}(\Lambda(g_{-1})\circ s))$. 
This is well-defined and extends continuously to a $\sigma$-operator $\Lambda_{-\infty} \colon \mathcal{B}^{-\infty}\to \mathcal{B}^{-\infty}$. Furthermore, the operation $(-)^{-\infty}$ is a functor, sending a morphism $\upsilon \colon (Z,\Lambda)\to (Z', \Lambda')$ to the morphism $\upsilon^{-\infty} \colon (\sigma, \Lambda_{-\infty})\to (\sigma', \Lambda'_{-\infty})$ defined as $\upsilon^{-\infty}(\ldots, y_{-2},y_{-1}) = (\ldots, \upsilon_{Y}(y_{-2}), \upsilon_{Y}(y_{-1}))$. 
We call this the \emph{paths functor}. 
Note that if $(T,\Lambda)$ is a measured dynamical system, then $(T,\Lambda)^{-\infty}$ is naturally conjugate to $(T,\Lambda)$. 
We will always make this identification.
\setlength{\parindent}{0cm} \setlength{\parskip}{0.5cm}

Let us see how the paths functor interacts with the minimal restriction functors.
\begin{proposition}
\label{prop:commuting_functors}
 $(-)_{\text{min}}\circ (-)^{-\infty} = (-)^{-\infty}\circ (-)_{\text{min}}$ as functors defined on the category of measured topological graphs with proper source maps. The $(-)_{\text{min}}$ functor on the left hand side is the minimal dynamical restriction, while the $(-)_{\text{min}}$ functor on the right hand side is the minimal restriction as a topological graph.
\end{proposition}
\setlength{\parindent}{0cm} \setlength{\parskip}{0cm}

\begin{proof}
    Let $(Z,\Lambda)$ be a topological graph with $s \colon Y\to X$ proper. Then $(\mathcal{B}_{\text{min}})^{-\infty}$ and $(\mathcal{B}^{-\infty})_{\text{min}}$ are both closed $^*$-sub-algebras of $\mathcal{B}^{-\infty}\subseteq \ell^{\infty}_{c}(Z^{-\infty})$ with $\sigma$-operator the restriction of $\Lambda_{-\infty}$. Hence, $(\mathcal{B}^{-\infty})_{\text{min}}\subseteq (\mathcal{B}_{\text{min}})^{-\infty}$. We need to show the reverse inclusion holds. Functions $\underline{g} = g_{-k}\otimes\cdots\otimes g_{-1}$, for $g_{i}\in \mathcal{F}_{n}$, $k,n\in\mathbb{N}$ generate $(\mathcal{B}_{\text{min}})^{-\infty}$ and $\underline{g} = h(g_{-k}\circ \sigma^{k-1}\cdot g_{-k+1}\circ \sigma^{k-2}\cdots g_{-1})$ for some $h\in C_{c}(Z^{-\infty})$, so by $\sigma^{*}((\mathcal{B}^{-\infty})_{\text{min}})\subseteq M((\mathcal{B}^{-\infty})_{\text{min}})$ it suffices to show $ \mathcal{F}_{n}\subseteq (\mathcal{B}^{-\infty})_{\text{min}}$ for all $n\in\mathbb{N}$. We proceed by induction on $n$, noting that $n = 0$ is trivial. Suppose $\mathcal{F}_{n-1}\subseteq (\mathcal{B}^{-\infty})_{\text{min}}$ for some $n-1\geq 0$. 
Then, for every $g,h\in\mathcal{F}_{n-1}$, we have $g(\Lambda(h)\circ s) = g\Lambda_{-\infty}(h)$ in $(\mathcal{B}^{-\infty})_{\text{min}}$ and $g(\Lambda(h)\circ r) = g(\Lambda_{-\infty}(h)\circ\sigma)$ in $(\mathcal{B}^{-\infty})_{\text{min}}$. 
Therefore, $\mathcal{F}_{n}\subseteq (\mathcal{B}^{-\infty})_{\text{min}}$.
\end{proof}
\setlength{\parindent}{0cm} \setlength{\parskip}{0.5cm}

\subsection{Dynamical covers}

By Theorem~\ref{thm:correspondence}, we can understand the minimal dynamical restriction functor in terms of continuous data.

\bdefin
We denote by $(T_{\min},\bm{\lambda}_{\min})$ the associated cover of $(T,\Lambda_{\min})$, which is a continuously measured dynamical system. We call this the \emph{minimal dynamical cover} of $(T,\Lambda)$ and denote its underlying topological space by $X_{\min}$ and its structure maps by $\pi \colon X_{\min} \to X$ and $\iota \colon X \to X_{\min}$.
\edefin
The minimal dynamical cover has the following universal property (the proof is analogous to the one of Theorem~\ref{thm:mincover}).
\begin{theorem}
\label{thm:min_dyn_cover}
    Suppose that $(\dot{T},\dot{\bm{\lambda}})$ is another continuously measured dynamical system that covers $(T,\Lambda)$ in the sense that there is a continuous proper surjection $\dot{\pi} \colon \dot{X}\to X$ and a morphism $\dot{\iota} \colon (T,\Lambda)\to (T,\dot{\bm{\lambda}})$ with dense image such that $\dot{\pi} \circ \dot{T} = T \circ \dot{\pi}$ and $\dot{\pi} \circ \dot{\iota} = \id_{Z}.$ Then there is a unique continuous proper morphism $\chi \colon (\dot{T},\dot{\bm{\lambda}})\to (T_{\min},\bm{\lambda}_{\min})$ such that $\pi\circ \chi = \dot{\pi}$ and $\chi \circ \dot{\iota} = \iota$.
\end{theorem}

We have the following corollary of Proposition~\ref{prop:commuting_functors}.

\begin{corollary}
\label{cor:cover_inf_paths}
If $(Z,\Lambda)$ is a measured topological graph with $s \colon Y\to X$ proper, then the minimal dynamical cover of $(\sigma, \Lambda_{-\infty})$ is isomorphic to $(\sigma, (\hat{\Lambda})_{-\infty})$ as covers, with structure maps given by $\pi^{-\infty} \colon (\hat{Z})^{-\infty}\to Z^{-\infty}$ and $\iota^{-\infty} \colon Z^{-\infty}\to \hat{Z}^{-\infty}$.
In particular, the minimal dynamical cover of a measured dynamical system $(T,\Lambda)$ is isomorphic as covers to the shift on the infinite path space of its minimal graph cover.
\end{corollary}

\subsection{Essential dynamical covers}
\label{ss:ess_dyn_cover}
We now briefly introduce minimal essential dynamical covers in analogy to \S~\ref{ss:Ess}. 
\setlength{\parindent}{0cm} \setlength{\parskip}{0cm}

\bdefin
Let $T \colon X\to X$ be a topological dynamical system. We say a subset $\crX \subseteq X$ is \emph{fully invariant} if $T^{-1}(\crX) = \crX$.
\setlength{\parindent}{0.5cm} \setlength{\parskip}{0cm}

An \emph{essential dynamical cover} of $T \colon X\to X$ is a continuously measured dynamical system $(\mathcal{T},\bm{\lambda})$ together with a continuous proper factor map $\pi \colon \mathcal{X}\to X$ satisfying the property that there is a fully invariant set $\crX \subseteq X$ such that $\crX$ and the (automatically fully invariant) set $\pi^{-1}(\crX)$ are co-meagre and $\pi \colon \pi^{-1}(\crX) \to \crX$ is a bijection.
\edefin
\setlength{\parindent}{0cm} \setlength{\parskip}{0cm}

Note that every essential dynamical cover is an essential cover in the sense of Section \ref{ss:Ess}.
\setlength{\parindent}{0cm} \setlength{\parskip}{0.5cm}

Let us now outline the construction of the minimal essential dynamical cover of $T \colon X \to X$, which is analogous to the construction of minimal essential covers in Section \ref{ss:Ess}.

Let $(T,\Lambda)$ be a measured topological dynamical system. Recall the critical and post-critical sets $C_{\Lambda}$ and $P_{\Lambda}$ from Section \ref{ss:min_ess}. In the setting of dynamical systems, we have $P_{\Lambda} = \bigcup_{n\geq 0} T^{n}(C_{\Lambda})$. 
\bdefin
We define $G_{\Lambda} = \bigcup_{n\leq 0}T^{n}(P_{\Lambda})$ and call it the \emph{saturated critical set} of $(T,\Lambda)$.
\edefin
We have the following analogue of Proposition \ref{continuous_on_regular} (the proof is similar).
\begin{proposition}
    Let $(T,\Lambda)$ be a measured topological dynamical system. Then $\iota \colon X\to X_{\min}$ is continuous at $x\in X\setminus G_{\Lambda}$.
\end{proposition}

\bdefin
We call $(T,\Lambda)$ \emph{sparsely critically saturated} if $G_{\Lambda}$ is a meagre set. 
In that case, we let $X_{\ess}  \coloneq  \overline{\iota(X\setminus G_{\Lambda})}$ and define $T_{\ess}$ as the restriction of $T_{\min}$ to $X_{\ess}$. 
Now, $T_{\ess} \colon X_{\ess} \to X_{\ess}$ together with the restriction $\pi_{\ess} \colon X_{\ess} \to X$ of $\pi$ and the restriction $\bm{\lambda}_{\ess} \colon C_{c}(X_{\ess})\to C_{c}(X_{\ess})$ of $\bm{\lambda}_{\min}$ form an essential cover, which we call the \emph{minimal essential dynamical cover}. 
\edefin
\setlength{\parindent}{0cm} \setlength{\parskip}{0cm}

Note that if $T$ is weakly open, then $(T,\Lambda)$ is sparsely critically saturated if and only if it is post-critically sparse.
\setlength{\parindent}{0cm} \setlength{\parskip}{0.5cm}

Just as in Section~\ref{ss:min_ess}, there is an alternative construction of the minimal essential dynamical cover as the Gelfand spectrum of the topological $^*$-algebra $(\mathcal{B}_{\text{min}})_{\ess}$.
Using this alternate description we can describe the minimal essential dynamical cover in terms of the minimal essential graph cover.

\begin{corollary}
\label{cor:inf_path_ess}
    The minimal essential dynamical cover of a sparsely critically saturated measured dynamical system $(T,\Lambda)$ is isomorphic (as covers) to the shift on the infinite path space of the minimal essential graph cover.
\end{corollary}
\setlength{\parindent}{0cm} \setlength{\parskip}{0cm}

\begin{proof}
    Let $\mathcal{D}$ be the domain of the minimal graph cover of $(T,\Lambda)$. By Corollary \ref{cor:cover_inf_paths}, $\mathcal{B}_{\text{min}}$ is generated by $C_{c}(X)(T^{n})^{*}\mathcal{D}$ for all $n\in\mathbb{N}$. We have (by invariance) $\big ( (T^{n})^{*}\mathcal{D}\big )|_{X\setminus G_{\Lambda}} = (T^{n})^{*}(\mathcal{D}|_{X\setminus G_{\Lambda}})$ for all $n\in\mathbb{N}$. Hence $(\mathcal{B}_{\text{min}})_{\ess} = (\mathcal{B}_{\text{min}})|_{X\setminus G_{\Lambda}}$ is generated by $(T^{n})^{*}(\mathcal{D}|_{X\setminus G_{\Lambda}})$ for $n\geq 0$. Applying the Gelfand correspondence finishes the proof.
\end{proof}
\setlength{\parindent}{0cm} \setlength{\parskip}{0.5cm}

We state the analogous universal property and lifting properties of morphisms for sparsely critically saturated dynamical systems. They are proved in a similar way to the results of Section \ref{ss:min_ess}.
\begin{theorem}\label{thm:universal_ess_dyn}
    Suppose $(\dot{T},\dot{\bm{\lambda}})$ along with $\dot{\pi} \colon \dot{X}\to X$ is an essential dynamical cover of $T$ such that $\dot{\pi}_{*}\dot{\lambda}_{\dot{x}} = \lambda_{\dot{\pi}(\dot{x})}$ for all $\dot{x}$ in a dense subset of $\dot{X}$. Then there is a unique continuous proper morphism $\chi \colon (\dot{T},\dot{\bm{\lambda}}) \to (T_{\ess},\bm{\lambda}_{\ess})$ such that $\pi \circ \chi = \dot{\pi}$.
\end{theorem}

\begin{remark}\label{rem:redundant_dyn_measure}
    As in Remark \ref{redundant_measure}, the measure condition in Theorem \ref{thm:universal_ess_dyn} is redundant if the essential cover and the original dynamical system are weighted by $w=1$.
\end{remark}

\begin{theorem}\label{thm:dyn_lifting_res_ess}
    Let $(T,\Lambda)$ and $(T',\Lambda')$ be sparsely critically saturated topological dynamical systems. Suppose that $\chi \colon X\to X'$ is a continuous proper and equivariant map such that $\chi_{*}\lambda_{x} = \lambda'_{\chi(x)}$ for all $x$ in a dense subset and $(X\setminus G_{\Lambda})\cap\chi^{-1}(X'\setminus G_{\Lambda'})$ is dense. Then there is a unique morphism $\chi_{\ess} \colon (T_{\ess}, \bm{\lambda}_{\ess}) \to (T'_{\ess}, \bm{\lambda}'_{\ess})$ such that $\pi_{\ess}' \circ \chi_{\ess} = \chi \circ \pi_{\ess}$.
\end{theorem}

\begin{corollary}
    The mapping sending $(T,\Lambda)$ to $(T_{\ess},\bm{\lambda}_{\ess})$ and $\chi \colon X \to X'$ to $\chi_{\ess} \colon (T_{\ess},\bm{\lambda}_{\ess})\to (T'_{\ess},\bm{\lambda}'_{\ess})$ is a functor from the category of sparsely critically saturated topological dynamical systems with continuous, proper, weakly open and equivariant maps $\chi \colon X\to X'$ satisfying $\chi_{*}\lambda_{x} = \lambda'_{\chi(x)}$ for all $x$ in a dense subset of $X$ to the category of continuously measured topological dynamical systems with morphisms given by continuous, proper, weakly open morphisms.
\end{corollary}

\bremark
\label{rem:TessTopTrans}
Let $(T,\Lambda)$ be a sparsely critically saturated dynamical system. Assume that $X$ is a compact metric space and $T$ is topologically transitive (i.e., there is a dense forward orbit under $T$). Then any essential cover $(\ti{T},\ti{\bm{\lambda}})$ of $T$ is topologically transitive. Indeed, by \cite[Theorem~5.9~(iv)]{Wal82}, the set of points in $X$ with dense forward orbit under $T$ forms a dense $G_{\delta}$ subset. As the surjection $\pi \colon \ti{X} \to X$ restricts to a homeomorphism on a fully invariant co-meagre subset of the domain and co-domain, it follows that there exists a point $x$ in this restricted co-domain whose forward orbit under $T$ is dense in $X$. Hence the forward orbit under $\ti{T}$ of the unique element in $\pi^{-1}(x)$ is dense in $\ti{X}$, as desired.
\eremark

\subsection{Essential dynamical covers of local injections}

Let us apply our results in the locally injective case, where several reductions on the hypotheses can be made. 
We will assume throughout this section that our dynamical systems are weighted by $w=1$. 
Recall from Example \ref{example:local_injections} that in this case, the minimal graph cover is a local homeomorphism weighted also by $w = 1$, and therefore the same is true for the minimal dynamical cover.

\begin{lemma}
\label{lem:meagre_to_meagre}
    Let $T \colon X\to X$ be a continuous and locally injective map defined on a second countable locally compact Hausdorff space $X$. If $C\subseteq X$ is meagre, then so is $T(C)$.
\end{lemma}
\setlength{\parindent}{0cm} \setlength{\parskip}{0cm}

\begin{proof}
    By second countability and local compactness, it suffices to prove the lemma for $C\subseteq X$ compact with empty interior, contained in some open neighbourhood $U$ such that $T|_{U}$ is injective. Suppose $T(C)$ does not have empty interior and let $T(x) = y\in \text{int}(T(C))$ for some $x\in U$. Let $(x_{i})$ be a net in $X\setminus C$ converging to $x$. Since $y\in \text{int}(T(C))$, there is $i'$ such that $T(x_{i}) = T(c_{i})$ for some $c_{i}\in C$ and $x_{i}\in U$, for all $i\geq i'$. By injectivity of $T|_{U}$ it follows that $x_{i} = c_{i}$ for all $i\geq i'$, a contradiction to $x_{i}\in X\setminus C$. 
\end{proof}
\setlength{\parindent}{0cm} \setlength{\parskip}{0.5cm}

We will say $T \colon X\to X$ is \emph{open at $x\in X$} if for every neighbourhood $U$ of $x$, we have $T(x)\in\inte(T(U))$. For a locally injective map $T \colon X\to X$, we will call $C = \menge{x\in X}{T\text{ is not open at }x}$ the \emph{critical point set}.

\begin{corollary}\label{cor:loc_inj_meagre}
   Let $T \colon X\to X$ be a continuous and locally injective map on a second countable locally compact Hausdorff space $X$. Then $T$ is sparsely critically saturated if and only if the critical point set $C$ is meagre.
\end{corollary}
\setlength{\parindent}{0cm} \setlength{\parskip}{0cm}

\begin{proof}
    Let us first assume $C$ is meagre. 
    It is easy to see that $T(C) = C_{\Lambda}$, so if $C$ is meagre, then $P_{\Lambda} = \bigcup_{n\geq 1}T^{n}(C)$ is meagre by Lemma \ref{lem:meagre_to_meagre} and the Baire category theorem.
    Since $C$ is meagre, we have $U\setminus C \neq\emptyset$ for every open set $U$. Therefore, we have $\emptyset\neq T(U\setminus C)\subseteq \text{int}(T(U))$, so that $T$ is weakly open. Hence, $G_{\Lambda} = \bigcup_{n\leq 0}T^{n}(P_{\Lambda})$ is meagre, proving that $T$ is sparsely critically saturated.
\setlength{\parindent}{0cm} \setlength{\parskip}{0.5cm}

For the converse assume $T$ is sparsely critically saturated.
Then the minimal essential cover exists and $C_{\Lambda}$ is meagre. 
Hence, by Corollary \ref{cor:wkly_open}, $T$ is weakly open. So $T(C) = C_{\Lambda}$ implies $C\subseteq T^{-1}(C_{\Lambda})$ is meagre.
\end{proof}
\setlength{\parindent}{0cm} \setlength{\parskip}{0.5cm}

\begin{corollary}\label{cor:a1to1}
    Let $X$ and $X'$ be second countable locally compact Hausdorff spaces, $T' \colon X'\to X'$ a local homeomorphism, $T \colon X\to X$ a continuous and locally injective map and $\pi \colon X'\to X$ a proper and continuous equivariant map. Then $\pi$ is an essential cover if and only if $X_0' \coloneq \menge{x'\in X'}{|\pi^{-1}(\pi(x'))| = 1}$ is dense.
\end{corollary}
\setlength{\parindent}{0cm} \setlength{\parskip}{0cm}

\begin{proof}
The ``only if'' direction is trivial, so we prove the ``if'' direction.
If $X'_{0}$ is dense, then \cite[Lemma~2.2]{BOT06} implies that $X'_{0}$ and $X_{0} \coloneq  \pi(X'_{0})$ are co-meagre. Set $C'_{0} = X'\setminus X'_{0}$. By Lemma \ref{lem:meagre_to_meagre}, weak openness of $T'$ and the Baire category theorem, it follows that $G'_{0} = \bigcup_{n,m\geq 0}(T')^{-n}((T')^{m}(C'_{0}))$ is meagre. Therefore $X'_{00} \coloneq X\setminus G'_{0}\subseteq X'_{0}$ is a co-meagre and fully $T'$-invariant subset of $X'$. Since $\pi \colon X'_{0}\to X_{0}$ is a homeomorphism (Proposition \ref{continuity_points}), it follows that $X_{00} \coloneq  \pi(X'_{00})$ is co-meagre in $X_{0}$ and therefore co-meagre in $X$. By $T'$-invariance of $X'_{00}$ we have $T(X_{00}) \subseteq X_{00}$. Hence $T \colon X_{00}\to X_{00}$ is weakly open, so that $T \colon X\to X$ is weakly open. The same procedure as above (now using weak openness of $T$ and Lemma \ref{lem:meagre_to_meagre}) produces a fully $T$-invariant co-meagre subspace $X_{000}\subseteq X_{00}$. Then $X'_{000} \coloneq  \pi^{-1}(X_{000})$ is a fully $T'$-invariant and co-meagre (by weak openness of $\pi$) subset of $X'$. It follows that $\pi$ is an essential cover.
\end{proof}
\setlength{\parindent}{0cm} \setlength{\parskip}{0.5cm}

Now we can state our theorem for locally injective maps. For a map $f \colon X\to Y$ between topological spaces, we will say $f$ is \emph{almost one-to-one} if $\menge{x\in X}{|f^{-1}(f(x))| = 1}$ is dense in $X$.

\begin{theorem}\label{thm:up_loc_inj}
   Suppose $T \colon X\to X$ is locally injective and $X$ is a second countable locally compact Hausdorff space. If the critical point set of $T$ is meagre, then the minimal essential cover $T_{\ess} \colon X_{\ess} \to X_{\ess}$ exists and is a local homeomorphism defined on a second countable locally compact Hausdorff space. The factor map $\pi_{\ess} \colon (T_{\ess},X_{\ess})\to (T, X)$ is continuous, almost one-to-one and proper.

If $(T', X')$ is another local homeomorphism defined on a second countable locally compact Hausdorff space with continuous, almost one-to-one and proper factor map $\pi' \colon (T', X')\to (T,X)$, then there is a unique continuous factor map $\chi \colon (T', X')\to (T_{\ess}, X_{\ess})$ such that $\pi' = \pi_{\ess} \circ \chi$. Moreover, $\chi$ is proper, almost one-to-one and the restriction $(T')^{-1}(x')\to T_{\ess}^{-1}(\chi(x'))$ of $\chi$ is a bijection, for all $x'\in X'$.
\end{theorem}
\setlength{\parindent}{0cm} \setlength{\parskip}{0cm}

\begin{proof}
    This follows immediately from our description of sparsely critically saturated dynamical systems (Corollary \ref{cor:loc_inj_meagre}), our description of essential covers (Corollary \ref{cor:a1to1}) and  the universal property of the minimal essential cover Theorem \ref{thm:universal_ess_dyn}. The bijection property is equivalent to $\chi$ being a weighted morphism (relative to weights $w=1$). Using the fact that a locally compact Hausdorff space is second countable if and only if $C_{c}(X)$ is separable (contains a countable dense set), it is easy to see that if $X$ is second countable, then so is $X_{\ess}$.
\end{proof}
\setlength{\parindent}{0cm} \setlength{\parskip}{0.5cm}

By the discussion in the introduction, we also obtain a lifting result that can be seen as a non-invertible analogue of a lifting result of Putnam for Smale spaces \cite[Theorem~1.1]{Put}.

\begin{theorem}\label{thm:loc_inj_lifting_thm}
    Suppose $(T, X)$ and $(T', X')$ are locally injective dynamical systems defined on second countable locally compact Hausdorff spaces with meagre critical point sets, and suppose $\chi \colon (T, X)\to (T', X')$ is a continuous, almost one-to-one and proper factor map. Then there is a unique continuous factor map $\chi_{\ess} \colon (T_{\ess}, X_{\ess}) \to (T'_{\ess}, X'_{\ess})$ such that $\pi_{\ess} \circ \chi_{\ess} = \chi \circ \pi'_{\ess}$. 
    Moreover, $\chi_{\ess}$ is almost one-to-one, proper and $\chi_{\ess}$ restricts to bijections $T_{\ess}^{-1}(\tilde{x}) \to (T'_{\ess})^{-1}(\chi_{\ess}(\tilde{x}))$, for all $\tilde{x} \in X_{\ess}$.
\end{theorem}
\setlength{\parindent}{0cm} \setlength{\parskip}{0cm}

\begin{proof}
    Using the weak open property of $T$, $T'$ and Lemma \ref{lem:meagre_to_meagre}, we can find (as in Corollary \ref{cor:a1to1}) a co-meagre and fully $T'$-invariant subset $X'_{0}\subseteq X'$ such that $X_{0}  \coloneq  \chi^{-1}(X'_{0})$ is co-meagre (as well as fully $T$-invariant) and $\chi$ restricts to a bijection $X_{0}\to X'_{0}$. Therefore $\chi$ restricts to bijections $T^{-1}(x_{0})\to (T')^{-1}(\chi(x_{0}))$, for all $x_{0}\in X_{0}$, and $\chi$ is weakly open. We may therefore apply the lifting result of Theorem~\ref{thm:dyn_lifting_res_ess}.
\end{proof}
\setlength{\parindent}{0cm} \setlength{\parskip}{0.5cm}

\subsection{Positive expansivity} \label{sec:T-expansive}

Expansivity is a fundamental concept in dynamics with applications to symbolic dynamics and thermodynamic formalism (see for example \cite{PP}). Recall that a continuous map $T\colon X\to X$ on a compact metrizable space $X$ is \emph{positively expansive} if there is a metric $d$ and a constant $c>0$ such that for every two distinct $x, z\in X$ there is $n\in \Nz$ such that $d(T^n x, T^n z) \geq c$. We shall need the following reformulation of positive expansivity. Recall that a compact Hausdorff space is metrizable if and only if it is second countable.

\begin{lemma}
\label{lem:expansive}
  A continuous map $T\colon X\to X$ on a metrizable compact Hausdorff space $X$ is positively expansive if and only if there is a finite open cover $\{U_{i}\}^{m}_{i=1}$ of $X$ such that $T^{n}(x), T^{n}(x')\in U_{i_{n}}$ for some $i_{n}\leq m$, for all $n\geq 0$ implies $x= x'$.
\end{lemma}
\setlength{\parindent}{0cm} \setlength{\parskip}{0cm}

\begin{proof}
    The ``if'' direction follows by choosing a metric $d$ and letting $c$ be the Lebesgue number of the cover $\{U_{i}\}^{m}_{i=1}$. The ``only if'' direction follows by choosing a finite open cover of the metric balls of radius $c/2$ relative to the metric $d$ and constant $c>0$ exhibiting positive expansivity of $T$.
\end{proof}

\begin{remark}
    We shall take the characterization of positive expansivity from Lemma~\ref{lem:expansive} as the definition in the case of non-metrizable compact Hausdorff spaces. The definition we use is equivalent to the property called ``weakly positively expansive'' in \cite[Definition~1]{RW07}.
\end{remark}
\setlength{\parindent}{0cm} \setlength{\parskip}{0.5cm}

Let us prove a general result about expansivity passing to covers.

\begin{proposition}
\label{prop:expansive}
    Let $T\colon X\to X$ be a continuous and positively expansive map on a compact Hausdorff space $X$ and let $\pi \colon \hat{Z} \to X$ be a surjective morphism from a compact topological graph $\hat{Z} = (\hat{r}, \hat{s} \colon \hat{Y} \to \hat{X})$ to $X$ (considered as a topological graph with $r = T$ and $s = \text{id}_{X}$). If $\pi_{Y}$ is locally injective, then the shift $\sigma \colon \hat{Z}^{-\infty}\to \hat{Z}^{-\infty}$ is positively expansive.
\end{proposition}
\setlength{\parindent}{0cm} \setlength{\parskip}{0cm}

\begin{proof}
    Let $\{U_{i}\}^{m}_{i=1}$ be a finite cover of $X$ satisfying the positive expansivity characterization in Lemma~\ref{lem:expansive} (applied to $T$), and let $\{V_{j}\}^{n}_{j=1}$ be a finite open cover of $\hat{Y}$ such that $\pi_{\hat{Y}} \vert_{V_{j}}$ is injective, for all $1 \leq j \leq n$. Then $W_{ij} = \hat{Z}^{-\infty} \cap \big (\prod_{i < -1} \hat{Y} \times (\pi^{-1}_{\hat{Y}}(U_{i})\cap V_{j})\big)$ for $1 \leq i \leq m$, $1 \leq j \leq n$ form a finite open cover of $\hat{Z}^{-\infty}$ satisfying the positive expansivity characterization of Lemma \ref{lem:expansive} (for $\sigma$).
\end{proof}
\setlength{\parindent}{0cm} \setlength{\parskip}{0.5cm}

Our goal is to show that if $T$ is positively expansive and the minimal graph cover associated with weight $w =1$ terminates after finitely many steps, 
then the minimal dynamical cover of $T$ is again positively expansive.

Let $r, s \colon \hat{Y} \to \hat{X}$ be the minimal graph cover of $T \colon X\to X$ with weight $w =1$ in the sense of Section \ref{ss:DescCovers}. 
As in the $n$-th step of our construction, let $r,s \colon Y_n \to X_{n}$ be the graph, let $\lambda_{n-1,x}$ be the measures (for $x \in X_n$)  
and let $\pi_{Y_n} \colon Y_n \to Y = X$, $\pi_{X_n} \colon X_n \to X$ be the canonical projections. 

We prove the key step in our main result of this section.
\begin{lemma} 
\label{lem:PreimagesFarApart-check} If $T \colon X\to X$ is locally injective and $X$ is a compact Hausdorff space, then $\pi_{Y_{n}}$ and $\pi_{X_{n}}$ are locally injective for every $n\geq 0$.
\end{lemma}
\setlength{\parindent}{0cm} \setlength{\parskip}{0cm}

\begin{proof}
  We proceed by induction, noting that the case $n=0$ is trivial.
  Suppose the proposition is true for $n-1\geq 0$, and we first describe the measures $\lambda_{n-1,x}$, $x\in X_{n}$. For $x\in X_{0}\subseteq X_{n}$, we have $\lambda_{n-1, x} = \sum_{y\in \iota_{n-1}(T^{-1})(x))\subseteq (T\circ \pi_{Y_{n-1}})^{-1}(x)}\delta_{y}$. Any limit $x\in X_{n}$ is therefore of the form $\lambda_{n-1,x} = \sum_{y\in \text{supp}(\lambda_{n-1,x})}n_{y}\delta_{y}$ for some $n_{y}\in \mathbb{N}$ and with $\text{supp}(\lambda_{n-1,x})\subseteq (T \circ \pi_{Y_{n-1}})^{-1}(\pi_{X_n}(x))$. Local injectivity of $T\circ \pi_{Y_{n-1}}$ implies $n_{y} = 1$ for all $y\in \text{supp}(\lambda_{n-1,x}$).
\setlength{\parindent}{0cm} \setlength{\parskip}{0.5cm}
  
  We now prove that $\pi_{X_{n}}$ is locally injective. For $\lambda_{n-1,x}$, let $(\lambda_{n-1, x_{i}})_{i}$ and $(\lambda_{n-1, x'_{i}})_{i}$ be two nets weak$^{*}$ converging to $\lambda_{n-1,x}$ such that $(\pi_{X_{n}}(x_{i}) = \pi_{X_{n}}(x'_{i}))_{i}$ converges to $\pi_{X_{n}}(x)$. By local injectivity of $T\circ \pi_{Y_{n-1}}$ and compactness, there are pairwise disjoint open sets $\{U_{j}\}^{k}_{j=1}$, $\{V_{j}\}^{k}_{j=1}$ of $Y_{n-1}$ and an open neighbourhood $W\subseteq X$ of $\pi_{X_{n}}(x)$ such that $T\circ \pi_{Y_{n-1}}|_{U_{j}}$ is injective and $\overline{V_{j}}\subseteq U_{j}$ for all $1 \leq j \leq k$, and $(T\circ \pi_{Y_{n-1}})^{-1}(W)\subseteq \bigcup^{k}_{j=1} V_{j}$. Therefore, if we choose $\{0\leq \phi_{j}\}^{k}_{j=1}\subseteq C_{c}(Y_{n-1})$ such that $\phi_{j} = 1$ on $\overline{V_{j}}$ and $\phi_{j} = 0$ on $Y_{n-1} \setminus U_{j}$, then any two Radon measures $\lambda = \sum_{y \in S} \delta_y$ and $\lambda' = \sum_{y\in S'}\delta_{y}$ with $\supp(\lambda)\cup \supp(\lambda')\subseteq (T\circ \pi_{Y_{n-1}})^{-1}(x')$, for $x'\in W$ are equal if and only if $\lambda(\phi_{j}) = \lambda'(\phi_{j})$ for all $1 \leq j \leq k$. 
  
  It follows from the description $\lambda_{n-1, x} = \sum_{y\in \text{supp} (\lambda_{n-1,x})}\delta_{y}$ with $\text{supp}(\lambda_{n-1,x})\subseteq (T \circ \pi_{Y_{n-1}})^{-1}(\pi_{X_n}(x))$ that $\lambda_{n-1, x_{i}}(\phi_{j}) = \lambda_{n-1, x}(\phi_{j}) = \lambda_{n-1, x'_{i}}(\phi_{j})$ eventually, for all $1 \leq j \leq k$. 
  Therefore, $\lambda_{n-1, x_{i}} = \lambda_{n-1,x_{i}'}$ eventually, so $\pi_{X_{n}}$ is locally injectivity.
  Now, the fact that $\pi_{Y_{n}}$ is locally injective follows from local injectivity of $\pi_{X_{n}}$ and the description of $y\in Y_{n}$ as the triple $(\pi_{Y_{n}}(y), \lambda_{n-1, r(y)}, \lambda_{n-1,s(y)})$.
\end{proof}

\begin{remark}
 If $T \colon X\to X$ is sparsely critically saturated, let $Y'_{n}\subseteq Y_{n}$ and $X'_{n}\subseteq X_{n}$ denote the $n$-th step of the essential cover construction. If $T \colon X\to X$ is locally injective, then the maps $\pi_{Y'_{n}}$ and $\pi_{X'_{n}}$ are locally injective, for all $n\geq 0$, since they are restrictions of $\pi_{Y_{n}}$ and $\pi_{X_{n}}$.
\end{remark}

\begin{theorem}
\label{thm:hatTExpansive-check}
  Let $T\colon X\to X$ be a continuous and positively expansive map on a compact Hausdorff space $X$. If the minimal graph cover (or minimal essential graph cover) of $T \colon X\to X$ with weight $w=1$ terminates after finitely many steps, i.e. $Y_{N} = \plim_{n}Y_{n}$ for some $N\in \Nz$,
  then $T_{\min}\colon X_{\min} \to X_{\min}$ (or $T_{\ess} \colon X_{\ess}\to X_{\ess}$) is positively expansive.
\end{theorem}
\bproof
Expansivity implies local injectivity of $T \colon X\to X$, so Lemma \ref{lem:PreimagesFarApart-check} and the hypothesis imply $\pi_{\hat{Y}} \colon \hat{Y}\to X$ is locally injective. By Proposition \ref{prop:expansive}, $(\sigma, (\hat{Z})^{-\infty})$ is positively expansive. By Corollary \ref{cor:cover_inf_paths}, $(\sigma, (\hat{Z})^{-\infty})$ is conjugate to $(T_{\min}, X_{\min})$. A similar proof holds for the essential cover case, using Corollary \ref{cor:inf_path_ess}.
\eproof
\setlength{\parindent}{0cm} \setlength{\parskip}{0.5cm}

\bremark
Note that using the picture of the $n$-th step of the minimal essential cover construction as the essential algebra of the $n$-th step of the minimal cover construction, it is easy to see that the minimal essential cover terminating is a weaker condition than the minimal cover terminating.
\eremark

\bremark
    \label{rem:hatTExpansive_dyn_version}
A similar theorem holds when the dynamical cover construction of Section \ref{ss:meas_top_dyn_sys} terminates after finitely many steps.
\setlength{\parindent}{0.5cm} \setlength{\parskip}{0cm}

Alternatively, there is the following link between our finiteness condition and the condition that the minimal dynamical cover terminates after finitely many steps: 
for $m\in\mathbb{N}$ and $n\in\mathbb{N}$, let $\mathcal{F}_{m,n}$ denote the closed $^*$-sub-algebra generated by $(T^{k})^{*}\mathcal{F}_{m}$ for $k\leq n$, 
where $\mathcal{F}_{n}$ are the algebras for the $n$-th step of the minimal topological graph cover of $T \colon X\to X$ in Section \ref{measured_graphs}. 
Denote by $\mathcal{B}_{n}$ the algebras in the $n$-th step of the minimal dynamical cover appearing in Section \ref{ss:meas_top_dyn_sys}. By comparing generators of these algebras, we see that for $n\in\mathbb{N}$, there exists $n' \in \Nz$ with $n' \geq n$ such that $\mathcal{B}_{n}\subseteq \mathcal{F}_{n,n}\subseteq \mathcal{B}_{n'}$. 
So, if there is $K\in\mathbb{N}$ such that $\mathcal{B}_{K+1} = \mathcal{B}_{K}$, then there is $M\in\mathbb{N}$ and $N\in\mathbb{N}$ such that $\mathcal{F}_{M+1, N} = \mathcal{F}_{M,N}$. Notice that the Gelfand spectrum of $\mathcal{F}_{M,N}$ is the space $Z^{N}_{M}$ of length $N$ paths in the graph $Z_{M}$. Therefore, the minimal dynamical cover construction terminating after finitely many steps implies the length $N$ path spaces of the minimal graph cover construction terminate after finitely many steps. An application of Proposition \ref{prop:expansive} then implies $T_{\min}^{N}$ is positively expansive. Therefore, $T_{\min}$ is positively expansive.
\eremark
\setlength{\parindent}{0cm} \setlength{\parskip}{0.5cm}

\bremark
  In general, we cannot expect $T_{\min}$ to be positively expansive even if $T$ is positively expansive.
  A one-sided subshift (not necessarily of finite type) is always positively expansive, and the cover constructed in \cite{CarlsenPhD} is always a local homeomorphism on a totally disconnected compact Hausdorff space, and coincides with our cover (see Section \ref{sec:ExSubshifts}).
  If this cover were positively expansive, then it would be (conjugate to) a shift of finite type and $T$ would be a sofic shift (sofic shifts are homomorphic images of shifts of finite type \cite[Chapter 3]{Lind-Marcus2021}).
  However, there are uncountably many subshifts that are not sofic.
\eremark

\bcor
\label{cor:Prep-TDF}
Let $T\colon X\to X$ be a continuous and positively expansive map on a compact metrizable space $X$. Assume that $T$ is topologically transitive (i.e., there is a dense forward orbit under $T$), that $T$ with weight $w = 1$ is sparsely critically saturated, and that our essential cover construction terminates after finitely many steps. Then the minimal essential dynamical cover $T_{\ess} \colon X_{\ess} \to X_{\ess}$ is a  topologically transitive, positively expansive local homeomorphism on a compact metrizable space $X_{\ess}$.
\ecor
\setlength{\parindent}{0cm} \setlength{\parskip}{0cm}

The point is that the original transformation does not need to be open (see for example \cite{Hir,Ros}).

\bproof
This follows from Lemma~\ref{lem:full}, Remark~\ref{rem:TessTopTrans} and Theorem~\ref{thm:hatTExpansive-check}.
\eproof
\setlength{\parindent}{0cm} \setlength{\parskip}{0.5cm}

\bremark
\label{rem:TDF-FiniteSteps}
One consequence of Corollary~\ref{cor:Prep-TDF} is that for continuous and positively expansive transformations $T$ satisfying the conditions in Corollary~\ref{cor:Prep-TDF}, the thermodynamic formalism is valid, in the following sense (we take the axiomatic approach to the thermodynamic formalism from \cite{GKLM}): there is a Banach space $\cX$ of functions on $X$, a open set $\cU \subseteq \cX$ of potentials such that for every $A \in \cU$, the associated transfer operator $\crL_A(f)(x) = \sum_{y \in T^{-1}(x)} e^{A(y)} f(y)$ has a positive maximal eigenvalue $\lambda_A$ which is a simple eigenvalue with a spectral gap below it, a positive eigenfunction $h_A$ corresponding to $\lambda_A$ which is bounded away from $0$, and the dual operator $\crL_A^*$ has an eigenmeasure $\nu_A$ for the eigenvalue $\lambda_A$ so that $h_A \nu_A$ is an invariant measure for $T$. We refer to \cite[Section~2.2]{GKLM} for details. All this follows from the observation that validity of the thermodynamic formalism passes to factors, as explained in \cite[Section~2.4.2]{GKLM}, and that the thermodynamic formalism has already been established for open, topologically transitive, positively expansive continuous transformations on compact metric spaces (see for instance \cite[Section~4]{PU}).
\eremark

\bremark\label{rmk:finitely_presented}
    Here is another consequence: Given a positively expansive map $T\colon X\to X$ on a compact metrizable space $X$ for which our construction terminates after finitely many steps, the cover $T_{\min}$ is positively expansive and open, so by \cite[Theorem 1]{Sakai2003} $T_{\min}$ has the shadowing property. 
    This implies that $T_{\min}$ is finitely presented, so it is a factor of a shift of finite type \cite[Proposition 1.4.50]{Nekrashevych2022}.
    Since $T$ is positively expansive, it follows from \cite[Proposition 1.4.45]{Nekrashevych2022} that $T$ is finitely presented.
\eremark

\section{Examples from subshifts}

\label{sec:ExSubshifts}

In this section, we apply our construction of minimal and minimal essential covers to subshifts and compare with existing constructions such as Krieger and Fischer covers. We refer the reader to \cite{Lind-Marcus2021} for background on symbolic dynamics.
Let $\mathfrak{a} = \{1,\ldots,m\}$ be a finite set (viewed as a discrete space) of symbols. 
The space $\mathfrak{a}^\Nz$ is a Cantor space in the product topology and the shift $\sigma(x_{1}, x_{2},...) = (x_{2}, x_{3},...)$, defined for all $x = (x_{1}, x_{2}, ...) \in \mathfrak{a}^\Nz$, is a local homeomorphism.
The pair $(\sigma, \mathfrak{a}^\Nz)$ is the \emph{full one-sided $m$-shift}. 
A subspace $X$ of $\mathfrak{a}^\Nz$ is \emph{shift-invariant} if $\sigma(X) \subseteq X$ (we do not assume equality),
and if $X$ is closed and shift-invariant, then the restricted shift $\sigma_X \coloneqq \sigma|_X$ on $X$ is locally injective but not open in general. 
The pair $(\sigma_X, X)$ is called a \emph{one-sided subshift}. We denote by $L(X)$ the set of finite words in $\mfa$ which appear as subwords of elements in $X$. Moreover, $L_k(X)$ denotes the set of words in $L(X)$ of length $k$, and $L_{\leq k}(X)$ stands for the set of words in $L(X)$ of length at most $k$.

Let $(\sigma_X, X)$ be a one-sided subshift.
For $k\in \Nz$, the \emph{$k$-past} of an element $x\in X$ is the set $P_k(x) \coloneq \menge{ \nu\in L_{k}(X)}{\nu x\in X}$.
Let $P_{\leq k}(x) \coloneqq \bigcup_{j=1}^k P_j(x)$ and $P_{<\infty}(x) \coloneqq \bigcup_{j=1}^\infty P_j(x)$.
For two finite words $\alpha$ and $\beta$, we shall consider the closed (and generally not open) set $C(\beta,\alpha) \coloneq \menge{x\in X}{x\in Z(\alpha), \beta \sigma_X^{|\alpha|}(x) \in X}$ where $|\alpha|$ denotes the length (i.e. the number of symbols) of $\alpha$.

\subsection{The topological Krieger graph}
\label{ss:top_Krieg_graph}
Let $(\sigma_X, X)$ be a one-sided subshift.
If $\sigma_X$ is of finite type (i.e. a local homeomorphism), then it is a standard result that $\sigma_X$ is (conjugate to) the shift on the one-sided infinite path space of a finite directed graph \cite[Section 2.2]{Lind-Marcus2021}.
We now construct a labeled compact topological graph $K$ such that $(\sigma_X, X)$ is conjugate to the shift on the one-sided infinite path space of labels on $K$,
and another compact topological graph $\mathcal{K}$ which we identify with the minimal cover of $\sigma_X$ viewed as a topological graph with weight $w=1$.

For a subset $F\subseteq L(X)$, we use the notation $F|_{\leq k} \coloneqq F\cap L_{\leq k}(X)$.
The vertices of $K$ are given as 
\[
I \coloneqq \menge{F\subseteq L(X)}{\forall \, k> 0 \,, \exists \, x^{(k)} \in X \text{ such that } F|_{\leq k} = P_{\leq k}(x^{(k)})}
\]
and we equip this with the following topology:
A net $(F_i)_i\subseteq I$ converges to $F\in I$ if for all $k> 0$, we have $F_i|_{\leq k} = F|_{\leq k}$ eventually. 
Note that $I$ is equal to the closure of the image of the function $i\colon X \to I$ given by $x\mapsto P_{<\infty}(x)$. 
The set of edges is given as 
\[
J \coloneq \menge{ (F,a) \in P\times \mathfrak{a}}{a\in F}
\]
in the subspace topology of the product topology, and $J$ is equal to the closure of the image of the function $i \colon X\to J$ defined as $x \mapsto (P_{<\infty}(\sigma_X(x)),x_1)$.
The range and source maps are $r(F,a) = F$ and $s(F,a) = \menge{ \nu\in L(X)}{\nu a\in F}$. 
It is easily verified that this is well-defined, and that both range and source are continuous. 
Moreover, the range map $r$ is a local homeomorphism, so $K$ is continuously measured if weighted by $w=1$.
We shall refer to $K = (r,s \colon J\to I)$ as the \emph{topological Krieger graph} of $\sigma_X$.
Note also that the map $\cL\colon J \to \mathfrak{a}$ given by $\cL(F,a) = a$ defines a continuous labeling map. 
The subshift $(\sigma_X, X)$ is conjugate to the shift on the infinite paths of labels on $K^{-\infty}$, meaning that $\mathcal{L}^{-\infty} \colon (\sigma, K^{-\infty})\to (\sigma_{X}, X)$ is a factor map.
In this sense, the topological graph $K$ together with its labeling $\mathcal{L}$ represents the subshift $(\sigma_X,X)$.

\bremark
The pair of mappings $i\colon X \to J$, $i\colon X\to I$ define a graph morphism $i \colon (\sigma, X)\to K$. 
We refer to $i$ as the \emph{Krieger graph morphism}, and the image is usually known as the Krieger graph. 
As explained above, this is a countable and dense sub-graph of our topological Krieger graph $K$.
If $\sigma_X$ is sofic, then $K$ is a finite labeled graph and it coincides with the Krieger graph, see \cite[Section 3.2]{Lind-Marcus2021} or \cite{Krieger1984} where it is called the past set cover. 
We will see soon that $\mathcal{L}^{-\infty} \colon (\sigma, K^{-\infty})\to (\sigma_{X}, X)$ together with the section $i^{-\infty} \colon (\sigma_{X}, X)\to (\sigma, K^{-\infty})$ 
is isomorphic (as a cover) to the minimal dynamical cover of $\sigma_{X}$.
\eremark

\bremark
Matsumoto \cite{Matsumoto1997} (cf. \cite{Carlsen-Matsumoto2004}) studied the space $I$ (called $\Omega$ in \cite{Matsumoto1997}) in relation to $C^*$-algebras attached to general one-sided subshifts.
Based on this, Carlsen and Eilers \cite[Theorem 4.3]{Carlsen-Eilers2006} show that $I$ can be used to compute the $K$-theory for certain subshift $C^*$-algebras.
It would be interesting to see if the topological Krieger graph can be used to compute $K$-theory for $C^*$-algebras of our covers.
\eremark

\subsection{The augmented topological Krieger graph}
Let us construct an augmented graph $\mathcal{K}$ that we prove is a model for the minimal cover for $\sigma_X$ viewed as a topological graph. For each integer $k> 1$, we first construct a graph $\mathcal{K}_k$. 
The vertices and edges are given as
\begin{align}
\mathcal{I}_k &\coloneq \menge{ (P_{\leq k-1}(\alpha x) ,\alpha)}{\alpha\in L_{k-1}(X), \alpha x\in X}    \\
\mathcal{J}_k &\coloneq \menge{ (P_{\leq k}(\alpha x) ,a\alpha)}{\alpha\in L_{k-1}(X), a\in L_{1}(X),  a\alpha x\in X} 
\end{align}
and the range map is $r(P_{\leq k}(\alpha x) ,a\alpha) = (P_{\leq k-1}(\alpha x), \alpha)$ while the source map is defined as
$s(P_{\leq k}(\alpha x), a\alpha) = (P_{\leq k-1}(a\alpha x), (a \alpha)_{[1,k-1]})$.
Now $\mathcal{K}_k = (\mathcal{I}_k, \mathcal{J}_k, r,s)$ is a finite discrete graph.
There are surjective graph morphisms $\upsilon_{k+1} \colon \mathcal{K}_{k+1} \to \mathcal{K}_k$ given by 
$\upsilon_{\mathcal{I}_{k+1}}(P_{\leq k}(\alpha x),\alpha) = (P_{\leq k-1}(\alpha x), \alpha_{[1,k-1]})$ and $\upsilon_{\mathcal{J}_{k+1}}(P_{\leq k+1}(\alpha x),a \alpha) = (P_{\leq k}(\alpha x),(a \alpha)_{[1,k]})$.

We define the graph $\mathcal{K}$ to be the projective limit of $\mathcal{K}_k$ along the morphisms $(\upsilon_k)_k$. 
The vertices and edges may be identified with 
\begin{align}
  \mathcal{I} &= \menge{(F,x)\in I\times X}{\forall \, k\in \Nz \ \exists \, x^{(k)}\in Z(x_{[1,k]}) \text{ such that } F \vert_{\leq k} = P_{\leq k}(x^{(k)})} \\
    \mathcal{J} &= \menge{(G,by)\in I\times X}{(G,y)\in \mathcal{I}, b\in G}
\end{align}
in the subspace topology, and $r(G,by) = (G,y)$, $s(G,by) = (\menge{ \nu\in L(X)}{\nu b \in G}, by)$ are the range and source maps, respectively.
Then $\mathcal{K}$ is a compact topological graph whose range map is a local homeomorphism, so we weight $\mathcal{K}$ by $w=1$. 
We call $\mathcal{K}$ the \emph{augmented topological Krieger graph}.

There is a graph morphism $j\colon X \to \mathcal{K}$ given by $j_{\mathcal{I}}(x) = (P_{<\infty}(x), x)$ and $j_{\mathcal{J}}(x) = (P_{<\infty}(\sigma_X(x)), x)$, for all $x\in X$,
and this is injective with dense image. 
It is a section of the continuous graph morphism $\pi \colon \mathcal{K}\to X$ given by $\pi_{\mathcal{I}}(F,x) = x$ and $\pi_{\mathcal{J}}(G,x) = x$. 
Note that $j_{\mathcal{J}} \colon \sigma_{X}^{-1}(x)\to r^{-1}(j_{\mathcal{I}}(x))$ is a bijection, for all $x\in X$, so $j$ is a weighted graph morphism. 
Therefore, $\mathcal{K}$ together with $\pi$ and $j$ is a graph cover of $(\sigma_{X},X)$.

\begin{theorem}\label{thm:minimal-graph-cover-subshift}
    Let $(\sigma_X, X)$ be a one-sided subshift. 
    The minimal graph cover of $\sigma_X$ is isomorphic (as covers) to the graph $\mathcal{K}$ together with $\pi$ and $j$.
\end{theorem}
\setlength{\parindent}{0cm} \setlength{\parskip}{0cm}

\begin{proof}
Let $r, s \colon \hat{Y} \to \hat{X}$ denote the minimal graph cover of $\sigma_X$ we constructed in Section \ref{sec:app_min_cover}, 
and let $\mathcal{D}_{\text{min}}$ and $\mathcal{R}_{\text{min}}$ denote the domain and range for the minimal restriction of $(\sigma_{X},X)$, weighted by $w=1$.
By the universal property (see Theorem \ref{thm:mincover}) it suffices to show that 
$(j_{\mathcal{I}})^*(C(\mathcal{I})) \subseteq \mathcal{R}_{\text{min}} \simeq C(\hat{X})$ and $(j_{\mathcal{J}})^*(C(\mathcal{J})) \subseteq \mathcal{D}_{\text{min}}\simeq C(\hat{Y})$. 
We treat the cases of vertices and edges separately.
Note that $C(\mathcal{I})$ and $C(\mathcal{J})$ are the limits of $C(\mathcal{I}_{k})$ and $C(\mathcal{J}_{k})$, and each $\mathcal{K}_k$ is a finite discrete graph.
\setlength{\parindent}{0cm} \setlength{\parskip}{0.5cm}

For the case of vertices, let $(P_{\leq k}(\alpha x),\alpha) \in \mathcal{I}_k$ and $z\in X$, and observe that
\[
\chi_{(P_{\leq k}(\alpha x),\alpha)}( j_{\mathcal{I}}(z)) = 
\begin{cases}
    1 &\textrm{if } P_{\leq k}(\alpha x) = P_{\leq k}(z) \textrm{ and } \alpha = z_{[1,k-1]} \\
    0 &\textrm{otherwise.}
\end{cases}
\]
This equals the indicator over the set $Z(\alpha)\cap \bigcap_{\beta\in P_{\leq k}(\alpha x)} C(\beta, \varnothing)$
which is in $\mathcal{R}_{\text{min}}$ since $\Lambda^{|\beta|}(\chi_{Z(\beta)}) = \chi_{C(\beta,\varnothing)}$. 

For edges, we take $(P_{\leq k}(\alpha x),a\alpha) \in \mathcal{J}_k$ and $z\in X$ and observe that 
\[
\chi_{(P_{\leq k}(\alpha x),a\alpha)}( j_{\mathcal{J}}(z)) = 
\begin{cases}
    1 &\textrm{if } P_{\leq k}(\alpha x) = P_{\leq k}(\sigma_X(z)) \textrm{ and } a \alpha = z_{[1,k]} \\
    0 &\textrm{otherwise.}
\end{cases}
\]
This is equal to the indicator of $Z(a \alpha)\cap \sigma_X^{-1}( \bigcap_{\beta\in P_{\leq k}(\alpha x)} C(\beta, \varnothing))$ which is in $\mathcal{D}_{\text{min}}$.
\end{proof}
\setlength{\parindent}{0cm} \setlength{\parskip}{0.5cm}

There is a natural graph morphism $\mu\colon \mathcal{K} \to K$ given by $\mu_{\mathcal{I}}(F,x) = F$ for all $(F,x)\in  \mathcal{I}$, and $\mu_{\mathcal{J}}(G,by) = (G,b)$ for all $(G,by)\in \mathcal{J}$. 
Applying the paths functor yields an equivariant continuous map $\mu^{-\infty}\colon \mathcal{K}^{-\infty} \to K^{-\infty}$ given by applying $\mu$ to each edge in an infinite path.

\begin{lemma} \label{lem:R-conjugacy}
  The map $\mu^{-\infty}\colon \mathcal{K}^{-\infty} \to K^{-\infty}$ is a topological conjugacy and an isomorphism of dynamical covers of $(\sigma_{X}, X)$.
\end{lemma}
\setlength{\parindent}{0cm} \setlength{\parskip}{0cm}

\begin{proof}
Note that $\mu\circ j = i$ and so functorality implies $\mu^{-\infty}\circ j^{-\infty} = i^{-\infty}$. 
It suffices to show the image of $i^{-\infty}$ is dense in $K^{-\infty}$ and construct a continuous function $S \colon K^{-\infty}\to \prod_{n<0}(P\times X)$ such that $S\circ i^{-\infty} = j^{-\infty}$.
We first construct $S$. 
Take $(\ldots, (F_2,a_2),(F_1,a_1))\in K^{-\infty}$ and observe that $a \coloneqq a_1a_2\ldots$ is in $X$. 
Define $S(\ldots, (F_2,a_2),(F_1,a_1)) = (\ldots,(F_2, \sigma_{X}(a)), (F_1,a))\in \prod_{n<0}(I\times X)$. 
Then, $S\circ i^{-\infty} = j^{-\infty}$ and $S$ is continuous.
\setlength{\parindent}{0cm} \setlength{\parskip}{0.5cm}

We prove that the image of $i^{-\infty}$ is dense. 
Let $\underline{q}= (...(F_{2},a_{2}), (F_{1},a_{1}))\in K^{-\infty}$. 
For every $n$ in $\mathbb{N}$, we have $\menge{b\in L(X)}{b a_{n+1}\in F_{n+1}} = F_{n}$. 
Therefore, $\menge{b\in L(X)}{b a_{k+1}a_{k}...a_{n+1}\in F_{n+1}} = F_{k}$ for all $k \leq n$. 
In particular, for every $n$ in $\mathbb{N}$, there is $x_{n}$ in $X$ such that $F_{n+1}|_{\leq 2n+1} = P_{\leq 2n+1}(x_{n})$ 
and hence $y_{n} = a_{1}...a_{n+1}x_{n}$ is an element in $X$ such that 
$P_{\leq n}(\sigma^{k}(y_{n})) = \menge{b\in L_{\leq n}(X)}{b a_{k+1}...a_{n+1}x_{n}\in X} = \menge{b\in L_{\leq n}(X)}{b a_{k+1}...a_{n+1}\in F_{n+1}} = (F_{k})_{\leq n},$ for all $k\leq n$. 
Therefore, $\{i^{-\infty}(y_{n})\}_{n\in\mathbb{N}}$ converges to $\underline{q}$. This proves density of $i^{-\infty}(X)$.
\end{proof}
\setlength{\parindent}{0cm} \setlength{\parskip}{0cm}

We can now identify the minimal dynamical cover of a subshift in terms of the topological Krieger graph.
This follows immediately from Theorem \ref{thm:minimal-graph-cover-subshift}, Corollary \ref{cor:cover_inf_paths} and Lemma \ref{lem:R-conjugacy}.

\begin{corollary} 
  Let $(\sigma_X,X)$ be a subshift. 
  The minimal dynamical cover of $\sigma$ is isomorphic (as a cover) to the shift on $\mathcal{K}^{-\infty}$
  and also the shift on the space of infinite paths on the topological Krieger graph $K$.
\end{corollary}
\setlength{\parindent}{0cm} \setlength{\parskip}{0.5cm}

\bremark
In \cite[Chapter 2]{CarlsenPhD}, Carlsen defines a dynamical system $(\sigma_{\tilde{X}}, \tilde{X})$ from a subshift $(\sigma_X,X)$ in such a way that $\sigma_{\tilde{X}}$ factors onto $\sigma_X$
and $\sigma_{\tilde{X}}$ is a local homeomorphism.
Then $\sigma_{\tilde{X}}$ is called a \emph{cover} of $\sigma_X$, and this is further studied in \cite{Brix-Carlsen2020,Brix2023,He-Wei2024}
in which lifting results and connections to \'etale groupoids and $C^*$-algebras are discussed.
Using the universal property of our minimal covers (in a way similar to the proof of Theorem \ref{thm:minimal-graph-cover-subshift}), it can be shown that Carlsen's cover is isomorphic to our minimal dynamical cover, so our construction and results provide a conceptual explanation for these previous works.
\eremark

\subsection{The essential dynamical cover of a subshift}

In this section, we identify the minimal essential dynamical cover of a subshift $(\sigma_X,X)$.
We say a subshift $\sigma_X$ is \emph{regular} if $\sigma_X(X)$ is regular in $X$, in the sense that the interior of $\sigma_X(X)$ is dense in $\sigma_X(X)$.
Most subshifts of interest are regular (in particular, when $\sigma_X$ is surjective), but not all subshifts are regular. 
Below we see that this is exactly the condition we need to define the essential cover of $\sigma_X$.

\begin{lemma} \label{weakly_open}
  Let $(\sigma_X,X)$ be a subshift. Then the post-critical set (when $\sigma_{X}$ is weighted by $w=1$) is meagre. Moreover, $\sigma_{X}$ is weakly open if and only if it is regular.
\end{lemma}
\setlength{\parindent}{0cm} \setlength{\parskip}{0cm}

\begin{proof}
For a function $f \colon X\to Y$, we let $D(f) \coloneq \menge{x\in X}{f\text{ is not continuous at }x}$.
Recall the functions $P_{k} \colon X\to L_{k}(X)$ from Section \ref{ss:top_Krieg_graph}.
We regard $L_{k}(X)$ as a discrete topological space. 
Then $x$ is a critical point of $\sigma_{X}$ if and only if $x\in D(P_{1})$, and $x\in D(P_{k})$ implies $\sigma_{X}(x)\in D(P_{k+1})$. 
Therefore, to show the post-critical set is meagre, it suffices to show $D(P_{k})$ is meagre, for every $k\in\mathbb{N}$ (by the Baire category theorem).
\setlength{\parindent}{0.5cm} \setlength{\parskip}{0cm}

Choose a sequence $(U_n)_n$ of nested open and closed subset of $\mathfrak{a}^\Nz$ such that $X \subseteq U_n \subseteq U_{n+1}$ for all $n$ and $X = \bigcap_n U_n$.
As the full shift $\sigma$ is a local homeomorphism, 
the mappings $P_k^n\colon X \to \mathfrak{a}^k\subseteq L(\mathcal{A}^{\mathbb{N}})$ given by $P_k^n(x) = (\sigma^{-1}(x)\cap U_n)_{[1,k]}$ for all $x\in X$ are continuous. 
Since $P_k$ is the pointwise limit of the continuous functions $P_k^n$, it follows from the Baire category theorem that $D(P_{k})$ is meagre.
\setlength{\parindent}{0cm} \setlength{\parskip}{0.5cm}

We prove the second statement now. If $\text{int}(\sigma_{X}(X))$ is not dense in $\sigma_{X}(X)$, then the open dense set $U = \text{int}(\sigma_{X}(X))\cup X\setminus \sigma_{X}(X)$ does not have a dense pre-image $\sigma_{X}^{-1}(U)$. Hence, $\sigma_{X}$ is not weakly open.
\setlength{\parindent}{0.5cm} \setlength{\parskip}{0cm}

To prove the converse, consider $X_{0} \coloneq  \menge{x\in X}{\sigma_{X}(x)\notin \text{int}(\sigma_{X}(U)),\text{ for some neighbourhood }U\text{ of } x}$, where the interior is taken as a subset of $\sigma_{X}(X)$. 
Since $\sigma_{X} \colon X\to \sigma_{X}(X)$ is a quotient map, we may write 
\[
X_{0} = \menge{x\in X}{\Lambda(f)\circ \sigma\text{ is not continuous at }x\text{ for some positive function } f\in C(Z(x_{1}))}.
\]
Therefore, $X_{0}\subseteq D(j_{\mathcal{J}})$. 
Moreover, $D(j_{\mathcal{J}}) \subseteq\bigcup_{k\in\mathbb{N}}D(P_{k}\circ \sigma)$. Each set $D(P_{k}\circ \sigma)$ is meagre 
because $P_{k}\circ \sigma$ is the pointwise limit of the continuous maps $P_{k}^{n}\circ \sigma$. 
Therefore, $X_{0}$ is meagre, and it follows that $\sigma_{X} \colon X\to \sigma_{X}(X)$ is weakly open. 
Note that an inclusion $X\subseteq Y$ of a subspace $X$ is weakly open if and only if its interior is dense in $X$, so by the hypothesis, it follows that $\sigma_{X} \colon X\to X$ is weakly open.
\end{proof}
\setlength{\parindent}{0cm} \setlength{\parskip}{0.5cm}

Let $(\sigma_X,X)$ be a regular subshift and recall the Krieger graph morphism $i\colon X \to K$.
We define the \emph{topological Fischer graph} of $(\sigma_X,X)$ to be $K_\ess \coloneqq \overline{ i(R_{\Lambda})}$, where $R_{\Lambda}$ is the regular sub-graph of $\sigma_{X}$ weighted by $1$.
In particular, if $P_{\Lambda}$ is the postcritical set of $\sigma_{X}$, 
then $i_{I}\colon X\setminus P_{\Lambda} \to I_\ess$ and $i_{J}\colon X\setminus \sigma_X^{-1}(P_{\Lambda}) \to J_\ess$ are continuous inclusions with dense image.

\bremark
As in the case of the Krieger graph, we could define the augmented Fischer graph and show that it coincides with the essential graph cover of $\sigma_X$.
\eremark

The fact that the left-infinite path space of our topological Fischer graph is (isomorphic to) the essential dynamical cover of a regular subshift 
now follows from our construction and the identification of $K^{-\infty}$ with the minimal dynamical cover.

\begin{theorem} 
  Let $(\sigma_X,X)$ be a regular subshift.
  Then the minimal essential dynamical cover is isomorphic (as covers) to the shift on the left-infinite path space $K_\ess^{-\infty}$.
\end{theorem}

Let $(\sigma_X,X)$ be a subshift.
It is \emph{irreducible} if whenever $\lambda,\nu \in L(X)$, there is $\mu\in L(X)$ such that $\lambda\mu\nu\in L(X)$.
Recall that a word $\mu$ in $L(X)$ is \emph{synchronizing} if whenever $\nu$ and $\lambda$ are in $L(X)$ and $\nu \mu$ and $\mu\lambda$ are in $L(X)$, then $\nu \mu\lambda$ is in $L(X)$.
A point $x\in X$ is synchronizing if it contains a synchronizing word, and we let $X_\sync$ be the subset of $X$ of points that contain synchronizing words.
If $\sigma_X$ is irreducible and $X_\sync$ is nonempty, then it is open and dense.

Recall that a subshift is \emph{sofic} if it is a homomorphic image of a shift of finite type.
Sofic shifts always contain synchronizing words, and the (topological) Krieger graph $K$ is finite and discrete.
In particular, $I_{\ess} = i_{I}(X\setminus P_{\Lambda})$ and $J_{\ess} = i_{J}(X\setminus \sigma_X^{-1}(P_{\Lambda}))$, where $P_{\Lambda}$ is the post-critical set of $\sigma_{X}$ (weighted by $w=1$).
The \emph{Fischer graph} $F$ of an irreducible sofic subshift $(\sigma_X, X)$ is the unique minimal right-resolving graph presentation of $\sigma_X$ \cite[Section 3.4]{Lind-Marcus2021},
and it may be concretely defined as the sub-graph of $K$ whose vertices are represented by synchronizing points,
i.e., $F^0 = i_{I}(X_\sync) \subseteq I$ and $F^1 = r^{-1}(i_{I}(X_\sync))\cap s^{-1}(i_{J}(X_\sync)) \subseteq J$.

Let us relate synchronizing words to continuity points.

\begin{lemma}\label{lem:magic.word.is.cont.}
Let $\sigma_{X} \colon X\to X$ be a subshift and suppose $x$ is synchronizing. Then there is a neighbourhood $U$ of $x$ such that, for every $k$ in $\mathbb{N}$, $P_{k}|_{U}$ is continuous. 
Therefore, $X_{\sync}\subseteq X\setminus P_{\Lambda}$, $i_{I}$ is continuous at every $x\in X_\sync$ and $i_{J}$ is continuous at every $y\in \sigma^{-1}_{X}(X_\sync)$.
\end{lemma}
\setlength{\parindent}{0cm} \setlength{\parskip}{0cm}

\begin{proof}
We may write $x = \alpha m x'$, where $m$ is a synchronizing word. 
Then, for every $y$ in $Z(\alpha m)$, we have $P_{k}(y) = P_{k}(\alpha m)  \coloneq  \menge{f\in L_{k}(X)}{f \alpha m\in L(X)}$. 
The other conclusions follow from $P_{\Lambda}\subseteq \bigcup_{k}D(P_{k})$ (see the proof of Lemma \ref{weakly_open}) and $\mu\circ j = i$.
\end{proof}

For an irreducible sofic subshift, our topological Fischer graph coincides with the classical Fischer graph.

\begin{theorem} 
  Let $(\sigma_X,X)$ be an irreducible sofic subshift.
  Then $K_\ess$ is isomorphic to the Fischer graph of $\sigma_X$.
\end{theorem}

\begin{proof}
The edges of the Fischer graph are $r^{-1}(i_{I}(X_{\sync}))\cap s^{-1}(i_{I}(X_{\sync}))$. 
Note that $s\circ i_{J} = i_{I}$ implies $i_{J}(X_{\sync})\subseteq s^{-1}(i_{I}(X_{\sync}))$ and $j$ being a (weighted) morphism implies $i_{J}(\sigma_{X}^{-1}(X_{\sync}))= r^{-1}(i_{I}(X_{\sync}))$. 
Therefore, 
\begin{equation}\label{containment}
    i_{J}(X_{\sync}\cap \sigma_{X}^{-1}(X_{\sync}))\subseteq s^{-1}(i_{I}(X_{\sync}))\cap r^{-1}(i_{I}(X_{\sync}))\subseteq i_{J}(\sigma_{X}^{-1}(X_{\sync})).
\end{equation}
\setlength{\parindent}{0cm} \setlength{\parskip}{0.5cm}

The map $\sigma_{X}$ is irreducible, and hence surjective, so  Lemma \ref{weakly_open} implies $\sigma_{X}$ is weakly open. 
As $X_\sync$ is open and dense, Lemma \ref{lem:magic.word.is.cont.} implies that $\sigma_{X}^{-1}(X_{\sync})\cap X_{\sync}$ is an open dense set contained in $X\setminus \sigma_{X}^{-1}(P_{\Lambda})$. 
It follows from density of $X\setminus \sigma_{X}^{-1}(P_{\Lambda})$ and continuity of $i_{J}$ that each containment 
\[
  i_{J}(X_{\sync}\cap \sigma_{X}^{-1}(X_{\sync}))\subseteq i_{Q}(\sigma_{X}^{-1}(X_{\sync})) \subseteq \overline{i_{J}(X\setminus \sigma_{X}^{-1}(P_{\Lambda}))} = J_{\ess}
\]
is dense in the following set. 
Since $K$ is finite, we must have $J_{\ess} = i_{J}(\sigma_{X}^{-1}(X_{\sync})) = i_{J}(X_{\sync}\cap \sigma_{X}^{-1}(X_{\sync}))$. 
Therefore, \eqref{containment} implies $J_{\ess}\subseteq r^{-1}(i_{I}(X_{\sync}))\cap s^{-1}(i_{I}(X_{\sync}))\subseteq J_{\ess}$.
A similar argument shows $I_{\ess} = i_{I}(X_{\sync})$. 
Therefore, $K_{\ess}$ is the Fischer graph.
\end{proof}
\setlength{\parindent}{0cm} \setlength{\parskip}{0.5cm}

In Section \ref{sec:T-expansive} we gave a condition on the minimal graph cover construction that guaranteed expansivity of the minimal dynamical cover. Let us prove that it is necessary in the case of subshifts.
\setlength{\parindent}{0cm} \setlength{\parskip}{0cm}

\begin{proposition}
    Let $\sigma_{X} \colon X\to X$ be a subshift. Then the following are equivalent.

    \begin{enumerate}
        \item The minimal graph cover terminates after finitely many steps.
        \item The minimal dynamical cover is expanding.
        \item $\sigma_{X}$ is sofic.
    \end{enumerate}

Moreover, if $\sigma_{X}$ is regular then we may replace ``minimal graph cover'' with ``minimal essential graph cover'' in the above statements.
\end{proposition}

\begin{proof}
    $(1)\Rarr (2)$ follows from Theorem \ref{thm:hatTExpansive-check}. 
    $(2) \Rarr (3)$: the dynamical cover is an expanding local homeomorphism (hence conjugate to a shift of finite type), so $\sigma_X$ is sofic. 
    $(3) \Rarr (1)$: note that the spectrum $X_{n}$ of $\mathcal{E}_{n}$ can be represented as the closure of the image of the map $x\mapsto (P_{\leq n}(x), x)$, 
    while the spectrum $Y_{n}$ of $\mathcal{F}_{n}$ is the closure of the image of $x\mapsto (P_{\leq n}(\sigma(x)), x, P_{\leq n}(x))$, 
    with the projections $Y_{n+1}\to Y_{n}$, $X_{n+1}\to X_{n}$ given by the projections $F\subseteq L_{\leq k+1}(X)\to F|_{\leq k}\subseteq L_{\leq k}(X)$ and the identity $X\to X$ on the relevant factors. 
    As soficity is equivalent to the existence of $N\in\mathbb{N}$ such that $P_{\leq N}(x) = P_{\leq N}(y)$ implies $P_{k}(x) = P_{k}(y)$ for any $k\geq N$, $x,y\in X$, it follows from this and the above description of the cover that the maps $X_{k+1}\to X_{k}$ and $Y_{k+1}\to Y_{k}$ are bijections, for all $k\geq N$.
\end{proof}
\setlength{\parindent}{0cm} \setlength{\parskip}{0.5cm}

\begin{example}\label{exm:even}
  Let $(\sigma_X,X)$ be the even shift (only an even number of zeros are allowed between any two ones).
  This is irreducible, and $0^\infty$ is the unique critical point and it is not synchronizing.
  The (topological) Fischer graph has two vertices, unique edges between them both labeled with zero, and a single loop labeled with one.
\end{example}

\begin{example}\label{ex:Sturmian}
  If $(\sigma_X,X)$ is a one-sided Sturmian subshift, then $\sigma_X$ admits a unique branching element $x\in X$ for which $|\sigma_X^{-1}(x)| = 2$.
  The infinite path space on the topological Fischer cover is conjugate to the two-sided (i.e. bi-infinite) Sturmian subshift,
  and the infinite path space on the topological Krieger cover is conjugate to $K_\ess^{-\infty}$ and a discrete copy of the orbit of the branch point $x$ which is dense (see \cite{Brix2023}).
\end{example}

\bremark
  In \cite{Jonoska1996} (see also \cite[Example 3.3.21]{Lind-Marcus2021}), 
  Jonoska exhibits an example of two non-isomorphic labelled graphs that are both minimal and right-resolving and that present the same (reducible) shift of finite type.
  This explains why the Fischer graph defined as \emph{the} minimal right-resolving graph presentation is only defined for \emph{irreducible} sofic subshifts.
  Our topological Fischer graph is uniquely determined, it is discrete for sofic subshifts, and it is trivially labeled (every edge has a unique label) for shifts of finite type,
  so it cannot in general be minimal (in the sense of having the fewest number of vertices) when the subshift is reducible.
\eremark

\subsection{Natural extensions as minimal essential covers of almost one-to-one dynamical systems}

We can generalize Example~\ref{ex:Sturmian} in the following way. 
Let $T \colon X\to X$ be a proper dynamical system. 
Define $X^{\infty} \coloneq \menge{(x_{1}, x_{2},...)\in \prod_{i>0}X}{x_{i} = T(x_{i+1})\text{ }\forall i\in\mathbb{N}}$ and $\tau(x_{1},x_{2},...) \coloneq (T(x_{1}), x_{1},...)$. 
Then $\tau$ is a homeomorphism, known as the \textit{natural extension of $T$}. 
Let us characterize when the essential dynamical cover can be identified with the natural extension in the locally injective case.
\begin{theorem}\label{thm:almost_one_to_one}
    Let $T \colon X\to X$ be a locally injective, proper and finite degree map defined on a locally compact Hausdorff space $X$, weighted by $w = 1$. Then the following are equivalent.
\setlength{\parindent}{0cm} \setlength{\parskip}{0cm}

    \begin{enumerate}
        \item $X_{0} = \menge{x\in X}{|T^{-n}(T^{m}(x))| = 1\text{ }\forall n,m\in\mathbb{N}}$ is dense in $X$.
        \item $T$ is sparsely critically saturated and the minimal essential dynamical cover $T_{\ess}$ is a homeomorphism.
        \item $T$ is sparsely critically saturated and $T_{\ess}$ is conjugate to $\tau$.
    \end{enumerate}
\end{theorem}
\setlength{\parindent}{0cm} \setlength{\parskip}{0cm}

\begin{proof}
$(1) \Rarr (2)$: note that $T$ is surjective, as the image of a proper map is closed. 
Since $T$ is  surjective and locally injective, it follows that if $x\in X_{0}$, then there is a neighbourhood $U$ of $x$ such that $|T^{-1}(\tilde{x})| = 1$ for all $\tilde{x}\in U$. 
Therefore, $\Lambda(f)$ is continuous at $x$, for all $f\in C_{c}(X)$, and $X_{0}\cap C_{\Lambda}=\emptyset$, where $C_{\Lambda}$ is the critical values of $T$ weighted by $w=1$. 
Since $X_{0}$ is fully invariant, we also have that $X_{0}\cap G_{\Lambda} = \emptyset$, where $G_{\Lambda}$ is the saturation of $C_{\Lambda}$. 
By a similar argument, we have that $X^{m,n}_{0} = \menge{x\in X}{|T^{-n}(T^{m}(x))| = 1}$ is open. 
The Baire category theorem now implies that $X_{0} = \bigcap_{m,n\in\mathbb{N}}X^{m,n}_{0}$ is a dense $G_{\delta}$ set contained in $X\setminus G_{\Lambda}$, so $T$ is sparsely critically saturated.
In order to prove that $T_{\ess}$ is a homeomorphism, it suffices to show $|T_{\ess}^{-1}(x)|  = 1$. 
But this follows from the fact that $|T_{\ess}^{-1}(x)| = 1$ for all $x\in \iota(X_{0})$, density of $\iota(X_{0})$ in $X_{\ess}$, and continuity of $x\mapsto |T_{\ess}^{-1}(x)|$.
\setlength{\parindent}{0cm} \setlength{\parskip}{0.5cm}

$(1) \Rarr (3)$: 
consider the factor map $\mu \colon X^{\infty}\to X$ given by $\underline{x} = (x_{1},x_{2},...)\mapsto \mu(\underline{x}) = x_{1}$.
The set $\crX =\menge{\underline{x}\in X^{\infty}}{x_{1}\in X_{0}}$ is fully invariant and $G_{\delta}$, and since $X_{0}$ is fully invariant, a diagonalization argument shows that it is dense. 
Moreover, $\mu$ restricted to $\crX$ is a bijection onto $X_{0}$. It follows that $(\tau, X^{\infty})$ and $\mu$ form an essential dynamical cover of $T$ weighted by $1$. Therefore, by the universal property of essential dynamical covers, there is a continuous equivariant surjection $X^\infty\to X_{\ess}$ that is the identity on $X_{0}$ (once suitably identified). The map $X_{\ess}\to X^{\infty}$ via $\tilde{x}\mapsto (\pi(\tilde{x}), \pi(T_{\ess}^{-1}(\tilde{x})),...)$ is the inverse, since it is also the identity on $X_{0}$.

   $(3) \Rarr (2)$ is trivial, and $(2) \Rarr (1)$ follows from the fact that $\iota(T^{-n}(T^{m}(x))) = T_{\ess}^{-n}(T_{\ess}^{m}(\iota(x)))$ for all $x\in X\setminus G_{\Lambda}$.
\end{proof}
\setlength{\parindent}{0cm} \setlength{\parskip}{0cm}

In particular, we recover the description of Example~\ref{ex:Sturmian}.
This also applies to a large class of minimal subshifts with finitely many branch points (see e.g. \cite{He-Wei2024}).
\setlength{\parindent}{0cm} \setlength{\parskip}{0.5cm}

\begin{remark}
If $X$ is a second countable locally compact Hausdorff space, and $T \colon X\to X$ is locally injective, then using Lemma~\ref{lem:meagre_to_meagre} and the fact that almost one-to-one maps are weakly open, we see that the set $X_{0}$ in Theorem \ref{thm:almost_one_to_one} is dense if and only if $T \colon X\to X$ is almost one-to-one.
\end{remark}

\begin{remark}
If we consider the $C^*$-algebra of the Deaconu--Renault groupoid of $T_{\min}$ for a dynamical system satisfying the hypothesis and one of the three (equivalent) conditions in Theorem \ref{thm:almost_one_to_one}, 
then, since $X_{\ess}$ is a fully invariant sub-space of $X_{\min}$, it will always have a quotient isomorphic to the crossed product of $C(X^{\infty})$ with the integer action of $\mathbb{Z}$ induced by the natural extension of $T$. Moreover, the corresponding ideal will be Morita equivalent to a countable direct sum of copies of the compact operators if $X\setminus X_{0}$ is countably discrete.
\end{remark}

\section{Examples from piecewise invertible maps}

\label{sec:PiecewInv}

\subsection{The tent map}
\label{ss:TentMap}

Let $0 < \mu \leq 2$ and consider the tent map $T \colon [0,1] \to [0,1], \, x \ma \mu \min(x,1-x)$, i.e.,
\begin{equation*}
  T(x)
  =
  \bfa
  \mu x & \falls x \in [0,\frac{1}{2}],\\
  \mu (1-x) & \falls x \in [\frac{1}{2},1].
  \efa
\end{equation*}
Consider the weight $w = 1$. We now set out to describe the minimal cover $T_{\min} \colon X_{\min} \to X_{\min}$ and the minimal essential cover $T_{\ess} \colon X_{\ess} \to X_{\ess}$. In the following, we will write $O(x)$ for the grand orbit of a point $x \in X = [0,1]$, i.e., $O(x) \coloneq \bigcup_{\alpha=0}^{\infty} \bigcup_{\beta=0}^{\infty} T^{-\alpha}(T^{\beta}(x))$.

The case $\mu=2$ is special because this is the only case where $T$ is already an open map. In that case, $X_{\min} = [0,1] \amalg O(\frac{1}{2})$, where $[0,1]$ carries the usual topology, $O(\frac{1}{2})$ is equipped with the discrete topology, and $\amalg$ stands for topological disjoint union. Moreover, on $[0,1]$, $T_{\min}$ is given by $T$, and on $O(\frac{1}{2})$, $T_{\min}$ is given by the restriction of $T$. The new weight takes the constant value $1$ on $([0,1] \setminus \gekl{1}) \amalg O(\frac{1}{2})$ and the value $2$ on $1 \in [0,1]$. We obtain the minimal essential cover from the minimal cover by deleting the extra $O(\frac{1}{2})$-part. In other words, the minimal essential cover agrees with the original tent map, the only difference is that the weight changes its value at the point $1 \in [0,1]$ from its original value $1$ to the new value $2$.

Let us now treat the remaining case where $0 < \mu < 2$. Following the construction described in Section \ref{ss:meas_top_dyn_sys}, we first describe the spaces $X_n = \Spec(\mathcal{B}_{n})$. Each of these spaces is constructed from the unit interval $[0,1]$ by introducing a set of break points $B_n \subseteq [0,1]$ and a set of isolated points $D_n \subseteq [0,1]$. This means that 
\begin{equation*}
  X_n = \Big( ([0,1] \setminus B_n) \cup \menge{x_-, x_+}{x \in B_n} \Big) \amalg D_n,
\end{equation*}
where $D_n$ is given the discrete topology and $\amalg$ stands for topological disjoint union. To explain the topology on $X_n \setminus D_n$, first let $\pi_n \colon X_n \to [0,1]$ be the canonical projection map (which sends both $x_-$ and $x_+$ to $x$ and is the identity on $[0,1] \setminus B_n$ and $D_n$). Given a sequence $(x_i)$ in $X_n \setminus D_n$, $(x_i)$ converges to $x \in [0,1] \setminus B_n$ if and only if $(\pi_n(x_i))$ converges to $\pi_n(x)$; $(x_i)$ converges to $x_-$ for $x \in B_n$ if and only if $(\pi_n(x_i))$ converges to $\pi_n(x)$ and $\pi_n(x_i) < \pi_n(x)$ eventually; and $(x_i)$ converges to $x_+$ for $x \in B_n$ if and only if $(\pi_n(x_i))$ converges to $\pi_n(x)$ and $\pi_n(x_i) > \pi_n(x)$ eventually. Now the sets $B_n$ and $D_n$ are given recursively as follows: $B_1 = \gekl{\frac{\mu}{2}} = D_1$, $B_{n+1} = B_n \cup T(B_n) \cup T^{-1}(B_n \setminus \gekl{\frac{\mu}{2}})$ and $D_{n+1} = D_n \cup T(D_n) \cup T^{-1}(D_n)$.

The space $X_{\min}$ is given by $\plim X_n$, which is constructed in a similar way as the $X_n$ from the unit interval $[0,1]$ by introducing a set of break points $B$ and a set of isolated points $D$. Here $B = \bigcup_{n=1}^{\infty} B_n$ and $D = O(\frac{1}{2})$, i.e., 
\begin{equation*}
  \hat{X} = \Big( ([0,1] \setminus B) \cup \menge{x_-, x_+}{x \in B} \Big) \amalg D.
\end{equation*}
The map $T_{\min}$ coincides with (the restriction of) $T$ on $D$. For $x \in [0,1] \setminus B$, $T_{\min}(x) = T(x) \in [0,1] \setminus B$ unless $x = \frac{1}{2}$ and $\frac{1}{2} \notin B$, in which case $T_{\min}(\frac{1}{2}) = (\frac{\mu}{2})_-$. On $\menge{x_-, x_+}{x \in B}$, $T_{\min}$ is given by
\begin{equation*}
  T_{\min}(x_-)
  =
  \bfa
  (T(x))_- & \falls x \in B, \, x < \frac{1}{2},\\
  (T(x))_+ & \falls x \in B, \, x > \frac{1}{2};
  \efa
\end{equation*}
\begin{equation*}
  T_{\min}(x_+)
  =
  \bfa
  (T(x))_+ & \falls x \in B, \, x < \frac{1}{2},\\
  (T(x))_- & \falls x \in B, \, x > \frac{1}{2};
  \efa
\end{equation*}
and, if $\frac{1}{2} \in B$, $T_{\min}((\frac{1}{2})_-) = T_{\min}((\frac{1}{2})_+) = (\frac{\mu}{2})_-$.
Moreover, the new weight $\hat{w}$ takes the constant value $1$ on all points of $X_{\min}$ except at $(\frac{\mu}{2})_-$, where $\hat{w}((\frac{\mu}{2})_-) = 2$.

The minimal essential cover is constructed from the minimal cover by deleting $D$, i.e., $X_{\ess} = X_{\min} \setminus D$ and $T_{\ess} = T_{\min} \vert_{X_{\ess}}$.

With these descriptions of our covers, we can collect some immediate observations.
\bcor
The cover construction terminates after finitely many steps in the sense of Remark~\ref{rem:hatTExpansive_dyn_version} if and only if $\mu = 1$, in which case $X_{\min} = X_1$. 
In that case, $T_{\min}$ is a local homeomorphism.
We have $B = O(\frac{1}{2})$ if and only if $\frac{1}{2}$ lies in its forward orbit ${\rm O}^+(\frac{1}{2}) \coloneq \menge{T^{\beta}(\frac{1}{2})}{\beta \geq 1}$, and this statement is also equivalent to $\frac{1}{2} \in B$.
\ecor

\bremark
\label{rem:LocalHomeoCover}
The analysis of the example of the tent map motivates the construction of another cover by introducing break points at all points in $O(\frac{1}{2})$. Indeed, let
\begin{equation*}
  \ti{X} \coloneq ([0,1] \setminus O(\frac{1}{2})) \cup \{ x_-, x_+ \colon x \in O(\frac{1}{2})\},
\end{equation*}
with topology defined similarly as for (the non-discrete parts of) $X_n$ and $X_{\min}$. 
Moreover, let us introduce the map $\ti{T} \colon \ti{X} \to \ti{X}$ as follows: 
for $x \in [0,1] \setminus O(\frac{1}{2})$, $\ti{T}(x) \coloneq T(x) \in [0,1] \setminus O(\frac{1}{2})$. On $\menge{x_-, x_+}{x \in O(\frac{1}{2})}$, $\ti{T}$ is given by
\begin{equation*}
  \ti{T}(x_-)
  \coloneq
  \bfa
  (T(x))_- & \falls x \in O(\frac{1}{2}), \, x < \frac{1}{2},\\
  (T(x))_+ & \falls x \in O(\frac{1}{2}), \, x > \frac{1}{2};
  \efa
\end{equation*}
\begin{equation*}
  \ti{T}(x_+)
  \coloneq
  \bfa
  (T(x))_+ & \falls x \in O(\frac{1}{2}), \, x < \frac{1}{2},\\
  (T(x))_- & \falls x \in O(\frac{1}{2}), \, x > \frac{1}{2};
  \efa
\end{equation*}
and $\ti{T}((\frac{1}{2})_-) \coloneq \ti{T}((\frac{1}{2})_+) \coloneq (\frac{\mu}{2})_-$.
It is straightforward to see that $\ti{T} \colon \ti{X} \to \ti{X}$ is an essential cover of $T \colon X \to X$, and that $\ti{T}$ is always a local homeomorphism.
Moreover, if $1 < \mu \leq 2$, then $\ti{T}$ is positively expansive if the forward orbit ${\rm O}^+(\frac{1}{2})$ is finite or, more generally, if there exists $\gamma \geq 0$ such that ${\rm O}^+(\frac{1}{2}) \subseteq \bigcup_{\alpha = 0}^{\infty} \bigcup_{\beta = 0}^{\gamma} T^{- \alpha}(T^{\beta}(\frac{1}{2}))$.
\eremark

\subsection{General piecewise invertible maps}
\label{ss:PiecewInv}

The construction in Remark~\ref{rem:LocalHomeoCover} motivates a cover construction for general piecewise invertible maps, which we set out to explain now.

First, it is convenient to introduce the notions of regular open and regular closed partitions (these have been used before to construct covers, see for example \cite{Kul, KS}). Let $X$ be a compact Hausdorff space. A subset $C \subseteq X$ is called regular closed if $C$ is closed and satisfies $C = \overline{\inte(C)}$. A subset $U \subseteq X$ is called regular open if $U$ is open and $U = \inte(\overline{U})$. There is a one-to-one correspondence between regular closed and regular open subsets under which a regular closed subset $C$ corresponds to the regular open subset $\inte(C)$ (and the inverse sends $U$ to $\overline{U}$). A finite collection $\gekl{C_i}$ of subsets is called a regular closed partition of $X$ if all $C_i$ are non-empty and regular closed, $X = \bigcup_i C_i$ and for all $C, C' \in \gekl{C_i}$ with $C \neq C'$, we have $C \cap C' \subseteq \bd(C) \cap \bd(C')$. The latter condition is equivalent to asking that $\gekl{\inte(C_i)}$ are pairwise disjoint. A finite collection $\gekl{U_i}$ of subsets is called a regular open partition of $X$ if the $U_i$ are non-empty and regular open, $X = \bigcup_i \overline{U_i}$ and $\gekl{U_i}$ are pairwise disjoint. There is a one-to-one correspondence between regular closed and regular open partitions sending a regular closed partition $\gekl{C_i}$ to $\gekl{\inte(C_i)}$ and a regular open partition $\gekl{U_i}$ to $\{ \overline{U_i} \}$.
\setlength{\parindent}{0.5cm} \setlength{\parskip}{0cm}

Given regular open partitions $\cU = \gekl{U_i}_{i \in I}$ and $\cV = \gekl{V_j}_{j \in J}$, the common refinement $\cU \vee \cV$ is given by $\cU \vee \cV \coloneq \menge{U_i \cap V_j}{i \in I, \, j \in J, \, U_i \cap V_j \neq \emptyset}$. Moreover, we say that a regular open partition $\cV = \gekl{V_j}_{j \in J}$ is finer than (or refines) another regular open partition $\cU = \gekl{U_i}_{i \in I}$ if for all $V_j \in \cV$ there exists (a necessarily unique) $U_i \in \cU$ with $V_j \subseteq U_i$. In that case, we write $\cU \preceq \cV$. We define a similar refinement relation for regular closed partitions. In other words, given two regular closed partitions $\cC$ and $\cD$ with corresponding regular open partitions $\cU$ and $\cV$, we have $\cC \preceq \cD$ if and only if $\cU \preceq \cV$.
\setlength{\parindent}{0cm} \setlength{\parskip}{0.5cm}

With these preparations, let us now introduce the transformations of interest. Let $X$ be a compact Hausdorff space. A continuous map $T \colon X \to X$ is called \emph{piecewise invertible} if there exists a regular open partition $\cU = \gekl{U_i}_{i \in I}$ with corresponding regular closed partition $\cC = \gekl{C_i}_{i \in I}$ such that $T$ restricts to a homeomorphism $T_i \colon C_i \to T(C_i), \, x \ma T(x)$. Moreover, we require that $T(U_i) = \inte(T(C_i))$. In particular, $T(C_i)$ is regular closed for all $i \in I$. In addition, we call such a piecewise invertible map $T$ \emph{expansive} if there exists a metric $d$ on $X$ and a constant $c > 1$ such that for all $i \in I$ and $x, x' \in C_i$, we have $d(T(x),T(x')) \geq c d(x,x')$. Note that in general, a piecewise invertible map need not be open, and even if it is expansive in our sense, it might not be positively expansive in the usual sense. The tent map from \S~\ref{ss:TentMap} is an example of a piecewise invertible map, and it is actually expansive in our sense. For similar notions of (expansive) piecewise invertible maps, the reader may consult \cite{Hof,Tsu} and the references therein.

Let us now construct covers for piecewise invertible maps, motivated by the construction in Remark~\ref{rem:LocalHomeoCover}. Let $T \colon X \to X$ be a piecewise invertible map with a fixed regular open partition $\{U_i\}_{i\in I}$. We first need to introduce image and pre-image of regular open partitions under $T$. Let $\cV = \gekl{V_j}_{j \in J}$ be a regular open partition of $X$. Fix $i \in I$ and let $\cT_i \coloneq \menge{T(U_i \cap V_j)}{j \in J, \, U_i \cap V_j \neq \emptyset} \cup \gekl{X \setminus T(C_i)}$. It is straightforward to see that $\cT_i$ is a regular open partition of $X$. Now we define $T(\cV) \coloneq \bigvee_{i \in I} \cT_i$. Note that $T(\cU) = \bigvee_{i \in I} \gekl{T(U_i), \, X \setminus T(C_i)}$. Moreover, set $T^{-1}(\cV) \coloneq \menge{T_i^{-1}(O)}{O \in \cV \vee T(\cU), \, i \in I, \, O \subseteq T(U_i)}$. Note that for every $W \in T^{-1}(\cV)$, there exists a unique $V_j \in \cV$ with $T(W) \subseteq V_j$.

Now define the regular open partition $\cU_{m,n} \coloneq \bigvee_{0 \leq \alpha \leq m, \, 0 \leq \beta \leq n} T^{- \alpha}(T^{\beta}(\cU))$. Let $\cC_{m,n}$ be the corresponding regular closed partition. If $m' \geq m$ and $n' \geq n$, then $\cU_{m,n} \preceq \cU_{m',n'}$, so that we obtain a map $\cU_{m',n'} \to \cU_{m,n}$ sending $U'$ the the unique $U \in \cU_{m,n}$ with $U' \subseteq U$. With the partial order $(m',n') \geq (m,n)$ if $m' \geq m$ and $n' \geq n$, we get an inverse system and hence can form the limit $\bar{X} \coloneq \plim \, \cU_{m,n}$. Elements $\bmU \in \bar{X}$ are of the form $\bmU = (U_{m,n})$ with $U_{m,n} \in \cU_{m,n}$. Define 
\begin{equation*}
  X' \coloneq \menge{(x,\bmU) \in X \times \bar{X}}{\bmU = (U_{m,n}), \, x \in \overline{U_{m,n}} \ \forall \ m, n}.
\end{equation*}
Denote by $\pi$ the canonical projection $X' \to X, \, (x,\bmU) \ma x$. Moreover, define $\bar{T} \colon \bar{X} \to \bar{X}$ by $\bar{T}((U_{m,n})) = \bmV = (V_{m,n})$, where $V_{m,n}$ is the unique element of $\cU_{m,n}$ with $T(U_{m+1,n}) \subseteq V_{m,n}$. In addition, we define $T' \colon X' \to X', \, (x, \bmU) \ma (Tx, \bar{T}(\bmU))$.

It is straightforward to see that $T'$ is a cover of $T$ (a cover map is given by $\pi$). 

\bprop
\label{prop:tiTLocalHomeo}
Let $T$ be a piecewise invertible map. Then the map $T'$ is a local homeomorphism.
\eprop
\setlength{\parindent}{0cm} \setlength{\parskip}{0cm}

\bproof
Let $\mu, \nu\in\Nz$ with $\mu \geq 1$. Take $U \in \cU_{\mu, \nu}$. Let
\begin{equation*}
  \Vz \coloneq \menge{V \in \cU_{\mu - 1, \nu + 1}}{\exists \ W \in \cU_{\mu, \nu + 1} \text{ with } W \subseteq U, \, T(W) \subseteq V}.
\end{equation*}
Set $Z_U \coloneq \menge{(x,(U_{m,n})) \in X'}{U_{\mu,\nu} = U}$ and $Z_{\Vz} \coloneq \menge{(y,(V_{m,n})) \in X'}{V_{\mu - 1, \nu + 1} \in \Vz}$. From the definition of $T'$ and $Z_{\mathbb{V}}$, it is clear that $T'(Z_{U})\subseteq Z_{\mathbb{V}}$. We construct a continuous inverse to $T' \colon Z_{U}\to Z_{\mathbb{V}}$. Let $i \in I$ be the unique index such that $U_{\mu,\nu} \subseteq U_i$. Then, by definition of $\mathbb{V}$, we have $T(U_{i})\cap V_{\mu-1, \nu + 1}\neq\emptyset$ for $(y,(V_{m,n}))\in Z_{\mathbb{V}}$. Since each $V_{m,n}$ for $m\geq 0, n\geq 1$ lies in a refinement of $T(\mathcal{U})$, we have that $V_{m,n} \subseteq T(U_i)$. Now it is straightforward to check that the following map
\begin{equation*}
  Z_{\Vz} \to X', \, (y,(V_{m,n})) \ma (T_i^{-1}(y), (W_{m,n})),
\end{equation*}
where $W_{m,n} \in \cU_{m,n}$ is uniquely determined by the requirement that $T_i^{-1}(V_{m-1,n+1}) \subseteq W_{m,n}$ for all $m, n \in \Nz$ with $m \geq 1$, is continuous, maps $Z_{\mathbb{V}}$ into $Z_{U}$ and is the inverse of the restriction of $T'$ to $Z_U$.
\eproof 
\setlength{\parindent}{0cm} \setlength{\parskip}{0.5cm}

\blemma
\label{lem:tiT=barT}
If $T$ is an expansive piecewise invertible map, then the canonical projection $X' \to \bar{X}, \, (x,\bmU) \ma \bmU$ is a homeomorphism. In particular, $T'$ and $\bar{T}$ are conjugate.
\elemma
\setlength{\parindent}{0cm} \setlength{\parskip}{0cm}

\bproof
It suffices to prove that the map is injective. So assume that $(x,\bmU)$ and $(x',\bmU)$ are elements of $X'$. That means that for every $\beta \in \Nz$ there exists $i_{\beta} \in I$ such that $T^{\beta}(x)$ and $T^{\beta}(x')$ both lie in $\overline{U_{i_{\beta}}}$. Then it follows that $d(T^{\beta}(x),T^{\beta}(x')) \geq c^{\beta} d(x,x')$ for all $\beta \in \Nz$. But $X$ is compact, so that $\sup_{\beta} c^{\beta} d(x,x') \leq \sup_{\beta} d(T^{\beta}(x),T^{\beta}(x')) < \infty$. It follows that $d(x,x') = 0$, i.e., $x = x'$.
\eproof
\setlength{\parindent}{0cm} \setlength{\parskip}{0.5cm}

\bprop
\label{prop:PiecewInv-SFT}
Let $T$ be an expansive piecewise invertible map. Assume that there exists $\gamma \in \Nz$ such that for all $m,n \in \Nz$ there exists $\alpha_{m,n} \in \Nz$ such that $\cU_{m,n} \preceq \bigvee_{0 \leq \alpha \leq \alpha_{m,n}, \, 0 \leq \beta \leq \gamma} T^{- \alpha}(T^{\beta}(\cU))$. Then $\bar{T}$ is conjugate to a subshift. Thus $T'$ is conjugate to a shift of finite type, which is at the same time a cover of $T$.
\eprop
\setlength{\parindent}{0cm} \setlength{\parskip}{0cm}

\bproof
Let $\cV \coloneq \bigvee_{0 \leq \beta \leq \gamma} T^{\beta}(\cU)$ and $\pi_{\cV}$ the canonical projection $\bar{X} \onto \cV$. Then the map
\begin{equation*}
  \bar{X} \to \prod_{\alpha = 0}^{\infty} \cV, \, \bmU \ma (\pi_{\cV}(\bar{T}^{\alpha}(\bmU))_{\alpha}
\end{equation*}
is a continuous map which is injective by our assumption. Therefore, we obtain an identification of $\bar{X}$ with the image of this map. Under this identification, it is clear that the map $\bar{T}$ corresponds to the one-sided shift. In other words, this shows that $\bar{T}$ is conjugate to a subshift. 
\setlength{\parindent}{0cm} \setlength{\parskip}{0.5cm}

Now Lemma~\ref{lem:tiT=barT} implies that $T'$ is conjugate to $\bar{T}$, so that $T'$ is conjugate to a subshift, too. Since it is also a local homeomorphism by Proposition~\ref{prop:tiTLocalHomeo}, it follows that $T'$ is conjugate to a shift of finite type.
\eproof
\setlength{\parindent}{0cm} \setlength{\parskip}{0.5cm}

\begin{lemma}
    \label{lem:ess_cover_ex}
If $T$ is an expansive piecewise invertible map, then $T'$ is an essential cover of $T$ (weighted by $w =1$).
\end{lemma}
\setlength{\parindent}{0cm} \setlength{\parskip}{0cm}

\begin{proof}
    Recall that the covering map $\pi \colon X'\to X$ is defined by $(x,\bmU)\mapsto x$. By expansivity, the diameters of the elements in $\mathcal{U}_{m,n}$ shrink to zero uniformly as $m,n$ tend to infinity. Hence, for $x$ in the dense $G_{\delta}$ set $G' \coloneq \bigcap_{m,n}\bigcup_{U\in \mathcal{U}_{m,n}}U$, there is a unique pre-image element $\iota(x) \coloneq (x,\bmU)\in\pi^{-1}(x)$, and it is straightforward to see that $\pi^{-1}(G')$ is dense. Let $P' \coloneq X\setminus G'$. Since $T$ is piecewise invertible, it maps meagre sets to meagre sets, and is weakly open. Therefore, $P \coloneq \bigcup_{k\geq 0, l\geq 0} T^{-k}(T^{l}(P'))$ is a meagre and fully invariant subset containing $P'$. Hence $G \coloneq X\setminus P$ is a co-meagre and fully invariant subset contained in $G'$, so by density of $\pi^{-1}(G')$ and continuity of $\iota \colon G'\to X'$ (Lemma \ref{continuity_points}), it follows that $\pi^{-1}(G)$ is also co-meagre and fully invariant. Therefore, $T'$ is an essential cover for $T$.
\end{proof}

By the universal property of the minimal essential cover (Theorem \ref{thm:universal_ess_dyn}) and Remark \ref{rem:redundant_dyn_measure}, it follows that the minimal essential cover of $T$ is a factor of $T'$.
The corollary below follows immediately from Lemma \ref{lem:ess_cover_ex} and Remark \ref{rem:TessTopTrans}.

\bcor
\label{cor:TopTrans}
If $T$ is an expansive piecewise invertible map that is topologically transitive, then $T'$ is topologically transitive.
\ecor
\setlength{\parindent}{0cm} \setlength{\parskip}{0.5cm}

\bcor
\label{cor:PiecewInvNiceCover}
Let $T$ be an expansive piecewise invertible map which is topologically transitive. Assume that there exists $\gamma \in \Nz$ such that for all $m,n \in \Nz$ there exists $\alpha_{m,n} \in \Nz$ such that $\cU_{m,n} \preceq \bigvee_{0 \leq \alpha \leq \alpha_{m,n}, \, 0 \leq \beta \leq \gamma} T^{- \alpha}(T^{\beta}(\cU))$. Then our cover $T'$ of $T$ is (conjugate to) a topologically transitive shift of finite type.
\ecor

\bremark
As in Remark~\ref{rem:TDF-FiniteSteps}, we conclude that for piecewise invertible maps (and their minimal essential covers) satisfying the assumptions in Corollary~\ref{cor:PiecewInvNiceCover}, the thermodynamic formalism holds.
\eremark

\subsection{Williams pre-solenoids}
We present an example class of dynamical systems constructed from substitutions that are used to define the solenoids of \cite{RWil}.
We call them \textit{Williams pre-solenoids} and show that the minimal essential dynamical cover (in the sense of Section \ref{ss:ess_dyn_cover}, for the weight $w=1$) of a Williams pre-solenoid is equal to a naturally defined shift of finite type from the associated substitution. As we will explain, Williams pre-solenoids are particular examples of expansive piecewise invertible maps in the sense of Section \ref{ss:PiecewInv}.

Let $\mathcal{A}$ be a finite set and denote by $\mathcal{A}^{*} = \coprod_{n\in\mathbb{N}}\mathcal{A}^{n}$ the collection of words in $\mathcal{A}$. Given a word $\alpha = \alpha_{1}...\alpha_{n}\in\mathcal{A}^{n}$, we denote by $|\alpha| = n$ its length and $\alpha_{i}$ its $i^{th}$ entry, $i\leq n$. A \textit{substitution} is a function $S\colon\mathcal{A}\to \mathcal{A}^{*}$. We can extend the definition of the substitution to a function $S\colon\mathcal{A}^{*}\to\mathcal{A}^{*}$ by the formula $S(\alpha_{1}...\alpha_{n}) = S(\alpha_{1})...S(\alpha_{n})$. The Williams pre-solenoid of a substitution $S$ is constructed as follows.
\setlength{\parindent}{0.5cm} \setlength{\parskip}{0cm}

Consider the wedge sum $X_{\mathcal{A}} = \bigvee_{\alpha\in \mathcal{A}}\mathbb{R}/\mathbb{Z}\times\{\alpha\}$ of $\mathcal{A}$ copies of the circle, glued together at the origins $(0,\alpha)$, $\alpha\in\mathcal{A}$. Define a function $T_{S} \colon X_{\mathcal{A}} \to X_{\mathcal{A}}$ by
\begin{equation*}
  T_{S}(t,\alpha)
  \coloneq (|S(\alpha)|t, S(\alpha)_{i})\text{ whenever } t \in \big[ \frac{i-1}{|S(\alpha)|}, \frac{i}{|S(\alpha)|}\big].
\end{equation*}
Since sufficiently small sub-intervals of each individual circle are mapped homeomorphically onto their images under $T_S$, it follows that $T_S$ is a piecewise invertible map in the sense of Section \ref{ss:PiecewInv}. Such a system has a natural cover by a shift of finite type. Let $E^{1}_{S} \coloneq \{(\alpha,i)\in \mathcal{A}\times\mathbb{N} \colon i\leq |S(\alpha)|)$
and $E^{0}_{S} = \mathcal{A}$. We define range and source maps $r,s \colon E^{1}_{S} \to E^{0}_{S}$ by $r(\alpha,i) = \alpha$ and $s(\alpha,i) =  S(\alpha)_{i}$. The adjacency matrix of this graph is the well-known \textit{substitution matrix} of a substitution. Let $(\sigma, E_S^{\infty})$ be the corresponding shift of finite type. Let $\pi \colon E_{S}^{\infty} \to X_{\mathcal{A}}$ be defined for $x = ((\alpha_{1}, i_{1})(\alpha_{2}, i_{2})...)\in E^{\infty}_{S}$ as 
\begin{equation*}
\pi(x) \coloneq  \big(\sum^{\infty}_{n=1}\frac{i_{n}-1}{|S(\alpha_{1}...\alpha_{k})|}, \alpha_{1}\big). 
\end{equation*}
It is easy to see that $\sigma\circ\pi = \pi\circ T_{S}$. Since $|\pi^{-1}(x)| = 1$ except possibly at the (invariant set of) rational points and $\pi$ is weakly open, it follows that $\pi$ is an essential cover of $T_{S}$. Moreover, it can be seen as the cover constructed from the regular closed partition consisting of intervals $\{( \big[ \frac{i-1}{|S(\alpha)}, \frac{i}{|S(\alpha)} \big], \alpha)\}_{\alpha\in\mathcal{A}, i\leq |S(\alpha)|}$. Let us sketch the argument showing $\pi$ is in fact the minimal essential cover of $T_{S}$ when the substitution satisfies the following (mild) condition: 
for every $\alpha\in\mathcal{A}$, there are $x_{1},x_{2},y_{1},y_{2}\in\mathcal{A}$ such that $x_{1}\neq x_{2}$, $y_{1}\neq y_{2}$ and the words $x_{1}\alpha$, $x_{2}\alpha$, $\alpha y_{1}$, $\alpha y_{2}$ appear as subwords in the words $\bigcup_{k\in\mathbb{N}}S^{k}(\mathcal{A})$ with their first and last letters deleted.
\setlength{\parindent}{0cm} \setlength{\parskip}{0.5cm}

Let $\fX = X_{\mathcal{A}}\setminus (\bigvee_{\alpha\in \mathcal{A}}\mathbb{Q}/\mathbb{Z}\times\{\alpha\})$ and $i = \pi^{-1} \colon  \fX \to E_{S}^{\infty}$. 
Then, $i^{*}$ maps $C(E^{\infty}_{S})$ into $\ell^{\infty}_{\ess}(X_{\mathcal{A}})$, the collection of bounded functions on $X_{\mathcal{A}}$ modulo the functions that are zero on a co-meagre set (in this case, the set of irrational points). 
The minimal essential cover is now the Gelfand spectrum of the image $\mathcal{B}_{\ess}$ of the minimal cover $\mathcal{B} \subseteq \ell^{\infty}(X_{\mathcal{A}})$ mapped into $\ell^{\infty}_{\ess}(X_{\mathcal{A}})$. To show $\pi$ is isomorphic to the minimal essential cover, it suffices, by the universal property of the minimal essential cover, to show $i^{*}(C(E^{\infty}_{S}))\subseteq \mathcal{B}_{\ess}$. Since $i^{*}$ is a $^*$-homomorphism, it suffices to show $i^{*}(\chi_{e})\in \mathcal{B}_{\ess}$ for every characteristic function $\chi_{e}$ supported on the cylinder set of $e$ in $E_{S}^{n}$. Since $\chi_{e} = \chi_{e_{1}}\cdot (\chi_{e_{2}}\circ\sigma) \cdot \dotso \cdot (\chi_{e_{n}}\circ\sigma^{n-1})$, we have $i^{*}(\chi_{e}) = i^{*}(\chi_{e_{1}})\cdot (i^{*}(\chi_{e_{2}})\circ T_{S}) \cdot \dotso \cdot (i^{*}(\chi_{e_{n}})\circ T_{S}^{n-1})$, so it suffices to show $i^{*}(\chi_{(\alpha,i)})\in \mathcal{B}_{\ess}$ for all $(\alpha,i)\in E^{1}_{S}$.
\setlength{\parindent}{0.5cm} \setlength{\parskip}{0cm}

We have $i^{*}(\chi_{(\alpha,i)}) = \chi_{ \big( \big[ \frac{i-1}{|S(\alpha)|},\frac{i}{|S(\alpha)|} \big],\alpha \big) }$. By the assumed condition, there are $k_{1}, k_{2}\in\mathbb{N}$ and $\alpha_{1},\alpha_{2},y_{1},y_{2}\in\mathcal{A}$ such that $y_{1}\neq y_{2}$, and $\alpha y_{1}$, $\alpha y_{2}$ appear in the words $S^{k_{1}}(\alpha_{1})$, $S^{k_{2}}(\alpha_{2})$ with their first and last letters deleted. Therefore, for any $\delta < 1$, there is an $\varepsilon > 0$ and $q_{1},q_{2}\neq 0$ in $\mathbb{Q}/\mathbb{Z}$ such that 
$((q_{1}-\varepsilon, q_{1}+\varepsilon),\alpha_{1})$, $(q_{2}-\varepsilon, q_{2} +\varepsilon),\alpha_{2})$ are open neighbourhoods in $X_{\mathcal{A}}$ and $T_{S}^{k_{1}}(((q_{1}-\varepsilon, q_{1}+\varepsilon),\alpha_{1}))\cap T^{k_{2}}_{S}((q_{2}-\varepsilon, q_{2} +\varepsilon),\alpha_{2})) = ([0,\delta),\alpha)$. So, for any $\delta' < \delta < 1$ we may choose $\phi_{1},\phi_{2}$ in $C(X_{\mathcal{A}})$ such that $f_{\alpha} = \Lambda^{k_{1}}(\phi_{1})\Lambda^{k_{2}}(\phi_{2})\leq 1$, has open support contained in $([0,\delta),\alpha)$, is equal to $1$ on $([0,\delta'],\alpha)$ and is continuous as a function on $([0,\delta),\alpha)$.

Similarly, by the assumed condition, for any $0 < \varepsilon <\varepsilon' < 1$ we may choose $\psi_{1},\psi_{2}$ in $C(X_{\mathcal{A}})$ and $l_{1}, l_{2}$ such that $g_{\alpha} = \Lambda^{l_{1}}(\psi_{1})\Lambda^{l_{2}}(\psi_{2})\leq 1$, has open support contained in $((\varepsilon,1],\alpha)$ and is equal to $1$ on $([\varepsilon',1],\alpha)$, and is continuous as a function on $([\varepsilon,1),\alpha)$.

By choosing $\delta,\delta',\varepsilon,\varepsilon'$ appropriately, $f_{\alpha} + g_{\alpha} = h\chi_{(\mathbb{R}/\mathbb{Z}\times\{\alpha\})}$ for some invertible $h\in C(X_{\mathcal{A}})$, hence $\chi_{(\mathbb{R}/\mathbb{Z}\times\{\alpha\})}\in \mathcal{B}_{\ess}$. Now, for $i\leq |S(\alpha)|$, and any $0 < \delta' < \delta < 1$, $(f_{\alpha_{i}}\circ T_{S} )\chi_{(\mathbb{R}/\mathbb{Z},\alpha)} = \sum_{j \colon  \, S(\alpha)_{j} = \alpha_{i}}f_{j}$, where for each $j$, $\supp(f_{j})\subseteq \big( \big[ \frac{j-1}{|S(\alpha)|}, \frac{j-1 + \delta}{|S(\alpha)|} \big),\alpha \big)$, $f_{j}\leq 1$, $f_{j} = 1$ on $\big[ \frac{j-1}{|S(\alpha)|}, \frac{j-1 + \delta'}{|S(\alpha)|} \big]$, and is continuous as a function on $\big( \big[ \frac{j-1}{|S(\alpha)|}, \frac{j-1 + \delta}{|S(\alpha)|} \big),\alpha \big)$. In particular, $f_{i}$ is in $\mathcal{B}_{\ess}$. Similarly, for any $0 < \varepsilon <\varepsilon' < 1$, there is $g_{i}$ in $\mathcal{B}_{\ess}$ satisfying $\text{supp}(g_{i})\subseteq ([\frac{i-\varepsilon}{|S(\alpha)|}, \frac{i}{|S(\alpha)|}],\alpha)$, $g_{i}\leq 1$, $g_{i} = 1$ on $\big[ \frac{i-\varepsilon'}{|S(\alpha)|}, \frac{i}{|S(\alpha)|} \big]$, and is continuous as a function on $\big( \frac{i-\varepsilon}{|S(\alpha)|}, \frac{i}{|S(\alpha)|} \big]$. Then, by choosing $\delta,\delta',\varepsilon,\varepsilon'$ appropriately, we have 
$h^{-1}(f_{i} + g_{i}) =\chi_{ \big( \big[ \frac{i-1}{|S(\alpha)|},\frac{i}{|S(\alpha)|} \big],\alpha \big) }$ for some invertible $h\in C(X_{\mathcal{A}})$. This finishes the proof that $\pi$ is isomorphic to the minimal essential cover.
\setlength{\parindent}{0cm} \setlength{\parskip}{0.5cm}

\bremark
These examples show that a dynamical system need not be positively expansive, even if its minimal or minimal essential cover is positively expansive. Non-sofic subshifts provide examples where the original system is positively expansive but our covers are not. So in general there is no relationship between positive expansivity of the minimal or minimal essential cover and positive expansivity of the original system.
\eremark

\section{Examples from semi-\'etale groupoids}

In this section, we consider a locally compact Hausdorff groupoid $G$ with unit space $G^{(0)} = X$ and range and source maps $r,s\colon G\to X$.
We assume that $G$ is semi-\'etale in the sense that $r$ (and hence also $s$) is locally injective.
We view $G$ as a topological graph $r,s\colon G \to X$ with the canonical weight $w=1$, system of measures $\lambda_x = \sum_{g \in r^{-1}(x)} \delta_g$ and transfer operator $\Lambda$.
We study its minimal (graph) cover $\hat{r},\hat{s}\colon \hat{G} \to \hat{X}$ as introduced in Section \ref{sec:app_min_cover} and observe that the process of constructing this minimal cover as explained in \S~\ref{ss:DescCovers} terminates after a single step, and moreover admits a canonical {\'e}tale groupoid structure.

First we need the following terminology: 
every element $g \in G$ induces a bijection $G^{s(g)} = r^{-1}(s(g)) \to G^{r(g)} = r^{-1}(r(g)), \, \gamma \ma g \gamma$. 
Now, given a measure $\mu$ with $\supp (\mu) \subseteq G^{s(g)}$, we let $g.\mu$ denote the pushforward of $\mu$ with respect to this bijection (so $\supp (g.\mu) \subseteq G^{r(g)}$).

\blemma
\label{lem:gtogANDlrgtomu}
If $g_i$ converges to $g$ in $G$ and $\lambda_{r(g_i)}$ converges to $\mu$ in $\cM(G)$, then $\lambda_{s(g_i)}$ converges to $g^{-1}. \mu$.
\elemma
\setlength{\parindent}{0cm} \setlength{\parskip}{0cm}

\bproof
Since $r$ is locally injective, the limit points of the measures $\lambda_{x}$, $x\in X$, are of the form $\mu = \sum_{y\in S\subseteq r^{-1}(x)}\delta_{y}$. Therefore, to prove the lemma it suffices to show $g^{-1} \gekl{\text{limit points of } r^{-1}(r(g_i))} = \gekl{\text{limit points of } r^{-1}(s(g_i))}$. If $y$ is a limit point of $r^{-1}(r(g_i))$ then there exists a net $(y_i)$ with $r(y_i) = r(g_i)$ such that $y = \lim y_i$. This implies that $g^{-1}y = \lim g_i^{-1} y_i$ is the limit point of $g_i^{-1} r^{-1}(r(g_i)) = r^{-1}(s(g_i))$. Conversely, if $y$ is a limit point of $r^{-1}(s(g_i))$, then there exists a net $(y_i)$ such that $y = \lim y_i$ and $r(y_i) = s(g_i)$. Hence $gy = \lim g_i y_i$ is a limit point of $g_i r^{-1}(s(g_i)) = r^{-1}(r(g_i))$. 
\eproof
\setlength{\parindent}{0cm} \setlength{\parskip}{0.5cm}

\begin{lemma} \label{lem:terminates-single-step}
  The cover construction of $r,s\colon G\to X$ terminates after a single step, i.e. $\hat{G} = G_1$ and $\hat{X} = X_1$.
\end{lemma}
\setlength{\parindent}{0cm} \setlength{\parskip}{0cm}

\begin{proof}
  Recall the terminology from Section \ref{measured_graphs}.
  Let $\cF'$ be the closure of ${}^*\text{-}{\rm alg}\big( C_c(G), C_c(G) r^* \Lambda(C_c(G)) \big)\subseteq \ell^{\infty}_{c}(G)$ in the inductive limit topology of $\ell_c^\infty(G)$ and let $G'$ be the spectrum of $\mathcal{F}'$.
  Clearly, $\mathcal{F}' \subseteq \mathcal{F}_1$. 
  Let $X_1$ be the spectrum of $\mathcal{E}_1$. 
  From the description in Section \ref{ss:DescCovers}, we identify $G_1$ with the closure of $\menge{(g,\lambda_{r(g)}, \lambda_{s(g)})}{ g\in G }$ in $G\times M(G)\times M(G)$,
  and a similar method shows $G'$ can be identified with the closure of $\menge{ (g,\lambda_{r(g)})}{ g\in G }$ in $G\times M(G)$.
  The range and source maps extend to $r_1,s_1\colon G_1 \to X_1$, and by construction it is clear that $r_1$ factors through to a continuous map $r'\colon G' \to X_1$.
  Similarly, $s_1$ factors through to a continuous map $s'\colon G' \to X_1$ given by $s'(g,\mu) = (s(g), g^{-1}.\mu)$, for all $(g,\mu) \in G'$ 
  (Lemma \ref{lem:gtogANDlrgtomu} ensures that this is well-defined).
  It follows from Gelfand duality that $C_c(G) s^*\Lambda(C_c(G)) \subseteq \mathcal{F}'$, so $\mathcal{F}_1 = \mathcal{F}'$.
  It follows that $\mathcal{F} = \mathcal{F}'$ and $\mathcal{E} = \mathcal{E}_1$, so $\hat{G} = G'$ and $\hat{X} = X_1$.
\end{proof}
\setlength{\parindent}{0cm} \setlength{\parskip}{0.5cm}

\begin{proposition} 
  Let $G$ be a locally compact Hausdorff and semi-\'etale groupoid.
  The cover $\hat{r},\hat{s}\colon \hat{G} \to \hat{X}$ admits an \'etale groupoid structure which extends the given groupoid structure on $G$.
\end{proposition}
\setlength{\parindent}{0cm} \setlength{\parskip}{0cm}

\begin{proof}
  We saw in the proof of Lemma \ref{lem:terminates-single-step} that $\hat{G}$ is (homeomorphic to) the closure of $\menge{ (g,\lambda_{r(g)})}{ g\in G }$ in $G\times M(G)$
  and $\hat{X}$ is (homeomorphic to) the closure of $\menge{ (x,\lambda_x) }{ x\in X }$ in $X\times M(G)$.
  For $(g,\mu) \in \hat{G}$ we have $\supp(\mu) \subseteq r^{-1}(r(g))$ and for $(x,\mu)\in \hat{X}$ we have $\supp(\mu) \subseteq r^{-1}(x)$.
  The range and source maps are $\hat{r}(g,\mu) = (r(g), \mu)$ and $\hat{s}(g,\mu) = (s(g), g^{-1}\mu)$ for all $(g,\mu)\in \hat{G}$.
  The product of $(g,\mu)$ and $(h,\nu)$ is defined precisely if $g$ and $h$ are composable and $g^{-1}\mu = \nu$ in which case $(g,\mu)(h,\nu) = (gh, \mu)$,
  and inversion is given by $(g,\mu)^{-1} = (g^{-1}, g^{-1}\mu)$.
\setlength{\parindent}{0.5cm} \setlength{\parskip}{0cm}

  As explained in Example \ref{example:local_injections}, $\hat{r}$ is a local homeomorphism. Therefore, $\hat{G}$ is an \'etale groupoid.
\end{proof}
\setlength{\parindent}{0cm} \setlength{\parskip}{0.5cm}

\begin{remark}[Universality and lifting properties]\label{rmk:universal_property}
    Recall that a groupoid homomorphism $\pi \colon G\to H$ is called \emph{$r$-bijective} if $\pi \colon G^{x}\to H^{\pi(x)}$ is a bijection, for all $x\in G^{(0)}$. This is equivalent to $\pi$ being a morphism as a weighted topological graph when weighted by $w=1$. It is easy to see $\iota \colon G\to \hat{G}$ is a groupoid homomorphism; under our identifications it is given by $g\mapsto (g,\lambda_{r(g)})$. Moreover, since it is a measured morphism, it is also $r$-bijective.
\setlength{\parindent}{0.5cm} \setlength{\parskip}{0cm}

    The minimal cover of a semi-\'etale groupoid has the following universal property: given an \'etale groupoid $\dot{G}$ that covers $G$ in the sense that there is a continuous homomorphism $\dot{\pi} \colon \dot{G}\to G$ and an $r$-bijective homomorphism $\dot{\iota} \colon G\to \dot{G}$ with dense image such that $\dot{\pi}\circ \dot{\iota} = \text{id}_{G}$, there is a unique continuous and $r$-bijective homomorphism $\phi \colon \dot{G}\to \hat{G}$ such that $\phi \circ \dot{\iota} = \iota$ and $\pi\circ \phi = \dot{\pi}$.

    Moreover, the cover satisfies the following lifting property: If $\varphi \colon G\to H$ is a continuous and $r$-bijective homomorphism between semi-\'etale groupoids, then there is a unique continuous and $r$-bijective groupoid homomorphism $\hat{\varphi} \colon \hat{G}\to \hat{H}$ satisfying $\hat{\varphi} \circ \iota_G = \iota_H \circ\varphi$ and hence $\varphi \circ \pi_{\hat{G}} = \pi_{\hat{H}} \circ \varphi$.

    These properties follow from the usual universal property and lifting property of the minimal topological graph cover since the identification of compactly supported functions on the covers as functions in $\ell^{\infty}_{c}(G)$ preserves convolution.
\end{remark}
\setlength{\parindent}{0cm} \setlength{\parskip}{0.5cm}

Next, we identify the {\'e}tale cover groupoid $\hat{G}$ as the transformation groupoid of a canonical action of $G$ on $\hat{X}$.
We refer the reader to \cite[Chapter 2.1]{Wil} for details about groupoid actions.

\begin{proposition} 
  There is an action of $G$ on $\hat{X}$ given by $g.(x,\mu) \coloneqq (r(g), g.\mu)$, for all $g\in G$ and $(x,\mu)\in \hat{X}$ with $s(g) = x$.
  Moveover, there is an isomorphism of topological groupoids $\hat{G}\to G\ltimes \hat{X}$ given by $(g,\mu) \mapsto (g, (s(g),g^{-1}.\mu))$, for all $(g,\mu)\in \hat{G}$.
\end{proposition}
\setlength{\parindent}{0cm} \setlength{\parskip}{0cm}

\begin{proof}
  We verify that $G$ acts on $\hat{X}$.
  The map $\rho\colon \hat{X}\to X$ given by $\rho(x,\mu) = x$ defines a continuous anchor map.
  Let $G\ast \hat{X} \coloneq \menge{(g, (x,\mu))\in G\times \hat{X}}{ s(g) = \rho(x,\mu) = x }$ and define $g.(x,\mu) \coloneq (r(g), g.\mu)$ for all $(g, (x,\mu))\in G\ast \hat{X}$.
  It is straightforward to verify that $\rho(x,\mu).(x,\mu) = (x,\mu)$, for all $(x,\mu)\in \hat{X}$,
  and that when $g_1$ and $g_2$ are composable in $G$ and $(g_2,(x,\mu))\in G\ast \hat{X}$,
  then $(g_1,g_2.(x,\mu))\in G\ast \hat{X}$ and $(g_1g_2).(x,\mu) = g_1.(g_2.(x,\mu))$.

  The transformation groupoid of the action of $G$ on $\hat{X}$ is of the form
  \[
    G\ltimes \hat{X} = \menge{ (g, (x,\nu))\in G\times \hat{X} }{ s(g) = x }
  \]
  with unit space $\{(x,(x,\mu))\in X\times \hat{X}\}$ that we identify with $\hat{X}$.
  The range and source maps are $r(g,(s(g),\nu)) = (r(g), g.\nu)$ and $s(g,(s(g),\nu)) = (s(g),\nu)$,
  and $(g_1,(s(g_1),\nu))$ and $(g_2,(s(g_2), \mu))$ are composable precisely if $g_1$ and $g_2$ are composable and $g_2.\mu=\nu$ in which case
  $(g_1,(s(g_1),g_2\mu))(g_2,(s(g_2), \mu)) = (g_1g_2, (s(g_2),\mu))$. 
  It is straighforward to verify that the map $\hat{G} \to G\ltimes \hat{X}$ given by $(g,\mu) \mapsto (g, (s(g),g^{-1}.\mu))$ for all $(g,\mu)\in \hat{G}$
  defines an isomorphism of topological groupoids.
  The inverse is given by $(g, (s(g),\nu)) \mapsto (g, g.\nu)$, for all $(g, (s(g),\nu))\in G\ltimes \hat{X}$.
\end{proof}

\begin{remark}
    The minimal essential cover of $G$ (when $G$ is post-critically sparse) is also an {\'e}tale groupoid since $\tilde{X} = \overline{\iota(X\setminus P_{\Lambda})}$ satisfies $s(r^{-1}(\tilde{X}))\subseteq \tilde{X}$. Moreover, $\tilde{G}$ is isomorphic to $G\ltimes \tilde{X}$.
\end{remark}
\setlength{\parindent}{0cm} \setlength{\parskip}{0.5cm}

We derive some structural similarities between the groupoid $G$ and its minimal cover.
Recall that a topological groupoid $G$ is \emph{effective} if the unit space is the interior of the isotropy of the groupoid. By \cite[Theorem~15.1]{Ke}, every second countable semi-\'etale groupoid $G$ has a Borel Haar system given by the counting measure $\lambda^{x}$ on $G^{x}$, $x\in X$. In this case, it follows from \cite[Proposition~2.4]{Ren2} that $G$ is \textit{Borel amenable} if and only if there is a \textit{Borel approximate invariant density}, which is a sequence of Borel functions $(g_{n})$ satisfying the conditions in \cite[Definition~2.3]{Ren2}. If $G$ is additionally \'etale, then $(g_{n})$ can be chosen to be continuous with additional convergence properties (see \cite[Theorem~2.14]{Ren2}). With these properties, $(g_{n})$ is called a \textit{topological approximate invariant density} \cite[Definition~2.7]{Ren2}.
\begin{corollary} 
  Let $G$ be a locally compact Hausdorff and semi-\'etale groupoid, and let $\hat{G}$ be the \'etale cover groupoid.
\setlength{\parindent}{0cm} \setlength{\parskip}{0cm}

  \begin{enumerate}[label=(\arabic*)]
 \item  If $G$ is second countable, then $G$ is Borel amenable if and only if $\hat{G}$ is Borel amenable.\label{i:amenable}
    \item If $\hat{G}$ is topologically transitive (or minimal), then $G$ is topologically transitive (or minimal). \label{i:transitive}
    \item $G$ is principal if and only if $\hat{G}$ is principal. \label{i:principal}
    \item If $\hat{G}$ is effective, then $G$ is effective.  \label{i:effective}
  \end{enumerate}
\end{corollary}
\setlength{\parindent}{0cm} \setlength{\parskip}{0cm}

\begin{proof}
    \labelcref{i:amenable}: The proof is similar to the proof of \cite[Proposition~2.4]{ES17}. 
    Note that $\pi \colon \hat{G}^{x}\to G^{\pi(x)}$ is a bijection for all $x\in \hat{X}$.
    If $G$ is Borel amenable, then there exists a Borel approximate invariant density $(h_{n})$ on $G$, and it follows that $(h_{n}\circ\pi)$ is a Borel approximate invariant density on $\hat{G}$. Conversely, since $G$ is second countable, so is $\hat{G}$, so that if $\hat{G}$ is Borel amenable, then there is a topological approximate invariant density $(g_{n})\subseteq C_{c}(\hat{G})$. Since $\iota \colon G^{x}\to \hat{G}^{\iota(x)}$ is a bijection for all $x\in X$, it follows that $(g_{n}\circ\iota)$ is a Borel approximate invariant density on $G$, because our transfer operator preserves bounded Borel functions with compact support.
\setlength{\parindent}{0cm} \setlength{\parskip}{0.5cm}
    
   \labelcref{i:transitive}: Suppose $\hat{G}$ is topologically transitive and choose $(x,\mu)\in \hat{X}$ such that the orbit $G.(x,\mu)$ is dense in $\hat{X}$.
  Then $G.x$ is dense in $X$, so $G$ is topologically transitive.
  If $\hat{G}$ is minimal, then this argument works for every unit in $\hat{X}$, so it follows that $G$ is minimal.
  
  \labelcref{i:principal}: Suppose first that $G$ is principal and take an isotropy element $(g, (x,\mu))\in \hat{G}$.
  Then $g$ is an isotropy element in $G$, so that $g = x$, and therefore $(g, (x,\mu))$ is a unit in $\hat{G}$.
  Conversely, suppose $\hat{G}$ is principal and take an isotropy element $g\in G$.
  Then $(g, (s(g),\lambda_{s(g)}))$ is an isotropy element in $\hat{G}$, hence the element is a unit, so $g = s(g)$.

  \labelcref{i:effective}: 
  Suppose $\hat{G}$ is effective, and take an element $g$ in $G$ that is in the interior of the isotropy. This means that there is an open set $U$ of $G$ which is pure isotropy and contains $g$.
  The canonical map $\pi\colon \hat{G} \to G$ given by $(g, (s(g),\mu)) \mapsto g$ is continuous, 
  and the preimage $\pi^{-1}(U)$ is open and pure isotropy and it contains $(g, (s(g),\lambda_{s(g)}))$.
  As $\hat{G}$ is effective, we see that $(g, (s(g),\lambda_{s(g)}))$ is a unit, so it follows that $g$ is a unit.
  Therefore, $G$ is effective.
\end{proof}
\setlength{\parindent}{0cm} \setlength{\parskip}{0.5cm}

\bremark
We have a canonical isomorphism of topological groupoids $G_{T_{\min}} \cong \hat{G}_T$, i.e., given a locally injective continuous transformation $T$ with discrete fibres, the {\'e}tale groupoid cover of its Deaconu--Renault groupoid coincides with the Deaconu--Renault groupoid of the minimal cover $T_{\min}$ as constructed in Section~\ref{s:SingleTrafo}.
\eremark

\begin{example}
  The even shift of Example \ref{exm:even} defines a semi-\'etale groupoid that is effective, but whose \'etale cover groupoid is not effective; the image of the unique critical point $0^{\infty}$ under $\iota$ is an isolated point in $\hat{X}$.
  The Sturmian shifts define semi-\'etale groupoids that are minimal but whose \'etale cover groupoids are not minimal \cite{Brix2023}.
\end{example}

In the following, we compare the reduced groupoid $C^*$-algebra of $\hat{G}$ with Thomsen's construction in \cite{Tho} of a convolution $C^*$-algebra attached to locally compact Hausdorff groupoids with locally injective range and source maps. First let us recall Thomsen's construction. Let $\cB$ be the $^*$-algebra generated by $C_c(G)$ in the algebra of bounded, compactly supported functions $G \to \Cz$ with respect to the usual involution $f^*(\gamma) = \overline{f(\gamma^{-1})}$ and convolution $(f * g)(\gamma) = \sum_{\alpha \beta = \gamma} f(\alpha) g(\beta)$. (Note that $f * g$ is no longer continuous in general.) Thomsen constructs the $C^*$-algebra $B$ as the closure of the image of $\cB$ under the left regular representation coming from $G$ (see \cite{Tho} for details). 

Let $\tilde{\cF} \coloneq {}^*\text{-}\alg(C_c(G), C_c(G) r^*\Lambda(C_c(G))) \subseteq C^*_r(\hat{G})$ and
$\tilde{\cE} \coloneq {}^*\text{-}\alg(C_c(X), \Lambda(C_c(G))) \subseteq C^*_r(\hat{G})$.
The proposition below now provides a conceptual explanation for the results in \cite[Section 2.2 and 2.3]{Tho}.

\bprop
\label{prop:B=C*G}
We have $B \subseteq C^*_r(\hat{G})$, and the following are true: 
\setlength{\parindent}{0cm} \setlength{\parskip}{0cm}

\begin{enumerate}
\item[(i)] $\Lambda(C_c(G)) \subseteq \cB$.
\item[(ii)] $\tilde{\cE} \subseteq \cB$.
\item[(iii)] $f r^*(g) \in \cB$ for all $f \in C_c(G)$ and $g \in {}^*\text{-}\alg(\Lambda(C_c(G)) \subseteq \tilde{\cE}$.
\item[(iv)] $\tilde{\cF} = \lspan(\menge{f r^*(g)}{f \in C_c(G), \, g \in {}^*\text{-}\alg(\Lambda(C_c(G))})$.
\end{enumerate}
Consequently, $\cB \supseteq \tilde{\cF}$, so $B = C^*_r(\hat{G})$.
\eprop
\setlength{\parindent}{0cm} \setlength{\parskip}{0cm}

\bproof
The inclusion $B \subseteq C^*_r(\hat{G})$ holds by construction because $X$ is dense in the unit space of $\hat{G}$, so that for the reduced C*-norm, it suffices to induce representations from points in $X$.
\setlength{\parindent}{0cm} \setlength{\parskip}{0.5cm}

Let us prove (i): Given $f \in C_c(U)$, where $U$ is such that $r_{U} = r \colon U\to r(U)$ and $s_{U} = s \colon U\to s(U)$ are bijective, choose $V \subseteq U$ open such that $\supp(f) \subseteq V \subseteq \overline{V} \subseteq U$. Take $g \in C_c(U^{-1})$ such that $g = 1$ on $V^{-1}$. Then, given $\gamma \in G$, either $(f * g)(\gamma) = 0 = \Lambda(f)(\gamma)$, or we have
\begin{align*}
 (f * g)(\gamma) &= f(r_U^{-1}(r(\gamma))) g(s_{U^{-1}}^{-1}(s(\gamma)))\\
 &= f(r_U^{-1}(\gamma)) g(s_{U^{-1}}^{-1}(\gamma))
 = f(r_V^{-1}(\gamma)) g(s_{V^{-1}}^{-1}(\gamma)) = f(r_V^{-1}(\gamma)) = \Lambda(f)(\gamma).
\end{align*}
For the second equality, observe that the value is zero unless $\gamma \in X$. For the third equality, note that we only get a non-zero value if $\gamma \in r(V)$. For the fourth equality, we used that $g = 1$ on $V^{-1}$. The argument is now complete because $f * g \in \cB$.

(ii) holds because $X$ is open in $G$ and on $C_c(X)$, convolution coincides with taking pointwise products.

(iii) is true because $(f r^*(g))(\gamma) = f(\gamma) g(r(\gamma)) = (f * g)(\gamma)$ for all $\gamma \in G$, $f \in C_c(G)$ and $g \in {}^*\text{-}\alg(\Lambda(C_c(G)))$.

(iv) holds because given $f, f' \in C_c(G)$ and $g, g' \in {}^*\text{-}\alg(\Lambda(C_c(G)))$, we have
\[
 f r^*(g) f' r^*(g') = (f f') r^*(g g').
\]
This implies that $\lspan(\menge{f r^*(g)}{f \in C_c(G), \, g \in {}^*\text{-}\alg(\Lambda(C_c(G)))})$ is closed under pointwise multiplication, which in turn implies the claim.

We conclude that $\cB \supseteq \tilde{\cF}$, and hence $B \supseteq C^*_r(\hat{G})$ because $\tilde{\cF}$ is dense in $C^*_r(\hat{G})$ in the C*-norm of $C^*_r(\hat{G})$.
\eproof
\setlength{\parindent}{0cm} \setlength{\parskip}{0.5cm}

\bremark
In the special case of Deaconu--Renault type groupoids arising from single transformations, Proposition~\ref{prop:B=C*G} is due to Thomsen in the sense that Thomsen already constructed {\'e}tale groupoid models for convolution C*-algebras arising from locally injective maps (see \cite{Tho2011}).
\eremark

\end{document}